\documentclass[a4paper]{amsart}
\pdfoutput=1

\usepackage[T1]{fontenc}
\usepackage{amssymb,mathrsfs,amsthm}
\usepackage{xy}

\usepackage{amsfonts}
\usepackage{marginnote}
\usepackage{amsmath}
\usepackage{empheq}
\usepackage{booktabs}
\usepackage{comment}
\usepackage{dsfont}
\usepackage{caption}
\usepackage{graphicx}
\usepackage{hhline}
\usepackage{longtable}
\usepackage{afterpage}
\usepackage{lscape}
\usepackage{tikz}
\usepackage{tikz-cd}
\usepackage{placeins}
\usepackage{pdfpages}
\usepackage{hyperref}
\usepackage[
backend=biber,
style=alphabetic,
sorting=nyt, maxnames=9,minnames=4,maxbibnames=99,maxalphanames=99
]{biblatex}
\usetikzlibrary{decorations.pathreplacing}
\usetikzlibrary{positioning}

\newsavebox{\pullback}
\sbox\pullback{%
\begin{tikzpicture}%
\draw (0,0) -- (1ex,0ex);%
\draw (1ex,0ex) -- (1ex,1ex);%
\end{tikzpicture}}

\usepackage{subcaption}
\usepackage{floatrow}
\usepackage{mathtools}

\newcommand{\Adams}{\mathrm{Ad}}

\newcommand{\bR}{{\mathbb{R}}}
\newcommand{\bC}{{\mathbb{C}}}

\newcommand{\bE}{{\mathbb{E}}}
\newcommand{\bF}{{\mathbb{F}}}

\newcommand{\bQ}{{\mathbb{Q}}}
\newcommand{\bS}{{\mathbb{S}}}
\newcommand{\bZ}{{\mathbb{Z}}}
\newcommand{\cC}{{\mathcal{C}}}
\newcommand{\cD}{{\mathcal{D}}}

\newcommand{\defo}{\mathrm{def}}

\newcommand{\cCf}{{\mathcal{C}^{\textrm{fil}}}}

\newcommand{\cA}{\mathcal{A}}
\newcommand{\cP}{\mathcal{P}}
\newcommand{\cE}{\mathcal{E}}

\newcommand{\Hom}{\mathrm{Hom}}
\newcommand{\Ext}{\mathrm{Ext}}
\newcommand{\Mod}{\mathrm{Mod}}
\newcommand{\Tot}{\mathrm{Tot}}
\newcommand{\Sp}{{\mathrm{Sp}}}

\newcommand{\Syn}{{\mathrm{Syn}}}
\newcommand{\Stable}{\mathrm{Stable}}
\newcommand{\Comod}{\mathrm{Comod}}

\newcommand{\E}{\mathrm{E}}

\newcommand{\one}{\mathbf{1}}
\newcommand{\fil}{\text{fil}}
\newcommand{\Fil}{\text{fil}}

\renewcommand{\Re}{\operatorname{Re}}

\newcommand{\Fun}{\mathrm{Fun}}
\newcommand{\id}{\mathrm{id}}

\newcommand{\BP}{\mathrm{BP}}
\newcommand{\MU}{\mathrm{MU}}

\newcommand{\fp}{\mathrm{fp}}
\newcommand{\Gr}{\mathrm{Gr}}
\newcommand{\Map}{\mathrm{Map}}

\newcommand{\h}{{\mathrm{H}}}
\newcommand{\fin}{\mathrm{fin}}
\newcommand{\cell}{\mathrm{cell}}

\newcommand{\SH}{\mathrm{SH}}

\newcommand{\hfp}{{\mathbb{F}_p}}
\newcommand{\starstar}{\star}
\newcommand{\starstarstar}{*,*,*}
\newcommand{\nubp}{\BP^{\hfp}}
\newcommand{\nubpbp}{\BP_{\starstar}\BP^{\hfp}}
\newcommand{\nubpbptwo}{\BP_{\starstar}\BP^{\bF_2}}
\newcommand{\degree}[1]{\vert #1\vert}

\newcommand{\customheader}[1]{
  \expandafter\def\csname @oddhead\endcsname{#1}
  \expandafter\def\csname @evenhead\endcsname{#1}
}

\makeatletter
\newcommand{\defaultheader}{
  \def\@oddhead{\normalfont\scriptsize\hfil\leftmark\hfil\thepage}
  
  \def\@evenhead{\normalfont\scriptsize\thepage\hfil\rightmark\hfil}
}
\makeatother

\makeatletter
\newcommand{\colim}{\operatorname{colim}}
\makeatother

\theoremstyle{definition}
\newtheorem{theorem}{Theorem}[section]
\newtheorem*{theorem*}{Theorem}%
\newtheorem{proposition}[theorem]{Proposition}
\newtheorem{lemma}[theorem]{Lemma}
\newtheorem{corollary}[theorem]{Corollary}
\newtheorem{definition}[theorem]{Definition}
\newtheorem{remark}[theorem]{Remark}

\newtheorem{example}[theorem]{Example}

\newtheorem{notation}[theorem]{Notation}
\newtheorem{assumption}[theorem]{Assumption}

\definecolor{chartgray}{gray}{0.5}
\definecolor{darkcyan}{rgb}{0, 0.7, 0.7}
\definecolor{truecyan}{rgb}{0, 1, 1}
\definecolor{darkgreen}{rgb}{0, 0.65, 0}
\definecolor{truemagenta}{rgb}{1, 0, 1}
\definecolor{amber}{rgb}{1.0, 0.75, 0.0}
\definecolor{fuchsia}{rgb}{0.75, 0.0, 1.0}
\definecolor{darkcyan}{rgb}{0, 0.7, 0.7}
\definecolor{electricpurple}{rgb}{0.75, 0.0, 1.0}

\title[The Synthetic Adams-Novikov Spectral Sequence]{Stable Comodule Deformations and the Synthetic Adams-Novikov Spectral Sequence}
\author{Jake Francis Baer, Maxwell Johnson, and Peter Marek}
\date{\today}
\begin{document}

\begin{abstract}
We study the Adams-Novikov spectral sequence in $\bF_p$-synthetic spectra, computing the synthetic analogs of $\BP$ and its cooperations to identify the synthetic Adams-Novikov $\E_2$-page, computed in a range with a synthetic algebraic Novikov spectral sequence. We then identify deformations associated to the Cartan-Eilenberg and algebraic Novikov spectral sequences in terms of stable comodule categories, categorifying an algebraic Novikov spectral sequence result of Gheorghe-Wang-Xu. We then apply Isaksen-Wang-Xu methods in $\bF_2$-synthetic spectra to deduce differentials in the $p=2$ synthetic Adams-Novikov for the sphere, producing almost entirely algebraic computations through the 45-stem.   
\end{abstract}
\maketitle
\setcounter{tocdepth}{1}
\tableofcontents

\section{Introduction}

\numberwithin{theorem}{subsection}

A prevailing technique in recent work on stable homotopy theory is the categorification of Adams spectral sequences into stable $\infty$-categories which record both the homotopy of classical spectra as well as the data of the spectral sequence pages and differentials. This technique was pioneered following the identification of the $p$-complete cellular motivic category over $\bC$ with a deformation associated to the Adams-Novikov spectral sequence based on complex cobordism $\mathrm{MU}$ \cite{GWX21}. These techniques have resulted in numerous theoretical and computational advancements, including the computation of stable homotopy groups of spheres (\cite{Isa19}, \cite{IWX20}, \cite{Bur20},\cite{BX23}), obstruction theories for multiplicative structures (\cite{Bur22}), and resolving classical questions about manifolds \cite{BHS23}. More recently, questions about identifying naturally occurring deformations have also led to the construction of novel spectral sequences \cite{AKBBK23}.

\subsection{Deformations by Spectral Sequences} The main setting of our paper is Pstr\k{a}gowski's theory of $E$-synthetic spectra, which serves as a deformation of classical stable homotopy theory encoding information about the $E$-Adams spectral sequence for an Adams-type ring spectrum $E$. However, historically the deformation theoretic approach to spectral sequences began first with the realization that $p$-complete, cellular stable motivic homotopy theory over $\bC$ encoded information about the Adams-Novikov spectral sequence for $\BP$ \cite{HKO11}, \cite{Ghe18},\cite{Isa19}, that this similarity upgraded to a categorical equivalence \cite{GWX21}, and later that its cellular objects could be constructed with entirely topological input \cite{GIKR18}. For the reader familiar with motivic methods, all instances of $\BP$-synthetic spectra in this article could be replaced with $\bC$-motivic spectra. We choose the synthetic formalism because we work primarily with $\hfp$-synthetic spectra which have no motivic counterpart.

\bigskip

The motivic version of the story may be summarized as follows: there is a self-map $\tau$ of the $p$-completed motivic sphere spectrum $\bS_{\bC}$ which acts on all objects in the category. The objects for which this map is an equivalence form a subcategory equivalent to the category of $p$-complete spectra, which is referred to as the generic fiber, taking $\tau=u$ for any unit $u\in \mathbb{A}^1$. Alternatively, one could set ``$\tau=0$'' by considering the category of modules over the ring of its cofiber $\bS_{\bC}/\tau$ (sometimes denoted $C\tau$), whose projection from $\bS_{\bC}$ is a map of $\bE_\infty$-rings \cite{Ghe18}. This category of modules was shown to be equivalent to a derived category of comodules over the Hopf algebroid $\BP_*\BP$ \cite{GWX21}, and its homotopy groups encode the $\E_2$-page of the Adams-Novikov spectral sequence \cite{Isa19}, \cite{Ghe18}. Pstr\k{a}gowski's theory of synthetic spectra \cite{Pst22} then extended the idea of deforming the category of spectra with respect to the Adams spectral sequence of an arbitrary Adams-type ring spectrum $E$. It is this final theory in which we will primarily work, as well as analogs of the construction in \cite{GIKR18}. We review the salient features of both in Section~\ref{syn}.

\subsection{The aNSS and the CESS}

A key difference between the identifications of the special fiber of the $\bC$-motivic deformation in \cite{GWX21} and \cite{Pst22} is that the former proves that the algebraic Novikov spectral sequence (aNSS), a spectral sequence for computing the $\E_2$-page of the Adams-Novikov spectral sequence (ANSS), appears as the $\hfp^{\mathrm{mot}}$-Adams spectral sequence for $\bS_{\bC}/\tau$. This extra result is a key input for the computations in \cite{IWX20}, which we review in Section \ref{iwxstuff}. One of the motivations for writing this article was the conjecture, which we prove, that a similar result can be established for $\hfp$-synthetic spectra and the Cartan-Eilenberg spectral sequence (CESS). We briefly review the aNSS and CESS at the beginning of Section \ref{deformationsection}. Our result follows from the identification \cite{Mar83}, \cite{Pal01}, \cite{Bel20} of the CESS with the $\cP_*$-Adams spectral sequence in $\Stable(\cA_*)$ and the following general result:

\begin{theorem}[\ref{ClambdaSS}]
 Let $R$ be an $\bE_1$-ring spectrum in $\Sp$ and $X$ be any spectrum. The $\nu_E(R)$-Adams spectral sequence for $\nu_EX/\tau$ in $\mathrm{Syn}_{E}$ is isomorphic to the $E_*R$-Adams spectral sequence for $E_*X$ in $\mathrm{Stable}_{E_*E}$.   
\end{theorem}

\begin{corollary}[\ref{cessexample}]
The $\nubp:=\nu_{\bF_p}\BP$-Adams spectral sequence for $\nu_{\bF_p}X/\lambda\in\Syn_{\bF_p}$ is isomorphic to the Cartan-Eilenberg spectral sequence computing $$\Ext_{\cA_*}^{*,*}(\bF_p,\h_*(X))$$ associated to the extension of Hopf algebras $\cP_*\to \cA_*\to \cE_*$. At the level of $\E_2$-pages, the isomorphism looks as follows:
\begin{equation*}
    \begin{tikzcd}
       \Ext_{\cP_*}^{s,t}(\bF_p,\Ext_{\cE_*}^{u,*}(\bF_p,\h_*(X))) \ar[d,Rightarrow,"\mathrm{CESS}"'] & \Ext_{\nubpbp}^{s,t-u,t}(\nubp_{\star},\nubp_{\star}(\nu_{\bF_p} X/\lambda)) \ar[l,"\cong"'] \ar[d,Rightarrow,"\nubp-\mathrm{Adams \textbf{ } SS}"] \\
       \Ext_{\cA_*}^{s+u,t}(\bF_p,\h_*(X)) & \pi_{t-s-u,t}(\nu_{\bF_p} X/\lambda) \ar[l,"\cong"] 
    \end{tikzcd}
\end{equation*}
\end{corollary}

The choice of grading here is made with respect to the traditional grading of the CESS: $s$ corresponds to the $\cP_*$-filtration degree, $t$ corresponds to internal degree, and $u$ corresponds to the $\cE_*$-filtration degree. We also choose to use the parameter $\lambda$ in $\Syn_{\bF_p}$ rather than $\tau$. See the discussion in Section~\ref{HFpsubsection}.

\bigskip

This provides the complementary result to \cite[Thm. 8.3]{GWX21}, with the roles of $\BP$ and $\hfp$ reversed. Our methods differ from theirs, however, as we furnish a direct comparison between cobar-type resolutions using results of \cite{Bel20} which describe Cartan-Eilenberg spectral sequences as certain Adams spectral sequences in stable comodules.

\subsection{Synthetic ANSS}

To emulate the IWX-method in $\Syn_{\bF_p}$, we need the analog of the motivic Adams spectral sequence in $\Syn_{\bF_p}$. For us, this is the \textit{synthetic Adams-Novikov spectral sequence} (synthetic ANSS). Classically, the construction of the ANSS first starts with the computation of the Hopf algebroid $(\BP_*,\BP_*\BP)$, originally due to \cite{BP66}. Let $\nubp_{\star}:=\pi_{\star}(\nu_{\hfp}\BP)$ and $\nubpbp:=\pi_{\star}(\nu_{\hfp}\BP\otimes \nu_{\hfp}\BP)$. We first compute the $\bF_p$-synthetic version $(\nubp_{\star},\nubpbp)$ of this Hopf algebroid:

\begin{theorem}[\ref{BPhomotopytheorem}, \ref{BPhopfalgebroidprop}]
As $\bZ_{(p)}[\lambda]$-algebras,
\begin{equation*}
    \begin{split}
    &\nubp_{\starstar}\cong \bZ_{(p)}[\lambda,h,v_1,v_2,\ldots]/(\lambda h=p), \\
    &\nubpbp\cong \nubp_{\starstar}[t_1,t_2,\ldots]. 
    \end{split}
\end{equation*}
where $v_i\in\pi_{2p^i-2,2p^i-1}$, $t_i\in\pi_{2p^i-2,2p^i-2}$, $\lambda\in\pi_{0,-1}$, $h\in\pi_{0,1}$, and $p\in\pi_{0,0}\cong\bZ_{(p)}$. Moreover, the Hopf algebroid structure maps for $\nubpbp$ are completely determined by the $\lambda$-localization functor $\lambda^{-1}:\Syn_{\hfp}\to\Sp$.
\end{theorem}

In particular, the formulas for the structure maps of the Hopf algebroid \linebreak $(\nubp_\starstar,\nubpbp)$ are the same as the classical ones up to multiples of $\lambda$ and $h$. For more, see Theorem~\ref{BPformulathm} and Remark~\ref{BPformulaexampleremark}. The algebra structure of the $\bF_p$-synthetic analog of $\BP_*$ is the classical one tensored with $\bZ_{(p)}[\lambda]$, except for the relation $\lambda h=p$. This relation expresses the fact that modulo $\lambda$, $h$ detects $p$ in $\bF_p$-Adams filtration 1. 

\bigskip

With this computation at hand, we can identify the $\E_2$-page (see Prop. \ref{synANSSe2page}) of the synthetic ANSS as $\Ext$ groups isomorphic to $$\Ext_{\nubpbp}^{\starstarstar}(\nubp_{\starstar},\nubp_{\starstar}X).$$ The synthetic ANSS converges to $\pi_{\star}(X_{\nubp}^{\wedge})$. The nilpotent completion $X_{\nubp}^{\wedge}$ turns out to be a $p$-localization (see Theorem~\ref{synANSSconvergencethm}).

\bigskip

The $\E_2$-page of the classical ANSS is difficult to compute and is often studied via the aNSS. Likewise, for the synthetic ANSS, its $\E_2$-page is accessible via the \textit{synthetic algebraic Novikov spectral sequence} (synthetic aNSS). We prove that for the $\bF_p$-synthetic sphere spectrum $\bS_{\bF_p}$, the synthetic aNSS is a $\lambda$-Bockstein spectral sequence:

\begin{theorem}[\ref{lambdabockstein}]
\label{lambdabocksteinintro}
The synthetic algebraic Novikov spectral sequence for $X=\bS_{\bF_p}$ is isomorphic to a $\lambda$-Bockstein spectral sequence. In particular, there is a 1-1 correspondence between classical aNSS differentials $d_r^{\mathrm{cl}}(x)=y$ and synthetic aNSS differentials $d_r^{\mathrm{syn}}(x)=\lambda^{r-1}y$.    
\end{theorem}

\begin{remark}
    Some authors have used the term \textit{rigid} (e.g. \cite{BX23}) to describe a spectral sequence whose classes and differentials are formally determined by another spectral sequence by adjoining a new parameter $\tau$ (or $\lambda$) and replacing classes killed by differentials with $\tau$-torsion. This is equivalent to our usage of the term Bockstein.
\end{remark}

This allows us to compute the $\E_2$-page of the synthetic ANSS for the sphere using the classical aNSS for the sphere. We discuss more details of the computation in Section~\ref{ANSSsphere}.

\subsection{Deformations of the aNSS and the CESS}

We also prove new deformation results in the stable homotopy theory of comodules over Hopf algebroids. Let $(\bF_p[\tau],\cA_{\starstar}^\BP)$ denote the $\BP$-synthetic dual Steenrod algebra \cite[Sec. 6.2]{Pst22} at a prime $p$. Let $\Stable(\Gamma)^{\cell}$ denote the smallest full subcategory of Hovey's \cite{Hov04} stable comodule category $\Stable(\Gamma)$ for a Hopf algebroid $(A,\Gamma)$ containing all spheres and closed under colimits. For what follows, let $\cP_*:=\h_*\BP$, considered as a commutative algebra object of either $\Stable(\cA_*)$ or $\Stable(\BP_*\BP)$ and for a symmetric monoidal category $\cC$ with $p$-completion, let $\cC_{ip}:=\Mod(\cC;\one_p^{\wedge})$.

\bigskip

Using the machinery developed in \cite[App. C]{BHS20} and a twisted $t$-structure on $\Stable(\Gamma)$ (see Section~\ref{tstructuresection}), we ``reconstruct'' the categories of synthetic stable comodules by identifying them with Adams spectral sequence-based deformations (denoted $\Adams_R(\cC)$, see  Definition~\ref{adamsdef}) analogous to those first studied in \cite{GIKR18}:

\begin{theorem}[\ref{cedef}]
\label{stabledefintro}
    The $\infty$-category of stable comodules over $\cA_{\starstar}^\BP$ is a deformation of $\Stable(\cA_*)$ with respect to the Cartan-Eilenberg spectral sequence for the extension $\cP_*\to \cA_*\to \cE_*$
    and is reconstructed by the equivalence
    \[
    \Stable(\cA_{\starstar}^\BP)^{\cell} \simeq \Adams_{\cP_*}(\Stable(\cA_*))
    \]
    of $\Stable(\cA_*)$-linear symmetric monoidal stable $\infty$-categories. This deformation has a parameter $\tau$ whose generic fiber recovers $\Stable(\cA_*)$.
\end{theorem}

\begin{theorem}[\ref{andef}]
\label{BPdeformationtheoremintro}
    The $\infty$-category of stable comodules over $\nubpbp$ is a deformation of $\Stable(\BP_*\BP)$ with respect to the algebraic Novikov spectral sequence and is reconstructed by the equivalence
    \[
  \Stable(\nubpbp)^{\cell}_{ip} \simeq \Adams_{\cP_*}(\Stable(\BP_*\BP))
    \]
    of $\Stable(\BP_*\BP)$-linear symmetric monoidal stable $\infty$-categories. This deformation has a parameter $\lambda$ whose generic fiber recovers $\Stable(\BP_*\BP)_{ip}$.
\end{theorem}

The reconstruction aspects of the above results are entirely analogous to the reconstructions of $(\SH(\bC)^{\cell})^{\wedge}_{2}$ and $\SH(\bR)^{\mathrm{AT}}_{ip}$ in \cite{GIKR18} and \cite{BHS20} respectively. 

For both categories, we can also provide a new description of the special fibers\footnote{The $\text{E}^{\text{CESS}}_2$ and $\text{E}_1^{\text{aNSS}}$ notations are slightly abusive. The precise statements are given in Theorems \ref{cedef} and \ref{andef}.}
    \begin{align*}
    \Mod(\Stable(\cA_{\starstar}^\BP)^{\cell}; \bF_p[\tau]/\tau)&\simeq \Mod(\Stable(\cA_*)^\Gr; \text{E}^{\text{CESS}}_2)\\
    \Mod(\Stable(\nubpbp)_{ip}^{\cell}; (\nubp_\starstar)_p^{\wedge}/\lambda)&\simeq \Mod(\Stable(\BP_*\BP)^\Gr; \text{E}_2^{\text{aNSS}})
    \end{align*}
    using \cite[Prop. C.2]{BHS20}. The algebraic objects over which we are taking modules are known to be isomorphic, although the underlying categories in which these modules live are not. It seems likely nonetheless that both of these special fibers are closely related as specializations of a larger theory described briefly in Remark \ref{bisyn}.

\bigskip

In view of Theorem~\ref{stabledefintro}, it seems likely that a result similar to \cite[Thm. A.8]{BHS23} would recover the Burklund-Xu result \cite[Thm. 2.8]{BX23} that the motivic Cartan-Eilenberg spectral sequence is a $\tau$-Bockstein spectral sequence, though we do not seek to establish this in the present article.

\bigskip

In the same way that Pstr\k{a}gowski's category $\Syn_E$ is a deformation of $\Sp$ with respect to the $E$-Adams spectral sequence, Theorems~\ref{stabledefintro} and~\ref{BPdeformationtheoremintro} say that $\Stable(\cA_{\starstar}^\BP)^{\cell}$ is a deformation of $\Stable(\cA_*)$ with respect to the CESS and, up to a $p$-completion, $\Stable(\nubpbp)^{\cell}$ is a deformation of $\Stable(\BP_*\BP)$ with respect to the aNSS. This means that synthetic tools can be used to study the categories $\Stable(\cA_*)$ and $\Stable(\BP_*\BP)$, just as they have been used to study $\Sp$.

\bigskip

As a consequence of Lemma~\ref{qlemma}, combined with the results Theorem~\ref{stabledefintro}, Theorem~\ref{BPdeformationtheoremintro}, Corollary~\ref{cessexample}, \cite[Thm. 1.17]{GWX21}, and \cite{Pst22}, the cofiber sequences
\begin{align*}              &\bF_p[\tau]\xrightarrow{\tau}\bF_p[\tau]\xrightarrow{i}\bF_p[\tau]/\tau\xrightarrow{q_1}\Sigma\bF_p[\tau] & \in\Stable(\cA_{\starstar}^\BP)^{\cell} \\
    &\nubp_\starstar\xrightarrow{\lambda}\nubp_\starstar\xrightarrow{i}\nubp_\starstar/\lambda\xrightarrow{q_2}\Sigma\nubp_{\starstar} & \in\Stable(\nubpbp)^\cell_{ip}
\end{align*}

where $\nubp_\starstar$ is implicitly $p$-complete, set up correspondences
\begin{equation*}
\begin{tikzcd}[cramped]
\left\{ 
 \begin{tabular}{c}
     classes in aNSS $\E_2$-page \\[5pt]
      which detect image of \\[5pt]
     $(q_1)_\starstar:\E_2(\bS_{\BP}/\tau)\to\E_2(\Sigma\bS_{\BP})$ \\[5pt]
     for $\bF_p^{\BP}$-Adams SS in $\Syn_{\BP}$      
 \end{tabular}
\right\} &
\longleftrightarrow &
\left\{ 
\begin{tabular}{c}
      differentials in the  \\[5pt]
    $\BP^{\bF_p}$-Adams SS \\[5pt]
    for $\bS_{\bF_p}/\lambda\in\Syn_{\bF_p}$
 \end{tabular}
\right\}  
\end{tikzcd}
\end{equation*}

\begin{equation*}
\begin{tikzcd}[cramped]
\left\{ 
 \begin{tabular}{c}
     classes in CESS $\E_2$-page \\[5pt]
     which detect image of \\[5pt]
     $(q_2)_\starstar:\E_2(\bS_{\bF_p}/\lambda)\to\E_2(\Sigma\bS_{\bF_p})$ \\[5pt]
     for $\BP^{\bF_p}$-Adams SS in $\Syn_{\bF_p}$      
 \end{tabular}
\right\} &
\longleftrightarrow &
\left\{ 
\begin{tabular}{c}
      differentials in the  \\[5pt]
    $\bF_p^{\BP}$-Adams SS \\[5pt]
    for $\bS_{\BP}/\tau\in\Syn_{\BP}$
 \end{tabular}
\right\}  
\end{tikzcd}
\end{equation*}

In the context of $\Stable(\cA^{\mathrm{mot}}_\starstar)$ for the $\bC$-motivic dual Steenrod algebra, $(q_1)_\starstar$ and algebraic Novikov differentials were determined by \cite{aNSS} up to topological degree 110 at $p=2$. We use this to compute Cartan-Eilenberg differentials and to determine classes which detect the image of $(q_2)_{\starstar}$ in Section~\ref{smodlambda}.

\subsection{The IWX-Comparison Method}\label{iwxstuff}

As was shown in \cite{Ghe18}, the natural map $\bS_{\bC}\to \bS_{\bC}/\tau$ is one of $\bE_\infty$-rings. The homotopy groups of the former encode the stable homotopy groups of spheres and the latter's homotopy groups are isomorphic to the Adams-Novikov $\E_2$-page. Moreover, one can run the $\hfp^{\mathrm{mot}}$-Adams spectral sequence for both, and get a comparison map from the motivic Adams spectral sequence to the algebraic Novikov spectral sequence. Because inverting $\tau$ in the motivic Adams spectral sequence recovers the classical Adams spectral sequence \cite{DI10}, this innocuously constructed map allows one to transfer differentials back and forth between the two spectral sequences. This idea was first used, to great effect, in \cite{IWX20} to drastically extend our understanding of the stable homotopy groups of spheres at the prime $2$. We show in Section~\ref{ANSSsphere} that our results allow for similar comparisons to be made between the Adams-Novikov spectral sequence and the Cartan-Eilenberg spectral sequence.

\subsection{Computational Results}

We conclude by using the IWX-method in $\Syn_{\bF_2}$ to algebraically compute the 2-primary synthetic Adams-Novikov spectral sequence for the sphere out to the 45-stem. We take as our input data the Isaksen-Wang-Xu computation of the algebraic Novikov spectral sequence \cite{aNSS}. Using the results of Section~\ref{BPsynanalogsection}, we can compute the $\E_2$-page of the synthetic ANSS with the synthetic algebraic NSS. It remains to compute differentials in the synthetic Adams-Novikov spectral sequence. Using the results from Section~\ref{Clambdasection}, we know that the synthetic ANSS for $\bS_{\bF_2}/\lambda$ is isomorphic to the Cartan-Eilenberg spectral sequence for a particular extension of the Steenrod algebra. In Section~\ref{smodlambda}, we apply the results from Section~\ref{defcomod} to compute the synthetic ANSS for $\bS_{\bF_2}/\lambda$ out to the 45-stem. Finally in Section~\ref{ANSSsphere}, we compute the synthetic ANSS for $\bS_{\bF_2}$ by comparison to the synthetic ANSS for $\bS_{\bF_2}/\lambda$.\footnote{To see interactive synthetic Adams-Novikov charts for $\bS_{\bF_2}$ and $\bS_{\bF_2}/\lambda$, we refer the reader to \cite{BJM25}.}

\subsection*{Organization}

In Section~\ref{recollections}, we recall the salient features of Pstr\k{a}gowski's category of synthetic spectra, record some folklore results about the category $\Syn_{\hfp}$ specifically, and review the filtered-object approach to synthetic spectra with an eye toward the work of \cite{BHS20} on recognizing deformations. Next in Section~\ref{BPsynanalogsection}, we compute the structure of the $\hfp$-synthetic analogs of $\BP$ and its Hopf algebroid of cooperations, leading to an identification of the $\E_2$-page of the $\nu_{\hfp}\BP$-Adams spectral sequence. We also construct an $\bF_p$-synthetic version of the algebraic Novikov spectral sequence. Section~\ref{Clambdasection} describes spectral sequences for $\bS_{E}/\tau$ in terms of Adams spectral sequences in stable comodules, allowing us to identify several examples including the classical Cartan-Eilenberg spectral sequence for the extension
\[
\cP_*\to \cA_* \to \cE_*.
\]
In Section~\ref{tstructuresection}, we produce twisted $t$-structures generally on stable $\infty$-categories $\cC$ with respect to a compact object $K\in\cC$ and an automorphism $F:\cC\to\cC$. These are applied in particular to $\cC=\Stable(\Gamma)^{\cell}$ and $\cC=\Syn_{E}^{\cell}$. In Section~\ref{defcomod}, we identify the categories of stable comodules over $\cA_{\starstar}^\BP$ and $\nubpbp$ as deformations of their classical analogues by the Cartan-Eilenberg and algebraic Novikov spectral sequence respectively. In Section~\ref{smodlambda}, we compute the synthetic Adams-Novikov spectral sequence for $\bS_{\bF_2}/\lambda$ through the 45-stem. In Section~\ref{ANSSsphere}, we apply all previous results to the computation of the synthetic Adams-Novikov spectral sequence for the sphere through the 45-stem. In Appendix~\ref{nilcompappendix}, we recall the results of \cite{Man21} and use them, together with the results of Section~\ref{tstructuresection}, to prove completion results in $\Stable(\Gamma)^{\cell}$ and $\Syn_{\bF_p}$ for several Hopf algebroids $(A,\Gamma)$ of interest.

\subsection{Notations and Conventions}

We will work freely with the higher categorical language of \cite{HTT} and \cite{HA}. To avoid clutter, we fix the following notations and conventions throughout the paper:
\begin{enumerate}
    \item When considering graded objects as a whole, we use $\ast$ to denote a single grading and $\starstar$ to denote a bigrading.

    \item We abuse notation to identify an abelian group $A$ with its image under the Eilenberg Mac-Lane functor, commonly denoted $\h A$. In particular, we will write $\hfp$ for the spectrum $\h\bF_p$, with the exception of denoting $\hfp$-homology $\h_*X$.

    \item We sometimes refer to the synthetic analog $\nu_EX$ of a spectrum $X$ as $X^E$. In particular, we do this for $\bF_p^{\BP}:=\nu_{\BP}\bF_p$ and $\BP^{\bF_p}:=\nu_{\bF_p}\BP$.

    \item We will write $\nubpbp$ for the Hopf algebroid of $\nubp$-cooperations and similarly $\cA_{\starstar}^\BP$ for those of $\hfp^{\BP}$.
\end{enumerate}

\subsection{Acknowledgements}

The authors would like to thank Gabriel Angelini-Knoll, William Balderrama, Eva Belmont, Robert Burklund, Myungsin Cho, Dan Isaksen, Mike Mandell, Piotr Pstr\k{a}gowski, Doug Ravenel, Noah Riggenbach, and Zhouli Xu for helpful conversations regarding the material in this paper.

\section{Recollections on Synthetic Spectra and Filtered Objects}
\label{recollections}

\subsection{Synthetic Spectra}
\label{syn}

Given an Adams-type ring spectrum $E$ \cite[Def. 3.14]{Pst22}, one can categorify its Adams spectral sequence as Pstr\k{a}gowski's category of synthetic spectra $\mathrm{Syn}_{E}$, which is a presentably symmetric monoidal stable $\infty$-category. The synthetic construction comes equipped with a functor $\nu_E:\Sp\to \Syn_E$ assigning to each spectrum $X$ its synthetic analog $\nu_E X$ (abbreviated $\nu X$ where there is no ambiguity) whose bigraded homotopy groups record both the original homotopy groups $\pi_*(X)$ and the $E$-Adams spectral sequence for $X$, in a sense made precise by \cite[Appendix A]{BHS23}.  We recall some salient features of $\nu_E$ below, which appear in \cite{Pst22} as Lemmas 4.4 and 4.23 and Corollary 4.28.

\begin{proposition}[\cite{Pst22}]
\label{nufacts}
    The functor $\nu_E:\Sp\to \Syn_E$
    \begin{itemize}
        \item[(a)] is fully faithful,
        \item[(b)] preserves filtered colimits,
        \item[(c)] preserves cofiber sequences $X\to Y\to Z$ which induce $E_*(-)$ short exact sequences,
        \item[(d)] and is lax symmetric monoidal.
    \end{itemize}
        Moreover, for any $X\simeq \colim X_\alpha$ which is a filtered colimit of finite spectra with finitely generated projective $E_*$ homology and any $Y\in \Sp$, the natural map $\nu_E X \otimes \nu_E Y\to \nu_E (X\otimes Y)$ is an equivalence.
\end{proposition}

\begin{notation}
    Such an $X_\alpha$ is said to be $E$-finite projective \cite[Def. 3.13]{Pst22}; the full subcategory of such is denoted $\Sp^\fp_E$.
\end{notation}

\begin{notation}
 There is a naturally bigraded family of spheres in $\Syn_E$ which we denote by $\bS_{E}^{t,w}=\Sigma^{t-w}\nu \bS^{w}$. They induce bigraded suspension functors $\Sigma^{t,w}(-)=(-)\otimes \bS^{t,w}_E$ and their mapping groups are denoted $\pi_{t,w}X=[\bS^{t,w},X]$. We often just write $\bS_E$ instead of $\bS_E^{0,0}$. We will often restrict attention to the subcategory $\Syn_E^\cell\subset \Syn_E$ which is the smallest subcategory of $\Syn_E$ containing the bigraded spheres $\bS^{t,w}$ and closed under colimits.    
\end{notation}

\begin{remark}
    It has recently been shown in \cite{Law24} that when $E$ is connective $\Syn_E\simeq \Syn_E^\cell$.
\end{remark}

Because the functor $\nu$ does not commute with the formal suspension in the stable $\infty$-categories $\Sp$ and $\Syn_E$, there is an induced universal comparison map $\tau:\Sigma^{0,-1}\bS_E\to \bS_E$ which controls much of the structure of the synthetic category. As $\tau$ is a self-map of the monoidal unit, it acts all on $E$-synthetic spectra $\tau:\Sigma^{0,-1}X\to X$.

\begin{remark}
    We use the notation $\tau:\Sigma^{0,-1}\bS_E\to\bS_E$ in this section when working with a general Adams-type ring spectrum $E$. In Section~\ref{HFpsubsection} and throughout the rest of this paper, we use $\lambda$ instead if $E=\hfp$.
\end{remark}

\begin{theorem}[\cite{Pst22}]
    There is a symmetric monoidal and fully faithful functor $Y:\Sp \to \Syn_E$ such that the map $\tau:\Sigma^{0,-1}Y(X)\to Y(X)$ is an equivalence for all $X$. Moreover, all such $\tau$-invertible synthetic spectra are in the essential image of $Y$. As a result, $Y$ provides an embedding of spectra inside of $\Syn_E$ consisting of the $\tau$-invertible objects.
\end{theorem}

The functor $X\mapsto \tau^{-1}X$ is known as the realization functor and is a symmetric monoidal left adjoint to $Y$. The effect of inverting $\tau$ is often referred to as the \textit{generic fiber} of such a deformation of spectra. We now turn to the special fiber obtained informally by setting $\tau=0$. Let $\bS_{E}/\tau$ denote the cofiber of $\tau:\Sigma^{0,-1}\bS_E\to \bS_E$. To state the result, we first recall that to a (graded) Hopf algebroid $(A,\Gamma)$ there is an associated category of stable comodules $\Stable(\Gamma)$ constructed by inverting the homotopy equivalences of chain complexes detected by projectives \cite{Hov04}. The following is an amalgamation of results in Sections 4.5 and 4.6 of \cite{Pst22}:

\begin{theorem}
\label{stablethm}
There is an adjunction
\[
\chi_*:\Mod(\Syn_E; \bS_E/\tau)\leftrightarrows \Stable(E_*E):\chi^*
\]
 whose left adjoint is fully faithful. On synthetic spectra in the image of $\nu$, the composition $\chi_L(-\otimes \bS_E/\tau)$ corresponds to taking $E_*$ homology and hence takes the bigraded shifts of the cofiber of $\tau$ to those of the unit of $\Stable(E_*E)$, inducing an equivalence on cellular objects. In addition, there is an isomorphism of abelian groups
    \[
    [\Sigma^{t,w}\nu X,\nu Y/\tau]\cong \Ext^{w-t,w}_{E_*E}(E_*X, E_*Y).
    \]
\end{theorem}

We see that the cofiber of $\tau$ records the $\E_2$ page of the $E$-Adams spectral sequence. In fact, the $\tau$-Bockstein spectral sequence encodes the entirety of the $E$-Adams, see \cite{BHS23}[Thm. 9.19].

\subsection{Synthetic Spectra over $\hfp$}
\label{HFpsubsection}
When $E=\hfp$, we can make significantly stronger statements than for general Adams-type homology theories. Because it is conventional to associate the map $\tau$ with motivic (and therefore $\mathrm{MU}$- or $\BP$-synthetic) homotopy theory, we will write $\lambda$ for the analogous map in $\Syn_{\hfp}$. We recall some results that are special to synthetic spectra over $\hfp$ below.

\begin{proposition}
\label{Fpomnibus}
    The following stronger statements hold for the category of $\hfp$-synthetic spectra:
    \begin{enumerate}
        \item[(a)] the functor $\nu_{\hfp}$ is symmetric monoidal,
        \item[(b)] and the inclusion $\Mod(\Syn_{\hfp}; \bS_{\bF_p}/\lambda)\hookrightarrow \Stable(\cA_*)$ is an equivalence.
    \end{enumerate}
\end{proposition}

\begin{proof}
 The first observation is that $\Sp_{\hfp}^\fp\simeq \Sp^\fin$ where the right-hand side is the category of finite spectra: the projectivity requirement is trivial for $\bF_p$-vector spaces and finiteness of a spectrum implies the finite generation of its mod $p$ homology. Because every spectrum is a filtered colimit of finite spectra, (a) follows from Proposition \ref{nufacts}. Statement (b) then follows from the fact that the inclusion is essentially surjective onto the cellular subcategory and $\Stable(\cA_*)$ is cellular \cite{HPS97}[Thm. 2.3.1].
\end{proof}

\subsection{Filtered Models}
\label{filtered}

The results in this section are collected from \cite{GIKR18}, \cite{Hed20}, and \cite{BHS20}. Any results not therein are certainly known to experts; we make no claims of originality. Let $\bZ_{\geq}$ denote the category of integers with unique morphisms $m\to n$ whenever $m\geq n$. Given a category $\cC$, a filtered object in $\cC$ is the data of a functor 
\[X_*:\bZ_\geq\to \cC,\] 
and these compile into a category $\cC^\Fil = \Fun(\bZ_\geq, \cC)$. We may consider a filtered object to be a diagram in $\cC$, and taking colimits yields a functor $\Re:\cC^\Fil \to \cC$ known as realization. Assuming that $\cC$ has a zero object, $\Re$ has a collection of sections $Y_k$, for $-\infty < k \leq \infty$, which are determined by:
\[
Y_kX_w = \begin{cases}
X & w\leq k\\
0 & \text{otherwise}
\end{cases}
\]
where the connecting maps are identity where possible and $0$ otherwise. This is still perfectly well defined for $k=-\infty$, but is the $0$ functor and, of course, no longer a section of $\Re$. 

\begin{remark}
    Where we feel notational confusion may be possible, we will write $\Re^\fil$ to distinguish the filtered-realization functor from that of other deformations.
\end{remark}

\begin{lemma}[\cite{glasman16}]
If $\cC$ is symmetric monoidal and its tensor product preserves colimits in each variable, then via Day convolution the category $\cC^\Fil$ is again symmetric monoidal, with unit $Y_0\one$. The tensor product on $\cC^\fil$ will then also preserve colimits in each variable.
\end{lemma}

The category of filtered objects also comes equipped with a distinguished autoequivalence $[1]:\cC^\Fil \to \cC^\Fil$ determined by the shift $X[1]_w=X_{w+1}$ distinct from the formal suspension $\Sigma$, which is computed levelwise in $\cC^\Fil$. We will write $\Sigma^{t,w}:=\Sigma^t\circ [w]$. 

\subsection{Associated Graded and $\tau$}

Given a filtered object $X\in \cC^\Fil$, we may take levelwise cofibers of the internal connecting maps to get an object of $\cC^\Gr$. This assignment is functorial and symmetric monoidal. We will denote this functor by $\Gr_*:\cC^\Fil\to \cC^\Gr=\Fun(\bZ,\cC)$ where $\bZ$ is the discrete category of natural numbers. The goal of this section is to understand how we can construct the functor $\Gr_*$ by tensoring with a certain object $\one/\tau$ in $\cC^\Fil$.

\bigskip

Given any filtered object $X\in \cC^\Fil$, the internal maps of $X$ compile to give an important self-map $\tau_X:\Sigma^{0,-1}X\to X$ depicted as:

\[\begin{tikzcd}
	{...} & {...} \\
	{X_{w+2}} & {X_{w+1}} \\
	{X_{w+1}} & {X_{w}} \\
	{X_{w}} & {X_{w-1}} \\
	{...} & {...}
	\arrow["{f_{w+2}}", from=2-1, to=2-2]
	\arrow["{f_{w+2}}"', dashed, from=2-1, to=3-1]
	\arrow["{f_{w+1}}", dashed, from=2-2, to=3-2]
	\arrow["{f_{w+1}}"', dashed, from=3-1, to=4-1]
	\arrow["{f_w}", dashed, from=3-2, to=4-2]
	\arrow["{f_{w+1}}", from=3-1, to=3-2]
	\arrow["{f_{w}}"', from=4-1, to=4-2]
	\arrow[dashed, from=1-2, to=2-2]
	\arrow[dashed, from=1-1, to=2-1]
	\arrow[dashed, from=4-2, to=5-2]
	\arrow[dashed, from=4-1, to=5-1]
\end{tikzcd}\]

 where the maps internal to the filtered spectrum are dashed and the components of $\tau_X$ are solid. The cofiber of $\tau_X$, which we will denote $X/\tau$, has trivial connecting maps as a filtered object, and little is lost by forgetting this information and considering $X/\tau$ as a graded object in $\cC$. We will denote $\tau_\one$ simply by $\tau$.

\begin{lemma}
There is an equivalence $\tau_X\simeq \id_X\otimes \tau$ so that $X/\tau\simeq X\otimes \one/\tau$. As graded objects, $\Gr_*X\simeq X/\tau$.
\end{lemma}

\begin{proof}
    This was known already to \cite{GIKR18} and amounts to the observation that at the graded level, both constructions are defined to be the cofiber of the same map.
\end{proof}

It is also valuable to understand the effect of inverting $\tau$ on an object. 
\begin{lemma}
    There is an equivalence $\tau^{-1}X\simeq Y_{\infty}(\Re(X))$. 
\end{lemma}

\begin{proof}
    This was implicit in \cite{GIKR18} and follows from noting that since colimits in functor categories are evaluated levelwise, $\tau^{-1}$ acts by taking the same colimit as $\Re$ but in every filtered degree.
\end{proof}

Because the functor $Y_\infty$ is a fully faithful embedding of $\cC$ into $\cC^\Fil$, it is common to identify the two functors $\Re$ and $\tau^{-1}$. 

\subsection{Deformations via Filtered Objects}
\label{deformationsection}

From here on we will assume that $\cC$ is presentably symmetric monoidal and that $\cC^\fil$ has been endowed with its Day convolution symmetric monoidal structure.

\begin{definition}[\cite{BHS20}]
    A tower functor is a lax monoidal functor $T:\cC\to \cC^\Fil$.
\end{definition}

\begin{example}
\label{adamstowerexample}
    For every $\bE_1$-ring $R$ in $\cC$ we can construct a lax monoidal Adams tower $T_R:\cC\to \cC^\Fil$. Explicitly this can be constructed by transporting the cobar cosimplicial object $X\otimes R^{\bullet+1}$ to filtered objects via the $\infty$-categorical Dold-Kan correspondence \cite[Thm. 1.2.4.1]{HA}.
\end{example}

 We now specialize to the case that $\cC$ has a $t$-structure which will remain implicit. In what follows, denote by $\Gamma_R(-):\cC\to\cCf$ the functor 
 \[
 X\mapsto \Tot(\tau_{\geq *}(X\otimes R^{\bullet + 1}))
 \]
We will use the notation $X^\wedge_R:=\Re(\Gamma_R(X))$. We often refer to $\Gamma_R$ as the decalage of the tower in Example~\ref{adamstowerexample}. This construction was first considered \cite{GIKR18}.

\begin{remark}
    The usage of the notation $X^\wedge_R$ is slightly abusive. When $\cC=\Sp$ with its usual $t$-structure, we see that $\Re(\Gamma_RX)$ will only necessarily agree with the $R$-nilpotent completion when $X$ is bounded below.
\end{remark}

\begin{definition}
\label{adamsdef}
    Given a stably symmetric monoidal $\cC$ with $t$-structure and a $\bE_n$-ring $R$ in $\cC$, the $R$-Adams deformation, denoted $\Adams_R(\cC)$, is defined to be the stable $\infty$-category of modules $\Mod(\cC^\Fil; \Gamma_R(\one))$.
\end{definition}

\begin{proposition}
    The category $\Adams_R(\cC)$ satisfies the following additional properties:
    \begin{enumerate}
        \item The category $\Adams_R(\cC)$ is $\bE_{n-1}$-monoidal,
        \item The category of $\tau$-invertible objects $\Mod(\cC_R; \tau^{-1}\one)$ is equivalent to \linebreak $\Mod(\cC;\one^\wedge_R)$.
        \item The $t$-structure homotopy group spectral sequence of $\Gamma_RX/\tau$ encodes the Adams spectral sequence for $\pi^t_*X$ starting at $\E_2$.
    \end{enumerate}
\end{proposition}

\begin{proof}
    The observations (1) and (2) were first observed in \cite{GIKR18} and were proven in general as \cite[Prop. C.2]{BHS20}. The claim (3) follows from computations analogous to \cite[Prop. 3.3]{GIKR18} and the general form appears as \cite[Cons. C.9]{BHS20}.
\end{proof}

\begin{remark}
    The notation $\Adams_{R}(\cC)$ comes from the fact that these deformations encode generalized Adams Spectral Sequences. The even variant of the category $\Adams_{\MU}(\Sp)$ was first studied in \cite{GIKR18}, where it was shown to recover the cellular motivic category over $\bC$ after $2$-completion. More generally, for $E$ an Adams-Type ring spectrum, the category $\Adams_E(\Sp)$ is nearly equivalent to the category of cellular $E$-synthetic spectra \cite[C.22]{BHS20}. Explicitly there is an equivalence $\Adams_E(\Sp)\simeq \Mod(\Syn_E^\cell; (\bS_E)^\wedge_\tau)$.         
\end{remark}


\begin{example}\label{CESSdef}
    Let $\cC=\Stable(\cA_*)$ and let $R=\cP_*$ as an object in the heart. Then if we let $X=\one=\hfp$ be the unit in $\Stable(\cA_*)$ we may consider the cofiber sequence
    \[
    \Gamma_{\cP_*}\hfp \xrightarrow{\tau} \Gamma_{\cP_*}\hfp \to \Gamma_{\cP_*}\hfp/\tau\to \Sigma\Gamma_{\cP_*}\hfp
    \]
    After applying trigraded homotopy groups, the term $\pi_{-*,*,*}\Gamma_{\cP_*} \hfp/\tau$ may be identified with the $\E_2$-page of the Cartan-Eilenberg spectral sequence as a consequence of Theorem \ref{cedef}.
\end{example}

\begin{example}\label{aNSSdef}
    Let $\cC=\Stable(\BP_*\BP)$ and let $R=\cP_*$ as an object in the heart. Then if we let $X=\one=\BP_*$ be the unit in $\Stable(\BP_*\BP)$ we may consider the cofiber sequence
    \[
    \Gamma_{\cP_*}\BP_* \xrightarrow{\lambda} \Gamma_{\cP_*}\BP_* \to \Gamma_{\cP_*}\BP_*/\lambda\to \Sigma\Gamma_{\cP_*}\BP_*
    \]
    After applying trigraded homotopy groups, the term $\pi_{-*,*,*}\Gamma_{\cP_*} \BP_*/\lambda$ may be identified with the $\E_2$-page of the algebraic Novikov spectral sequence as a consequence of Theorem \ref{andef}.
\end{example}

\subsection{Recognizing Deformations}
\label{recognitionsection}
Our primary reason for using filtered deformation machinery is the following recognition theorem which is the key lemma reducing our results in Section~\ref{deformationsection} to algebraic computations.

\begin{definition}[\cite{BHS20}]
A 1-parameter deformation pair is a pair of presentably symmetric monoidal stable $\infty$-categories $(\cC,\cC_{\defo})$ together with
\begin{itemize}
    \item[(a)] a symmetric monoidal left adjoint $\Re:\cC_{\defo}\to \cC$ known as the realization,
    \item[(b)] a symmetric monoidal left adjoint $c:\cC\to \cC_{\defo}$ providing $\cC_{\defo}$ with a $\cC$-enrichment,
    \item[(c)] a homomorphism $i:\bZ\to \pi_0\mathrm{Pic}(\cC_{\defo})$,
    \item[(d)] and a set of compact dualizable $\{K_{\alpha}\}_{\alpha \in I}$ for $\cC$.
\end{itemize}
This structure needs to satisfy the following:
\begin{itemize}
    \item[(i)] The functor $c$ is a section of $\Re$,
    \item[(ii)] The objects $i(n)$ realize to $\one_{\cC}$,
    \item[(iii)] The objects $\{K_\alpha\otimes i(n)\}$ generate $\cC_{\defo}$ as $\alpha,n$ range,
    \item[(iv)] The functor $\Re$ induces an equivalence
    \[
    \Map(i(n),i(m))\xrightarrow{\simeq} \Map(\one_{\cC},\one_{\cC})
    \]
    whenever $n\leq m$.
\end{itemize}
\end{definition}

Our interest in deformation pairs comes from the following recognition theorem which allows us to reconstruct the deformation using filtered objects.

\begin{theorem}[\cite{BHS20}]
\label{recognitiontheorem}
    For any deformation pair $(\cC,\cC_{\defo})$ there is a functor $i_*:\cC_{\defo}\to \cC^\Fil$ inducing a $\cC$-linear symmetric monoidal equivalence
    \[
    \Mod(\cC^\Fil;i_*\one_{\cC_{\defo}})\simeq \cC_{\defo}
    \]
\end{theorem}

Which appears as Proposition C.20 in loc. cit.

\subsection{Spectral Sequences and $\tau$-Bocksteins}
In this section, we recall the construction of the spectral sequence associated with a filtered object. We do this in order to describe the comparison with the $\tau$-Bockstein. Given a filtered spectrum $X\in \Sp^\fil$ we easily obtain an exact couple by applying the functor $\pi_{*}$ to the extended diagram
\[\begin{tikzcd}
	{...} & {X_2} & {X_1} & {X_0} & {X_{-1}} & {X_{-2}} & {...} \\
	& {\Gr_2X} & {\Gr_1X} & {\Gr_0X} & {\Gr_{-1}X} & {\Gr_{-2}X} 
	\arrow[from=1-4, to=1-5]
	\arrow[from=1-5, to=1-6]
	\arrow[from=1-6, to=1-7]
	\arrow[from=1-1, to=1-2]
	\arrow[from=1-2, to=1-3]
	\arrow[from=1-3, to=1-4]
 \arrow[from=1-2, to=2-2]
	\arrow[from=1-3, to=2-3]
	\arrow[from=1-4, to=2-4]
	\arrow[from=1-5, to=2-5]
 \arrow[from=1-6, to=2-6]
 \arrow[dashed, from=2-6, to=1-5]
	\arrow[dashed, from=2-5, to=1-4]
	\arrow[dashed, from=2-4, to=1-3]
	\arrow[dashed, from=2-3, to=1-2]
 \arrow[dashed, from=2-2, to=1-1]
\end{tikzcd}\]
and putting $E^{*,*}=\pi_*\Gr_*X=\pi_{*,*}(X/\tau)$ and $D^{*,*}=\pi_*X_*=\pi_{*,*}X$. This, however, is exactly the exact couple used to form the $\tau$-Bockstein by applying $\pi_{*,*}$ to the cofiber sequence 
\[\begin{tikzcd}
	{\Sigma^{0,-1}X} && X \\
	& {X/\tau} &
	\arrow[from=1-3, to=2-2]
	\arrow["\tau", from=1-1, to=1-3]
	\arrow[dashed, from=2-2, to=1-1]
\end{tikzcd}\]
As a result, unlike the case of synthetic spectra where the comparison to the $\tau$-Bockstein is rather involved, the comparison to the $\tau$-Bockstein is essentially tautological from the filtered perspective. In particular, as in \cite[Thm. 9.19]{BHS23}, there are cofiber sequences
\begin{align*}
    X/\tau\xrightarrow{\tau^r}&X/\tau^{r+1}\to X/\tau^r\xrightarrow{\beta}\Sigma X/\tau \\
    X&\xrightarrow{\tau} X\xrightarrow{i} X/\tau\xrightarrow{q} \Sigma X
\end{align*}
which satisfy $\beta_*(x)=-y$ and $i_*(\overline{z})=z$, where $d_r(x)=\tau^ry$ is a $\tau$-Bockstein differential and $\overline{z}\in\pi_{*,*}(X)$ is a non-$\tau$-divisible element detected by $z\in\pi_{*,*}(X/\tau)$. These same constructions generalize to any setting with an appropriate notion of homology or (multi-graded) homotopy groups. As a result, we have the following lemma, which is a generalization of \cite[Prop. 2.14]{IKLYZ23}:

\begin{lemma}
\label{qlemma}
    Let $\cC$ be a stable $\infty$-category equipped with an exact functor $\pi_{\star}:\cC\to \cA$ valued in an abelian category. Then if $(\E_r,d_r)$ is the spectral sequence associated to the a filtered object $X\in \cC^\fil$ and $x\in\pi_{\starstar}(X/\tau)$ survives to the $\E_r$ page, then $x$ supports a $\tau$-Bockstein differential $d_r(x)=\tau^r y$ if and only if in the sequence
    \[
    X\xrightarrow{\tau} X\xrightarrow{i} X/\tau\xrightarrow{q} \Sigma X
    \]
    we have that $q_\star(x)\in\pi_\starstar(\Sigma X)$ is detected by $-\tau^{r-1}y$.
\end{lemma}
\begin{proof}
    Similar to \cite[Prop. 2.14]{IKLYZ23}, both directions of the proof will use the following commutative diagram of horizontal cofiber sequences:
\[\begin{tikzcd}
	{X/\tau} & {X/\tau^{r+1}} & {X/\tau^r} & {\Sigma X/\tau} \\
	X & X & {X/\tau^r} & {\Sigma X} \\
	X & X & {X/\tau} & {\Sigma X}
	\arrow["q", from=3-3, to=3-4]
	\arrow["\tau", from=3-1, to=3-2]
	\arrow["i", from=3-2, to=3-3]
	\arrow["f", from=2-3, to=2-4]
	\arrow["{\tau^r}", from=2-1, to=2-2]
	\arrow[from=2-2, to=2-3]
	\arrow["{\tau^{r-1}}"', from=2-1, to=3-1]
	\arrow["\mathrm{Id}"', from=2-2, to=3-2]
	\arrow[from=2-3, to=3-3]
	\arrow["{\tau^{r-1}}", from=2-4, to=3-4]
	\arrow["\beta", from=1-3, to=1-4]
	\arrow["i", from=2-1, to=1-1]
	\arrow[from=2-2, to=1-2]
	\arrow["\mathrm{Id}", from=2-3, to=1-3]
	\arrow["i"', from=2-4, to=1-4]
	\arrow["{\tau^r}", from=1-1, to=1-2]
	\arrow[from=1-2, to=1-3]
\end{tikzcd}\]
For the forwards direction, since $x$ survives to the $\E_r$-page, $x$ lifts to $\pi_\starstar (X/\tau^r$). Note that $\beta_\starstar(x)=-y$ so that $q_\starstar(x)=\tau^{r-1}f(x)$ is detected by $$\tau^{r-1}i(f(x))=\tau^{r-1}\beta(x)=-\tau^{r-1}y.$$ For the backwards direction, note that $f(x)$ is detected by $-y$ so that $\beta(x)=i(f(x))=-y$ and, hence, $d_r(x)=\tau^ry$.
\end{proof}

\section{The Synthetic Adams-Novikov Spectral Sequence}
\label{BPsynanalogsection}

For an $\bF_p$-synthetic spectrum $X$, we will let $X_{\starstar}:=\pi_{\starstar}(X)$ and $X_{\starstar}X:=\pi_{\starstar}(X\otimes X)$. The goal in this section is to compute the bigraded Hopf algebroid $(\nubp_{\starstar},\nubpbp)$ and to give tools for computing the $\E_2$-page of the synthetic Adams-Novikov spectral sequence in terms of this Hopf algebroid. Throughout this section, we fix a prime $p$.

\subsection{Homotopy of the synthetic analogs of $\BP$ and $\BP\otimes \BP$}

We first compute $\nubp_{\starstar}$ and $\nubpbp$ as $\bZ_{(p)}[\lambda]$-algebras. Our strategy is to use the fracture square

\begin{equation}
\label{fracture_diagram}
    \begin{tikzcd}
        X \ar[r] \ar[d] & X_{\lambda}^{\wedge} \ar[d] \\
        X[\lambda^{-1}] \ar[r] & (X_{\lambda}^{\wedge})[\lambda^{-1}] \\
    \end{tikzcd}
\end{equation}

for an arbitrary $X\in\Syn_{\hfp}$. Though the existence of this fracture square has likely been known by experts in synthetic spectra, we reproduce a proof of its existence generally in $\Syn_E$ by imitating \cite[Section 6]{DFHH14}:

\begin{lemma}
\label{fracture}
    The commutative square (1) for $X\in\Syn_E$ is a pullback square.
\end{lemma}

\begin{proof}
    We first show that $X\simeq L_{\bS_E}X\simeq L_{(\bS_E[\lambda^{-1}]\vee \bS_{E}/\lambda)}$, where $L$ denotes Bousfield localization. The first equivalence is obvious. For the second, we show that the class of $\bS_E$-acyclics agrees with the class of $(\bS_E[\lambda^{-1}]\vee \bS_{E}/\lambda)$-acyclics. Note that $X$ is $\bS_E$-acyclic if and only if $X\simeq 0$. So immediately, $\bS_E$-acyclic implies $(\bS_E[\lambda^{-1}]\vee \bS_{E}/\lambda)$-acyclic. For the other direction, $X$ being $(\bS_E[\lambda^{-1}]\vee \bS_{E}/\lambda)$-acyclic is equivalent to both 
    
    $$
    \begin{aligned}
     X[\lambda^{-1}]&\simeq 0,\\
      X/\lambda &\simeq 0.  
    \end{aligned}
    $$
    
    If $X/\lambda\simeq 0$, then $\lambda:X\to\Sigma^{0,1}X$ is an equivalence so that $X[\lambda^{-1}]\simeq X$. Hence, $X\simeq 0$.

    \bigskip

    Since $(X/\lambda)[\lambda^{-1}]\simeq 0$ is true for any synthetic spectrum $X$, the fracture square follows from a synthetic version of \cite[Section 6, Prop. 2.2]{DFHH14}.
\end{proof}

The most mysterious part of this square is the $\lambda$-completion $X_{\lambda}^{\wedge}$. For synthetic analogs of ordinary spectra, the $\lambda$-completion is straightforward:

\begin{lemma}
\label{lambdacomp}
If $Y$ is a spectrum, there is an equivalence $\nu_E(Y)_{\lambda}^{\wedge}\xrightarrow{\simeq}\nu_E( Y_E^{\wedge})$. 
\end{lemma}

\begin{proof}
Note that $\lambda$-completion is the same as $\bS_{E}/\lambda$-Bousfield localization and $\nu_E(Y_E^{\wedge})$ is $\lambda$-complete by \cite[Prop. A.13]{BHS23}. Hence there exists a natural map $\nu_E(Y)_{\lambda}^{\wedge}\to\nu_E(Y_E^{\wedge})$ fitting in the commutative diagram
\begin{equation*}
    \begin{tikzcd}
        \nu_E(Y) \ar[r] \ar[dr] & \nu_E(Y)_{\lambda}^{\wedge} \ar[d,dashed] \\
              & \nu_E(Y_E^{\wedge}) 
    \end{tikzcd}
\end{equation*}

with the diagonal map being induced by the $E$-nilpotent completion $Y\to Y_E^{\wedge}$. It suffices to check this is an equivalence after applying $(-)/\lambda$. This is true because the top map becomes an equivalence by definition and because the map $Y\to Y_E^{\wedge}$ induces an isomorphism on Adams $\E_2$-pages.   
\end{proof}

\begin{corollary}
If $Y$ is a bounded-below spectrum, there is an equivalence $\nu_{\hfp}(Y)_{\lambda}^{\wedge}\xrightarrow{\simeq}\nu_{\hfp}( Y_p^{\wedge})$.    
\end{corollary}

From now on, we let $\nu Y$ denote $\nu_{\hfp}Y$. If we want to understand the bigraded homotopy of $\nu Y$ for a bounded-below spectrum $Y$, Lemmas~\ref{fracture}, \ref{lambdacomp}  allow us to understand $\pi_{\starstar}\nu Y$ in terms of
\begin{equation*}
    \begin{split}
        \pi_{\starstar}\nu Y[\lambda^{-1}]&\cong \pi_*Y[\lambda^{\pm 1}], \\
        \pi_{\starstar}(\nu Y_{\lambda}^{\wedge})&\cong \pi_{\starstar}(\nu (Y_p^{\wedge})), \\
        \pi_{\starstar}(\nu Y_{\lambda}^{\wedge}[\lambda^{-1}])&\cong \pi_{*}(Y_p^{\wedge})[\lambda^{\pm 1}],
    \end{split}
\end{equation*}

and the maps between them. The first and third isomorphisms follows from the equivalence $\Syn_{\hfp}[\lambda^{-1}]\simeq\Sp$, with $\pi_k$ living in bidegree $(k,k)$. The homotopy of $\nu (Y_p^{\wedge})$ can be computed using knowledge of the $\hfp$-Adams spectral sequence of $Y$ and \cite[Thm. 9.19]{BHS23}. We do this first for $Y=\BP$:

\begin{lemma}
\label{nuBPlemma}
As a $\bZ_p^{\wedge}[\lambda]$-algebra,
\begin{equation*}
    \pi_{\starstar}(\nu(\BP)_{\lambda}^{\wedge})\cong\pi_{\starstar}(\nu(\BP_p^{\wedge}))\cong \bZ_p^{\wedge}[\lambda,h,v_1,v_2,\ldots]/(\lambda h=p)
\end{equation*}
where $v_i\in\pi_{2p^i-2,2p^i-1}$ and $p\in\pi_{0,0}\cong\bZ_p^{\wedge}$.
\end{lemma}

\begin{proof}
Note first that the $\hfp$-Adams spectral sequence for $\BP$ collapses at the $\E_2$-page. Using the $\nu\hfp$-Adams spectral sequence and \cite[Thm. A.8]{BHS23}, this says that
\begin{equation*}
    {}_{\nu\hfp}\E_2^{\starstarstar}\cong {}_{\nu\hfp}\E_{\infty}^{\starstarstar}\cong \bF_p[\lambda,h,v_1,v_2,\ldots].
\end{equation*}

Passing on to $\pi_{\starstar}(-)$, by \cite[Thm. 9.19(4)]{BHS23} we get a surjective $\bZ_p^{\wedge}[\lambda]$-algebra map $$\bZ_p^{\wedge}[\lambda,h,v_1,v_2,\ldots]\to\pi_{\starstar}(\nu(\BP_p^{\wedge})).$$ The relation $\lambda h=p$ must hold because of the isomorphism of short exact sequences
\begin{equation*}
    \begin{tikzcd}
        0 \ar[r] & \pi_{0,1} \ar[r,"\cdot\lambda"] \ar[d,"\cong"] & \pi_{0,0} \ar[r] \ar[d,"\cong"] & \Ext^{0,0}_{\cA_*}(\bF_p,\h_*(\BP_p^{\wedge})) \ar[r] \ar[d,"\cong"] & 0 \\
        0 \ar[r] & \bZ_p^{\wedge} \ar[r,"\cdot p"] & \bZ_p^{\wedge} \ar[r] & \bF_p \ar[r] & 0
    \end{tikzcd}
\end{equation*}
which induces a surjective map $\bZ_p^{\wedge}[\lambda,h,v_1,v_2,\ldots]/(\lambda h=p)\to\pi_{\starstar}(\nu(\BP_p^{\wedge}))$. By a dimension count of both sides, this map must be an isomorphism.
\end{proof}

\begin{lemma}
As a $\bZ_p^{\wedge}[\lambda]$-algebra,
\begin{equation*}
    \pi_{\starstar}((\nubp\otimes\nubp)_{\lambda}^{\wedge})\cong\pi_{\starstar}(\nu((\BP\otimes\BP)_p^{\wedge}))\cong\pi_{\starstar}((\nubp)_{\lambda}^{\wedge})[t_1,t_2,\ldots]
\end{equation*}
where $t_i\in\pi_{2p^i-2,2p^i-2}$.  
\end{lemma}

\begin{proof}
The proof is very similar to the proof of Lemma~\ref{nuBPlemma}, using instead that
\begin{equation*}
    {}_{\nu\hfp}\E_2^{\starstarstar}\cong {}_{\nu\hfp}\E_{\infty}^{\starstarstar}\cong \bF_p[\lambda,h,v_1,v_2\ldots,t_1,t_2,\ldots],
\end{equation*}
where the $t_i$ are detected in Adams filtration 0.
\end{proof}

From the classical homotopy of $\BP$ and $\BP_*\BP$, we easily get that
\begin{equation*}
    \begin{split}
        \pi_{\starstar}(\nubp[\lambda^{-1}])&\cong\bZ_{(p)}[\lambda^{\pm 1},v_1,v_2,\ldots], \\
        \pi_{\starstar}((\nubp)_{\lambda}^{\wedge}[\lambda^{-1}])&\cong\bZ_{p}^{\wedge}[\lambda^{\pm 1},v_1,v_2,\ldots], \\
        \pi_{\starstar}((\nubp\otimes\nubp)[\lambda^{-1}])&\cong \bZ_{(p)}[\lambda^{\pm 1},v_1,v_2,\ldots,t_1,t_2,\ldots], \\
        \pi_{\starstar}((\nubp\otimes\nubp)_{\lambda}^{\wedge}[\lambda^{-1}])&\cong \bZ_{p}^{\wedge}[\lambda^{\pm 1},v_1,v_2,\ldots,t_1,t_2,\ldots].
    \end{split}
\end{equation*}

Using this, the previous two lemmas, and Lemma~\ref{fracture}, we can identify $\nubp_{\starstar}$ and $\nubpbp$:

\begin{theorem}
\label{BPhomotopytheorem}
As $\bZ_{(p)}[\lambda]$-algebras,
\begin{equation*}
    \begin{split}
    &\nubp_{\starstar}\cong \bZ_{(p)}[\lambda,h,v_1,v_2,\ldots]/(\lambda h=p), \\
    &\nubpbp\cong \nubp_{\starstar}[t_1,t_2,\ldots]. 
    \end{split}
\end{equation*}
\end{theorem}

\begin{proof}
By Lemma~\ref{fracture} and the fact that $\nubp_{\starstar}$ and $\nubpbp$ are concentrated in  topological degrees, $\nubp_{\starstar}$ and $\nubpbp$ each fit in pullback squares of $\bZ_{(p)}[\lambda]$-algebras

\begin{equation*}
    \begin{tikzcd}
        \nubp_{\starstar} \ar[r] \ar[d] & \bZ_p^{\wedge}[\lambda,h,v_1,v_2,\ldots]/(\lambda h=p) \ar[d] \\
        \bZ_{(p)}[\lambda^{\pm 1},v_1,v_2,\ldots] \ar[r] & \bZ_{p}^{\wedge}[\lambda^{\pm 1},v_1,v_2,\ldots] \\
    \end{tikzcd}
\end{equation*}

\begin{equation*}
    \begin{tikzcd}
        \nubpbp \ar[r] \ar[d] & \bZ_p^{\wedge}[\lambda,h,v_1,v_2,\ldots,t_1,t_2,\ldots]/(\lambda h=p) \ar[d] \\
        \bZ_{(p)}[\lambda^{\pm 1},v_1,v_2,\ldots,t_1,t_2,\ldots] \ar[r] & \bZ_{p}^{\wedge}[\lambda^{\pm 1},v_1,v_2,\ldots,t_1,t_2,\ldots]. \\
    \end{tikzcd}
\end{equation*}
\end{proof}

\subsection{Hopf algebroid structure of $(\nubp_{\starstar},\nubpbp)$}

Working over the ground ring $\bZ_{(p)}[\lambda]$, we can use Theorem~\ref{BPhomotopytheorem} and the $\lambda$-localization map to describe the remaining Hopf algebroid structure of $(\nubp_{\starstar},\nubpbp)$. Because the homotopy groups of $\nubp$ and $\nubp\otimes\nubp$ are both $\lambda$-torsion free, for all integers $k,s$ we have inclusions $\nubp_{k,k+s}\hookrightarrow\BP_k$ and $\nubp_{k,k+s}\nubp\hookrightarrow\BP_k\BP$ induced by $\lambda^{-1}:\Syn_{\hfp}\to\Sp$. This implies that the remaining Hopf algebroid structure maps are entirely determined by their classical analogs:

\begin{proposition}
\label{BPhopfalgebroidprop}
The $\bZ_{(p)}[\lambda]$-module maps
\begin{equation*}
    \begin{split}
       \eta_L,\eta_R: &\nubp_{\starstar}\to\nubpbp \\
       \epsilon: &\nubpbp\to \nubp_{\starstar} \\
       \Delta: &\nubpbp\to\nubpbp\otimes_{\nubp_{\starstar}}\nubpbp \\
       c:&\nubpbp\to\nubpbp
    \end{split}
\end{equation*}

associated to the Hopf algebroid $(\nubp_{\starstar},\nubpbp)$ are determined completely by the commuting diagrams
\begin{figure}[h!]
  \begin{subfigure}{0.5\textwidth}
    \centering
    \begin{tikzcd}
      \nubp_{\starstar} \ar[r,"\lambda^{-1}"] \ar[d,"\eta_L"'] & \BP_* \ar[d,"\eta_L"] \\
      \nubpbp \ar[r,"\lambda^{-1}"'] & \BP_*\BP 
    \end{tikzcd}
  \end{subfigure}%
  \begin{subfigure}{0.5\textwidth}
    \centering
    \begin{tikzcd}
      \nubp_{\starstar} \ar[r,"\lambda^{-1}"] \ar[d,"\eta_R"'] & \BP_* \ar[d,"\eta_R"] \\
      \nubpbp \ar[r,"\lambda^{-1}"'] & \BP_*\BP 
    \end{tikzcd}
  \end{subfigure}
\end{figure}

\begin{figure}[h!]
    \centering
    \begin{tikzcd}
      \nubpbp \ar[r,"\lambda^{-1}"] \ar[d,"\Delta"'] & \BP_*\BP \ar[d,"\Delta"] \\
    \nubpbp\otimes_{\nubp_{\starstar}}\nubpbp \ar[r,"\lambda^{-1}"'] & \BP_*\BP\otimes_{\BP_{*}}\BP_{*}\BP 
    \end{tikzcd}
\end{figure}

\begin{figure}[h!]
  \begin{subfigure}{0.5\textwidth}
    \centering
    \begin{tikzcd}
      \BP_{\starstar} \BP^{\bF_p} \ar[r,"\lambda^{-1}"] \ar[d,"\epsilon"'] & \BP_*\BP \ar[d,"\epsilon"] \\
      \nubp_{\starstar} \ar[r,"\lambda^{-1}"'] & \BP_* 
    \end{tikzcd}
  \end{subfigure}%
  \begin{subfigure}{0.5\textwidth}
    \centering
    \begin{tikzcd}
      \BP_{\starstar}\BP^{\bF_p} \ar[r,"\lambda^{-1}"] \ar[d,"c"'] & \BP_*\BP \ar[d,"c"] \\
    \nubpbp \ar[r,"\lambda^{-1}"'] & \BP_*\BP 
    \end{tikzcd}
  \end{subfigure}
\end{figure}
\end{proposition}

In particular, the formulas for each of these maps are the same as the classical ones \cite{Rav86}, up to multiples of $\lambda$ and terms which contain multiples of $h$ instead of multiples of $p$. In fact, we can precisely write down these synthetic formulas. Classically, the formulas for several of these maps are more easily described in terms of Hazewinkel's log coefficient generators $l_n\in\BP_*\otimes\bQ$ rather than $v_n$. Synthetically, we define $l_n\in\nubp_{\starstar}\otimes\bQ$ in terms of $v_n\in\nubp_{\starstar}$ as follows:

\begin{definition}
\label{lndef}
    Let $l_n\in\nubp_{2p^n-2,2p^n-1}\otimes\bQ$ for $n\geq 0$ be defined recursively as
    \begin{align*}
        pl_0=h, & \hspace{.5cm}pl_1 = v_1, \\
        pl_n = v_n + \sum_{0<i<n} &\lambda^{p^i}l_iv_{n-i}^{p^i},\hspace{.5cm} n\geq 2.     
    \end{align*}
\end{definition}

The $\lambda$-inversion functor $\lambda^{-1}:\nubp_{\starstar}\otimes\bQ\to\BP_*\otimes\bQ$ after tensoring with $\bQ$ sends $l_0\mapsto 1$ and $l_n\mapsto l_n$ for $n\geq 1$. By Proposition~\ref{BPhopfalgebroidprop} and the formulas in \cite{Rav76} and \cite{Rav86}, we get the following synthetic formulas:

\begin{theorem}
\label{BPformulathm}
On generators $v_n\in\nubp_{\starstar}$, $l_n\in\nubp_{\starstar}\otimes\bQ$, and $t_n\in\nubpbp$, the structure maps of Proposition~\ref{BPhopfalgebroidprop} are determined by the formulas
\begin{align*}
    \eta_L(v_n)&=v_n, \\
    \eta_R(l_n)&=\sum_{i=0}^{n}l_it_{n-i}^{p^i}, \\
    \epsilon(v_n)=v_n,&\hspace{.5cm} \epsilon(t_n)=0, \\
    \sum_{i+j=n}l_i(\Delta(t_j))^{p^i}&=\sum_{h+i+j=n}l_h(t_i^{p^h}\otimes t_j^{p^{h+i}}), \\
    \sum_{h+i+j=n}l_ht_i^{p^h}&c(t_j)^{p^{h+i}}=l_n.
\end{align*}
\end{theorem}

\begin{proof}
    This essentially follows from the formulas in \cite{Rav76} and \cite{Rav86} and a careful check that the weights of both sides of each formula coincide.
\end{proof}

\begin{remark}
\label{BPformulaexampleremark}
    Though the formulas for $\eta_R$, $\Delta$, and $c$ in Theorem~\ref{BPformulathm} have no powers of $\lambda$ present, powers of $\lambda$ can pop up when describing those formulas in terms of $v_n$ instead of $l_n$. This is essentially a consequence of the definition of $l_n$ in Definition~\ref{lndef}. For example, working in $\nubp_{\starstar}\otimes\bQ$, the first three values of $\eta_R$ on $h$, $v_1$, and $v_2$ are
    \begin{align*}
        \eta_R(h)&=h, \\
        \eta_R(v_1)&=v_1+ht_1, \\
        \eta_R(v_2)&=v_2+ht_2+v_1t_1^p-\frac{\lambda^p}{p}\left(\sum_{j=0}^p\binom{p+1}{j}h^{p-j+1}v_1^jt_1^{p-j+1}\right).
    \end{align*}
    If we take $p=2$, for example, then via the relation $\lambda h=2$, this reduces to
    \begin{equation*}
        \eta_R(v_2)=v_2+ht_2-5v_1t_1^2-2ht_1^3-3\lambda v_1^2t_1,
    \end{equation*}
    which makes sense over $\bZ_{(2)}$.
\end{remark}

\subsection{Synthetic Adams-Novikov SS}

As a consequence of Theorem~\ref{BPhomotopytheorem}, the Hopf algebroid $(\nubp_{\starstar},\nubpbp)$ is flat. In the usual way, this gives a nice identification of the $\E_2$-page of the $\nubp$-Adams spectral sequence in terms of Ext groups:

\begin{proposition}
\label{synANSSe2page}
Suppose $X$ is an $\bF_p$-synthetic spectrum. Then the $\E_2$-page of the $\nubp$-Adams spectral sequence for $X$ is isomorphic to
\begin{equation*}
    \E_2=\Ext_{\nubpbp}^{\starstarstar}(\nubp_{\starstar},\nubp_{\starstar}X).
\end{equation*}
\end{proposition}

Whenever $X$ is bounded below in the sense that there's an $n\in\bZ$ such $\pi_{k,w}X=0$ for $k<n$, the $\nubp$-Adams spectral sequence turns out to converge to the homotopy of the $p$-localization of $X$. This is analogous to the classical situation, where the $\BP$-Adams spectral sequence for a bounded-below spectrum converges to the homotopy of its $p$-localization. We state the main result here, but defer the proof to Appendix~\ref{nilcompappendix}, which uses the machinery of \cite{Man21}:

\begin{theorem}(\ref{synANSSconvergencethm})
For every bounded-below $X\in\Syn_{\bF_p}$,
    $$
    X_{\nubp}^{\wedge}\simeq X_{(p)}
    $$
    where $X_{\nubp}^{\wedge}$ denotes $\nubp$-nilpotent completion and $p$-localization is taken with respect to $p\in\pi_{0,0}\bS_{\bF_p}\cong\bZ$.    
\end{theorem}

\subsection{A synthetic algebraic Novikov SS}
\label{synalgnsssection}

Classically, the algebraic Novikov spectral sequence (aNSS) is a spectral sequence that tries to compute the $\E_2$-page
\begin{equation*}
    {}_{\mathrm{AN}}\E_2^{s,t}=\Ext_{\BP_*\BP}^{s,t}(\BP_*,\BP_*X)
\end{equation*}

of the Adams-Novikov spectral sequence (ANSS) for a spectrum $X$. The ANSS $\E_2$-page is the homology of the $\BP_*\BP$-cobar complex, which we can filter by powers of the ideal $I:=\ker (\BP_*\to \bF_p)=(p,v_1,v_2,\ldots)$. The resulting spectral sequence is the aNSS and it has $\E_2$-page
\begin{equation*}
    \E_2=\Ext_{\Gr_*(\BP_*\BP)}^{\starstarstar}(\Gr_*(\BP_*),\Gr_*(\BP_*X)),
\end{equation*}

where $\Gr_*(-)$ is the associated graded functor. The pair $(\Gr_*(\BP_*),\Gr_*(\BP_*\BP))$ is a bigraded Hopf algebroid satisfying $\Gr_*(\BP_*)\cong \bF_p[a_0,a_1,\ldots]$ with $a_i$ in bidegree $(u,t)=(1,2p^i-2)$ and $\Gr_*(\BP_*\BP)\cong \Gr_*(\BP_*)[b_1,b_2,\ldots]$ with $b_j$ in bidegree $(0,2p^j-2)$, where $u$ is the algebraic Novikov filtration and $t$ is internal degree. The $a_i$ and $b_j$ correspond to $v_i$ and $t_j$ respectively.

\bigskip

Using the results earlier in this section which identify the Hopf algebroid $(\nubp_{\starstar},\nubpbp)$, we can imitate the construction of the aNSS in $\Syn_{\hfp}$ to get what we call the \textit{synthetic algebraic Novikov spectral sequence} (synthetic aNSS). Our main result in this section is an identification of the synthetic aNSS of the sphere as a $\lambda$-Bockstein spectral sequence. This result implies that the synthetic aNSS and classical aNSS contain equivalent information. In particular, the $\E_2$-page of the synthetic ANSS can be computed using knowledge of classical aNSS differentials. We do this in Section~\ref{ANSSsphere}.

\bigskip

We first define the synthetic algebraic Novikov filtration of $\nubp_{\starstar}$, analogous to the classical one on $\BP_*$.

\begin{definition}
Let $J:=\ker (\nubp_{\starstar}\to\bF_p[\tau])=(h,v_1,v_2,\ldots)$ be the kernel of the map which kills $h$ and all of the $v_i$. We filter $\nubp_{\starstar}$ and $\nubpbp$ by letting $F^s\nubp_{\starstar}=J^s\cdot\nubp_{\starstar}$ and $F^s\nubpbp=J^s\cdot\nubpbp$.
\end{definition}

Immediately, we can identify the associated graded of $\nubp_{\starstar}$ and $\nubpbp$. It is a $\lambda$-extended version of the classical associated graded of $\BP$ and $\BP_*\BP$:

\begin{lemma}
\label{assgradlemma}
If $\Gr_*(-)$ is the associated graded functor, then
    \begin{align*}
        \Gr_*(\nubp_{\starstar})&\cong\bF_p[\lambda,a_0,a_1,\ldots]\cong \Gr_*(\BP_*)[\lambda], \\
        \Gr_*(\nubpbp)&\cong \Gr_*(\nubp_{\starstar})[b_1,b_2,\ldots]\cong\Gr_*(\BP_*\BP)[\lambda].
    \end{align*}
\end{lemma}

\begin{remark}
Note that here the elements $a_i$ live in tridegree $(u,t,w)=(1,2p^i-2,2p^i-1)$ and the elements $b_j$ live in tridegree $(0,2p^n-2,2p^n-2)$, where $u$ is the synthetic algebraic Novikov filtration, $t$ is internal degree, and $w$ is synthetic weight. The elements $a_i$ and $b_j$ correspond to $v_i$ and $t_j$ respectively.
\end{remark}

We can filter the $\nubpbp$-cobar complex in an analogous way to the classical one. The resulting spectral sequence is the synthetic algebraic Novikov spectral sequence with $\E_2$-page
\begin{equation*}
    \E_2^{f,u,t,w}(X)=\Ext_{\Gr_*(\nubpbp)}^{f,t,w}(\Gr_*(\nubp_{\starstar}),\Gr_u(\nubp_{\starstar}X))
\end{equation*}

for an $\bF_p$-synthetic spectrum $X$. Scarily enough, this is a quad-graded spectral sequence with differentials $d_r:\E_r^{f,u,t,w}\to \E_r^{f+1,u+r-1,t,w}$, where $u$ denotes the synthetic algebraic Novikov filtration. Luckily when $X=\bS_{\bF_p}$, we can identify this spectral sequence as a $\lambda$-Bockstein spectral sequence:

\begin{theorem}
\label{lambdabockstein}
    The synthetic algebraic Novikov spectral sequence for $X=\bS_{\bF_p}$ is isomorphic to a $\lambda$-Bockstein spectral sequence.
\end{theorem}

\begin{proof}
    The proof is similar to the proof of \cite[Thm. 2.8]{BX23}. By Lemma~\ref{assgradlemma}, we have
    \begin{equation*}
    \begin{split}
        \E_2^{*,*,*,*}(\bS_{\bF_p})&\cong\Ext^{*,*,*,*}_{\Gr_*(\BP_*\BP)[\lambda]}(\Gr_*(\BP_*)[\lambda],\Gr_*(\BP_*)[\lambda]) \\
        &\cong \Ext^{*,*,*}_{\Gr_*(\BP_*\BP)}(\Gr_*(\BP_*),\Gr_*(\BP_*)[\lambda]) \\
        &\cong \Ext^{*,*,*}_{\Gr_*(\BP_*\BP)}(\Gr_*(\BP_*),\Gr_*(\BP_*))[\lambda]
    \end{split}
    \end{equation*}

    where the first isomorphism is a change-of-ring isomorphism along the 
    map 
    $$(\Gr_*(\BP_*)[\lambda],\Gr_*(\BP_*\BP)[\lambda])\to (\Gr_*(\BP_*),\Gr_*(\BP_*\BP))$$
    which kills $\lambda$ and the second isomorphism holds because $\Gr_*(\BP_*)[\lambda]$ is a trivial $\Gr_*(\BP_*\BP)$-comodule. For $\bS_{\bF_p}/\lambda$, the $\E_2$-page is
    \begin{equation*}
        \E_2^{*,*,*,*}(\bS_{\bF_p}/\lambda)\cong \Ext^{*,*,*}_{\Gr_*(\BP_*\BP)}(\Gr_*(\BP_*),\Gr_*(\BP_*)),
    \end{equation*}
    the classical aNSS $\E_2$-page, with a generator $x\in\Ext^{f,u,t}$ living in quad-degree $(f,u,t,u+t)$. This spectral sequence collapses at the $\E_2$-page for degree reasons. Hence for $\bS_{\bF_p}$, differentials must be of the form $d_r(x)=\lambda^{r-1}y$. There are no $\lambda$-hidden extensions to worry about, so this finishes the proof.
\end{proof}

Classically, via work of Miller \cite{Mil81}, the $\E_2$-page of the aNSS for the sphere is isomorphic to the $\E_2$-page of the Cartan-Eilenberg spectral sequence associated with the extension of Hopf algebras $\cP_*\to\cA_*\to \cE_*$:
\begin{equation*}
    \Ext_{\Gr_*(\BP_*\BP)}^{\starstarstar}(\Gr_*(\BP_*),\Gr_*(\BP_*))\cong \Ext^{*,*}_{\cP_*}(\bF_p,\Ext_{\cE_*}^{*}(\bF_p,\bF_p)).
\end{equation*}

One can prove this \cite[Thm. 4.4.4, A1.3.12]{Rav86} by doing a change-of-ring isomorphism  along the map of Hopf algebroids $(\Gr_*(\BP_*),\Gr_*(\BP_*\BP))\to (\bF_p,\cP_*)$ and noticing that $\Gr_*(\BP_*)\cong \Ext_{\cE_*}^{*,*}(\bF_p,\bF_p)$. Curiously enough, the $\E_2$-page of the synthetic aNSS for $X=\bS_{\bF_p}$ corresponds to the $\E_2$-page of the $\bC$-motivic Cartan Eilenberg spectral sequence associated to the extension of Hopf algebras $\cP_*[\tau]\to\mathcal{A}_{\starstar}^{\BP}\to \cE_*[\tau]$:

\begin{proposition}
There is an isomorphism
\begin{equation*}
    \Ext_{\Gr_*(\nubpbp)}^{*,*,*,*}(\Gr_*(\nubp_{\starstar}),\Gr_*(\nubp_{\starstar}))\xrightarrow{\cong} \Ext_{\cP_*[\tau]}^{\starstarstar}(\bF_p[\tau],\Ext_{\cE_*[\tau]}^{*}(\bF_p[\tau],\bF_p[\tau]))
\end{equation*}
sending $\lambda$ to $\tau$.
\end{proposition}

\begin{proof}
Consider the composition of maps
\begin{equation*}
\begin{split}
(\Gr_*(\nubp_{\starstar}),\Gr_*(\nubpbp))&\xrightarrow{\cong}(\Gr_*(\BP_*)[\lambda],\Gr_*(\BP_*\BP)[\lambda]) \\
&\xrightarrow{} (\bF_p[\lambda],\cP_*[\lambda]) \\
&\xrightarrow{\cong}(\bF_p[\tau],\cP_*[\tau])    
\end{split}
\end{equation*}

of Hopf algebroids, where the last map changes coordinates from $\lambda$ to $\tau$. By change-of-rings \cite[A1.3.12]{Rav86}, we get isomorphisms
\begin{equation*}
\begin{split}
  \Ext_{\Gr_*(\nubpbp)}^{*,*,*,*}(\Gr_*(\nubp_{\starstar}),\Gr_*(\nubp_{\starstar}))&\cong\Ext_{\cP_*[\tau]}^{\starstarstar}(\bF_p[\tau],\Gr_*(\BP_*)[\tau]) \\
  &\cong\Ext_{\cP_*[\tau]}^{\starstarstar}(\bF_p[\tau],\Ext_{\cE_*[\tau]}^{*}(\bF_p[\tau],\bF_p[\tau]))  
\end{split}
\end{equation*}

where the last isomorphism follows from the fact that
\begin{equation*}
    \begin{split}
        \Gr_*(\BP_*)[\tau]&\cong\Ext_{\cE_*}^{*}(\bF_p,\bF_p)[\tau] \\
        &\cong \Ext_{\cE_*[\tau]}^{*}(\bF_p[\tau],\bF_p[\tau]).
    \end{split}
\end{equation*}
\end{proof}


\begin{remark}
\label{bisyn}
The result of this and \cite[Thm. 2.8]{BX23} is that we (almost) have a Miller square-type diagram of spectral sequences
\begin{equation*}
    \begin{tikzcd}
        \Ext^{\starstarstar}_{\Gr_*(\BP_*\BP)}(\Gr_*(\BP_*),\Gr_*(\BP_*))[\lambda] \ar[r,"\cong"] \ar[d, Rightarrow,"\lambda\mathrm{-BSS}"'] & \Ext_{\cP_*}^{*,*}(\bF_p,\Ext_{\cE_*}^{*}(\bF_p,\bF_p))[\tau] \ar[d, Rightarrow,"\tau\mathrm{-BSS}"] \\
        \Ext_{\nubpbp}^{\starstarstar}(\nubp_{\starstar},\nubp_{\starstar}) \ar[d, Rightarrow,"\BP^{\bF_p}\mathrm{-Adams\text{ }SS}"'] & \Ext^{\starstarstar}_{\mathcal{A}_\starstar^{\BP}}(\bF_p[\tau],\bF_p[\tau]) \ar[d, Rightarrow, "\bF_p^{\BP}\mathrm{-Adams\text{ }SS}"] \\
        \pi_{\starstar}(\bS_{\bF_p}) & \pi_{\starstar}(\bS_{\BP})
    \end{tikzcd}
\end{equation*}

This fails to be a Miller square because $\bS_{\bF_p}\in\Syn_{\hfp}$ and $\bS_{\BP}\in\Syn_{\BP}$ are in different categories of synthetic spectra. However, we conjecture this could be remedied by considering both objects in a category of \textit{bisynthetic spectra} which simultaneously deforms the $\BP^{\bF_p}$- and $\bF_p^{\BP}$-Adams spectral sequences. The first mention of such a category occurred in the introduction of \cite{BHS20}. The authors intend to produce a model for such a category in future work.   
\end{remark}

\numberwithin{theorem}{section}

\section{\texorpdfstring{Adams spectral sequences for $\bS_E/\tau$-modules}{Adams spectral sequences for cofiber of tau modules}}
\label{Clambdasection}
Let $R$ be an $\bE_1$-ring spectrum and let $E$ be an Adams-type homology theory. In this section we characterize the $\nu_E(R)$-based Adams spectral sequence for $\bS_E/\tau$ in $\mathrm{Syn}_E$ and identify it in a more familiar form in the case where $(E_*, E_*E)$ is a Hopf algebra. 

\begin{proposition}
\label{ClambdaSS}
The $\nu_E(R)$-Adams spectral sequence for $\nu_EX/\tau$ in $\mathrm{Syn}_{E}$ is isomorphic to the $E_*R$-Adams spectral sequence for $E_*X$ in $\mathrm{Stable}_{E_*E}$.
\end{proposition}
\begin{proof}
The $\nu_E(R)$-based Adams spectral sequence in $\mathrm{Syn}_E$ is given by applying the functor $\text{Hom}_{\mathrm{Syn_E}}(\nu_E(\bS), -)$ to the Moore complex for the cosimplicial resolution 
\begin{align*}
    \Delta &\to \mathrm{Syn}_{E}\\ n &\mapsto \nu X \otimes \nu R^{\otimes n + 1}
\end{align*}
and applying homotopy. Recall that $(-)/\tau$ is a symmetric monoidal colimit preserving functor from $\mathrm{Syn}_{E}$ to $\mathrm{Stable}_{E_*E}$ with $\nu_E(X)/\tau \cong E_*X$ \cite{Pst22}. So $(-)/\tau$ sends the synthetic $\nu_E(R)$-Adams resolution to the cosimplicial resolution
\begin{align*}
    \Delta &\to \mathrm{Stable}_{E_*E}\\
    n &\mapsto E_*X\otimes E_*R^{\otimes n + 1}
\end{align*}
associated to the $E_*R$-Adams spectral sequence for $E_*X$ in $\mathrm{Stable}_{E_*E}$. 
\end{proof}

Let $\Gamma$ be a Hopf algebra over a commutative ring $k$ and let $\Gamma \to \Sigma$ be a surjection of Hopf algebras for which the map $\Phi \coloneqq \Gamma \square_{\Sigma} k \to \Gamma$ is a map of $\Gamma$-comodule algebras. In \cite{tmf} Bruner-Rognes describe a Cartan-Eilenberg type spectral sequence for such a (not necessarily conormal) Hopf extension $\Phi \to \Gamma \to \Sigma$ which they call the Davis-Mahowald spectral sequence. This spectral sequence was first developed by Davis-Mahowald where they use it to compute $\Ext_{\mathcal{A}(2)}(M, \bF_2)$ for $\mathcal{A}(2)$-modules $M$ using the Davis-Mahowald spectral sequence for the non-normal map $\mathcal{A}(1) \to \mathcal{A}(2)$ of Hopf algebras.

\begin{remark}[\cite{Bel20}]
    If $\Phi$ is a $\Gamma$-comodule algebra and $M$ is a $\Gamma$-comodule then the $\Phi$-based Adams spectral sequence converging to $\Ext_{\Gamma}(k, M)$ in $\mathrm{Stable}(\Gamma)$ coincides with the Davis-Mahowald spectral sequence for the Hopf extension $\Phi \to \Gamma \to \Sigma$ starting at $\E_1$-page which has the form
    \begin{align*}
        \Ext_{\Sigma}(k, \overline{\Phi}^{\otimes s} \mathord{\otimes} M)
    \end{align*}
    where $\overline{\Phi}$ is the coaugmentation ideal. If $\Ext_{\Sigma}(k, \Phi)$ is flat as a $\Ext_{\Sigma}(k, k)$-module, then the $\E_2$-page has the form
    \begin{align*}
        \Ext_{\Ext_{\Sigma}(k, \Phi)}(\Ext_{\Sigma}(k, k), \Ext_{\Sigma}(k, M)).
    \end{align*}
\end{remark}

\begin{remark}
    In the category of stable comodules, $\Ext_{\Gamma}(M,N)$ is given as the homotopy groups of the mapping spectrum between $M$ and $N$. In analogy with stable homotopy theory, one might rewrite the $\E_2$-page above more recognizably as
    $\Ext_{\pi_{\star}\Phi}(\pi_{\star}(k), \pi_{\star}(M))$.
\end{remark}

\begin{proposition}[\cite{Bel20}]
For $\Phi \to \Gamma \to \Sigma$ a Hopf extension, the Davis-Mahowald spectral sequence for $\Phi \to \Gamma \to \Sigma$ is isomorphic to the spectral sequence given by filtering the cobar complex on $\Gamma$ by
\begin{align*}
    \Fil^s C^n_\Gamma = \{a_1 | \cdots | a_n \in C^n_\Gamma | \#[(a_1, ..., a_n)\cap G] \geq s\}
\end{align*}
for $G = \ker(\Gamma \to \Sigma)$.
\end{proposition}

\begin{example}\label{cessexample}
The $\nu_{\bF_2} (\BP)$-Adams spectral sequence for $\bS_{\bF_2}/\lambda$ in $\mathrm{Syn}_{\bF_2}$ is isomorphic to the Cartan-Eilenberg spectral sequence for the conormal extension
\begin{align*}
    \bF_2[\zeta_1^2, \zeta_2^2, ...] \to \bF_2[\zeta_1, \zeta_2, ...] \to E(\zeta_1, \zeta_2, ...).
\end{align*}
\end{example}

\begin{example}
The $\nu_{\bF_2} (\bZ)$-Adams spectral sequence for $\nu_{\bF_2} (ko) /\lambda$ in $\mathrm{Syn}_{\bF_2}$ is isomorphic to the Davis-Mahowald spectral sequence for the Hopf extension
\begin{align*}
    \mathcal{A}(1)_* \Box_{\mathcal{A}(0)_*} \bF_2 \to \mathcal{A}(1)_* \to \mathcal{A}(0)_*.
\end{align*}
\end{example}

\begin{example}
The $\nu_{\bF_2}(ko)$-Adams spectral sequence for $\nu_{\bF_2}(tmf) /\lambda$ in $\mathrm{Syn}_{\bF_2}$ is isomorphic to the Davis-Mahowald spectral sequence for the Hopf extension
\begin{align*}
    \mathcal{A}(2)_* \Box_{\mathcal{A}(1)_*} \bF_2 \to \mathcal{A}(2)_* \to \mathcal{A}(1)_*.
\end{align*}
\end{example}

\begin{example}
The $\nu_{\bF_2}(tmf)$-Adams spectral sequence for $\bS_{\bF_2}/\lambda$ in $\mathrm{Syn}_{\bF_2}$ is isomorphic to the Davis-Mahowald spectral sequence for the Hopf extension
\begin{align*}
    \mathcal{A}_* \Box_{\mathcal{A}(2)_*} \bF_2 \to \mathcal{A}_* \to \mathcal{A}(2)_*.
\end{align*}
\end{example}

\section{\texorpdfstring{Twisted $t$-structures of Stable $\infty$-Categories}{Twisted t-structures of Stable Infinity Categories}}
\label{tstructuresection}

Let $\cC$ be a stable $\infty$-category with a compact object $K$. By a twist of $\cC$ we mean an automorphism $F:\cC\to \cC$. We will notate by $\Sigma^{n,m}:\cC\to \cC$ the composition $\Sigma^n\circ F^m$. Note that $\Sigma^{n,m}$ will preserve compact objects for all $n,m$ as a composition of autoequivalences.

\begin{definition}
We say that $\cC$ is generated by $K$ under bigraded suspensions if the smallest subcategory of $\cC$ containing $\Sigma^{n,m}K$ for all $n,m\in \bZ$ and closed under colimits is $\cC$ itself.
\end{definition}

\begin{remark}
    We will mostly be working with cellular subcategories where the generation assumption is forcibly imposed.
\end{remark}

The goal of this section is to construct a ``twisted'' $t$-structure on $\cC$ whose connective part is detected by the mapping groups $[\Sigma^{n,m}K,-]$. To this end, fix $\cC,F,K$ as described above such that $K$ generates under colimits and put
\[
\cC_{\geq 0}=\{X\in \cC\mid [\Sigma^{n,m}K, X]=0\text{ whenever } n+m<0\}.
\]

\begin{lemma}
\label{twistediststructure}
    The category $\cC_{\geq 0}$ is the connective part of a $t$-structure on $\cC$ as soon as $K\in \cC_{\geq 0}$.
\end{lemma}

\begin{proof}
    Write $\cC'_{\geq 0}$ for the closure of the objects $\Sigma^{n,m}K$, $n+m\geq 0$ in $\cC$ under colimits and extensions. We claim that $\cC'_{\geq 0}\simeq \cC_{\geq 0}$. Having shown this, the lemma follows from \cite[1.4.4.11]{HA}.

    \bigskip
    
    To prove the equivalence, first observe that $\cC_{\geq 0}$ contains all of the $\Sigma^{n+m}K$ for $n+m\geq 0$. To prove it is closed under colimits it suffices to show it is closed under arbitrary coproducts and cofibers. For the former, finite coproducts are immediate. Because all coproducts are filtered colimits of finite coproducts, we have all coproducts by the compactness of $K$. That we have all cofibers follows from the induced long exact sequence. Closure under extensions follows similarly. As a result we have $\cC'_{\geq 0}\subset \cC_{\geq 0}$. To conclude it suffices to show that the latter category is generated under colimits by the same shifts of $K$. We note that $\cC$ being generated by the shifts is equivalent to asking that $X\to Y$ is an equivalence if and only if the induced
    \[
    \Map(\Sigma^{n,m}K,X)\to \Map(\Sigma^{n,m}K,Y)
    \]
    is an equivalence of spaces for all $n,m$. Replacing mapping spaces by mapping spectra it suffices to check that
    \[
    \Map^\Sp(\Sigma^{0,m}K,X)\to \Map^\Sp(\Sigma^{0,m}K,Y)
    \]
    is an equivalence for all $m$. But $\pi_k\Map^\Sp(\Sigma^{0,m}K,X)=\pi_0\Map^\Sp(\Sigma^{k,m}K,X)$ so that it suffices to check for an isomorphism on the mapping groups $[\Sigma^{k,m}K,-]$. For $X,Y$ connective, the latter only depends on those where $k+m\geq 0$.
\end{proof}

\begin{definition}
\label{twistedtstructdef}
    The $(K,F)$-twisted $t$-structure on $\cC$ is the $t$-structure with connective part $\cC_{\geq 0}$ as described above. 
\end{definition}

In practice, we will drop $K,F$ from the notation. We wish to use twisted $t$-structures to apply the nilpotent completion results of \cite{Man21} and for defining deformations in Section~\ref{deformationsection}. The remainder of this section will work to verify his assumptions laid out in Section 2.1.1 of loc. cit.

\begin{lemma}
\label{leftcomplemma}
     The $(K,F)$-twisted $t$-structure is left and right complete; i.e., for all $X$ we have $X\simeq\lim_n \tau_{\leq n} X\simeq\colim_n\tau_{\geq n}X$.
\end{lemma}

\begin{proof}
    By \cite[1.2.1.19]{HA} it suffices to show that the only object in $\cC_{\leq \infty}:=\cap_n\cC_{\geq n}$ is $0$. But due to colimit generation, an object $X$ is $0$ as soon as $[\Sigma^{n,m}K,X]=0$ for all $n,m$ as shown in the proof of Lemma~\ref{twistediststructure}. The analogous argument applies for right completeness.
\end{proof}

\begin{lemma}
\label{filtcolimlemma}
    The truncations $\tau_{\geq n}(-)$ in the $(K,F)$-twisted t-structure commute with filtered colimits.
\end{lemma}

\begin{proof}
    Given a filtered diagram $X_\alpha$ we see that the compactness of $K$ implies that $\colim \tau_{\geq n} X_{\alpha}\to \tau_{\geq n} \colim X_{\alpha}$ induces equivalences on $[\Sigma^{n,m}K,-]$ which detect equivalences as they detect contractibility.
\end{proof}

We now specialize to the case that $\cC$ is presentably symmetric monoidal with compact unit and $K=\one$.

\begin{lemma}
\label{tmultlemma}
    The $(K,F)$-twisted $t$-structure on $\cC$ is multiplicative; i.e., if $X\in \cC_{\geq a}$ and $Y\in \cC_{\geq b}$ then $X\otimes Y\in \cC_{\geq a+b}$.
\end{lemma}

\begin{proof}
    The functor $-\otimes Y$ preserves all colimits and the categories $\cC_{\geq k}$ are closed under colimits, so that we may reduce to checking on the shifts $\Sigma^{n,m}K\otimes Y$ for $n+m=b$ where the claim is clear.
\end{proof}

\begin{remark}
\label{extraaxis}
    We will use an extension of the arguments above to the case where $\cC$ requires more compact generators than are being used in the $t$-structure. Suppose $\cC$ is generated by objects $\Sigma^{n,m,k}K$, where $m$ corresponds to twists by the functor $F$ above and $k$ denotes twists by some new commuting automorphism and the shifts $\Sigma^{0,0,k}K\in \cC_{\geq 0}$. Then we may run the same arguments as above, replacing mapping objects with their graded versions in the third shift as necessary.
\end{remark}

For computing nilpotent completions in $\Stable(\Gamma)^{\cell}$ and $\Syn_E^{\cell}$ via Appendix~\ref{nilcompappendix}, we need to identify the heart of the $t$-structure of Definition~\ref{twistedtstructdef}. The Ext groups of an object $X$ in the heart vanish except for $\Ext_{\Gamma}^{t,t}(A,X)$ or $\Ext^{t,t,*}_{\Gamma}(A,X)$ with $t\in\bZ$, depending on whether $\Gamma$ is graded or bigraded. We first describe notation for a graded version of these Ext groups when $X=A$:

\begin{definition}
    Let $\Ext_{\Gamma}^{t=s}$ and $\Ext_{\Gamma}^{t=s,*}$ denote the respective graded and bigraded rings defined by
    \begin{align*}
     (\Ext_{\Gamma}^{t=s})_{t}&= \Ext^{t,t}_{\Gamma}(A,A), \\
    (\Ext_{\Gamma}^{t=s,*})_{t,w}&=          \Ext^{t,t,w}_{\Gamma}(A,A)
    \end{align*}
    depending on whether $\Gamma$ is graded or bigraded.
\end{definition}

We can identify the heart in terms of $\Ext_{\Gamma}^{t=s}$-mod and $\Ext_{\Gamma}^{t=s,*}$-mod in the case where the unit $A$ is connective:

\begin{theorem}
\label{comodheartthm}
    Suppose $(A,\Gamma)$ is a graded or bigraded Hopf algebroid such that $A\in\Stable(\Gamma)_{\geq 0}$. Then the functors
    \begin{align*}
        \Ext_{\Gamma}^{*,*}(A,-)&:(\Stable(\Gamma)^{\cell})^{\heartsuit}\to\Ext_{\Gamma}^{t=s}\text{-mod} \\
        \Ext_{\Gamma}^{*,*,*}(A,-)&:(\Stable(\Gamma)^{\cell})^{\heartsuit}\to\Ext_{\Gamma}^{t=s,*}\text{-mod}
    \end{align*}
    are equivalences of $\infty$-categories, where $\Ext_{\Gamma}^{t=s}\text{-mod}$ and $\Ext_{\Gamma}^{t=s,*}\text{-mod}$ are considered as discrete $\infty$-categories.
\end{theorem}

To prove this we follow the strategy of \cite[Section 3]{GWX21}, by showing that $\Ext_{\Gamma}^{*,*}(A,-)$ is fully faithful (Lemma~\ref{fflemma}) and essentially surjective (Lemma~\ref{eslemma}). To do so, we first need a version of the universal coefficient spectral sequence in $\Stable(\Gamma)^{\cell}$ as a technical tool:

\begin{lemma}
\label{uctlemma}
    Let $\pi_{*,*}$ denote $\Ext^{-*,*}_{\Gamma}$ and $\pi_{*,*,*}$ denote $\Ext^{-*,*,*}_{\Gamma}$. For any
    \begin{equation*}
        X,Y\in\Stable(\Gamma)^\cell,
    \end{equation*}
    there is a conditionally convergent spectral sequence with $\E_2$-page
    \begin{align*}
        \E_2^{s,t,w}&=\Ext_{\pi_{*,*}A}^{s,t,w}(\pi_{*,*}(X),\pi_{*,*}(Y)), \\
        \E_2^{s,t,w,v}&=\Ext_{\pi_{*,*,*}A}^{s,t,w,v}(\pi_{*,*,*}(X),\pi_{*,*,*}(Y)).
    \end{align*}
    and differentials
    \begin{align*}
     d_r:\E_r^{s,t,w}&\to \E_r^{s+r,t+r-1,w}, \\
     d_r:\E_r^{s,t,w,v}&\to \E_r^{s+r,t+r-1,w,v},
    \end{align*}
    converging to either $[\Sigma^{t-s,w}X,Y]_{\Stable(\Gamma)^{\cell}}$ or $[\Sigma^{t-s,w,v}X,Y]_{\Stable(\Gamma)^{\cell}}$.
\end{lemma}

\begin{proof}
The proof is identical in both cases, we will record the bigraded version. The proof follows \cite[Thm. IV.4.5]{EKMM} and the subsequent \cite{DI10} almost identically. Namely, we put $K_{-1}=X$ and define $K_{i}$ to be the fiber of a map $F_i\to K_{i-1}$ where $F_i$ is a free $A$-module and the map is surjective on homotopy. The result is a free resolution of $\pi_{*,*}X$ by the $\pi_{*,*}F_i$ and a tower in $\Stable(\Gamma)^{\cell}$
\[\begin{tikzcd}
	X & X & {K_0} & {K_1} & {K_2} & {...} \\
	0 & {F_0} & {\Sigma F_1} & {\Sigma^2 F_2} & {\Sigma^3F_3}
	\arrow[from=1-1, to=1-2]
	\arrow[from=1-2, to=1-3]
	\arrow[from=1-3, to=1-4]
	\arrow[from=1-4, to=1-5]
	\arrow[from=1-5, to=1-6]
	\arrow[from=2-1, to=1-1]
	\arrow[from=2-2, to=1-2]
	\arrow[from=2-3, to=1-3]
	\arrow[from=2-4, to=1-4]
	\arrow[from=2-5, to=1-5]
\end{tikzcd}.\]
After applying the spectral mapping functor $\Map^\Sp_{\Stable(\Gamma)}(-,Y)$ we observe that the $\E_1$-page will be given by $\pi_{*,*}\Map^\Sp_{\Stable(\Gamma)}(F_i,Y)$ which splits as sums of shifts of $\pi_{*,*}Y$ and the $\E_2$-page is identified with $\Ext$ as in \cite{EKMM}. For conditional convergence, we observe, as in \cite{DI10}, that the colimit of the above tower is contractible as it has vanishing homotopy and is cellular by construction.
\end{proof}

Our main use of Lemma~\ref{uctlemma} is to prove that $\Ext_{\Gamma}^{*,*}(A,-)$ is fully faithful:

\begin{lemma}
\label{fflemma}
    Suppose $(A,\Gamma)$ is a graded or bigraded Hopf algebroid such that $A\in\Stable(\Gamma)_{\geq 0}^{\cell}$. Then the functors
    \begin{align*}
        \Ext_{\Gamma}^{*,*}(A,-)&:(\Stable(\Gamma)^{\cell})^{\heartsuit}\to\Ext_{\Gamma}^{t=s}\text{-mod} \\
        \Ext_{\Gamma}^{*,*,*}(A,-)&:(\Stable(\Gamma)^{\cell})^{\heartsuit}\to\Ext_{\Gamma}^{t=s,*}\text{-mod}
    \end{align*}
    are fully faithful.
\end{lemma}

\begin{proof}
We first note that for $X,Y\in(\Stable(\Gamma)^{\cell})^{\heartsuit}$ and $n>0$,
\begin{align*}
\Hom_{\pi_{*,*}A}(\pi_{*,*}(\Sigma^{n,0}X),\pi_{*,*}(Y))&=0=\Hom_{\Ext^{t=s}_{\Gamma}}(\pi_{*,*}(\Sigma^{n,0}X),\pi_{*,*}(Y)), \\
\Hom_{\pi_{*,*,*}A}(\pi_{*,*,*}(\Sigma^{n,0,0}X),\pi_{*,*,*}(Y))&=0=\Hom_{\Ext^{t=s,*}_{\Gamma}}(\pi_{*,*,*}(\Sigma^{n,0,0}X),\pi_{*,*,*}(Y)),
\end{align*}

for degree reasons. Whenever $n=0$,
\begin{align*}
\Hom_{\pi_{*,*}A}(\pi_{*,*}(X),\pi_{*,*}(Y))&\cong\Hom_{\Ext^{t=s}_{\Gamma}}(\pi_{*,*}(X),\pi_{*,*}(Y)), \\
\Hom_{\pi_{*,*,*}A}(\pi_{*,*,*}(X),\pi_{*,*,*}(Y))&\cong\Hom_{\Ext^{t=s,*}_{\Gamma}}(\pi_{*,*,*}(X),\pi_{*,*,*}(Y)),
\end{align*}

by the assumption that $X$ and $Y$ are in the heart. So to show fully faithful, it suffices to show that
\begin{align*}
[\Sigma^{n,0}X,Y]_{\Stable(\Gamma)^{\cell}}&\to\Hom_{\pi_{*,*}A}(\pi_{*,*}(\Sigma^{n,0}X),\pi_{*,*}(Y)),  \\
[\Sigma^{n,0,0}X,Y]_{\Stable(\Gamma)^{\cell}}\to &\text{ }\Hom_{\pi_{*,*,*}A}(\pi_{*,*,*}(\Sigma^{n,0,0}X),\pi_{*,*,*}(Y)),
\end{align*}
is an isomorphism for $n\geq 0$. Whenever $n>0$, both sides are equal to 0 since $X$ and $Y$ are assumed to be in the heart.

\bigskip

To show it's an isomorphism whenever $n=0$, we use the spectral sequence of Lemma~\ref{uctlemma}. Note that $\E_1^{t,t,0}$ and $\E_1^{t,t,0,0}$ for $t\geq 0$ are the respective $\E_1$-page degrees which compute the bigraded and trigraded versions of $[X,Y]_{\Stable(\Gamma)^{\cell}}$. Since we assume that $A$ is connective, we can choose a free resolution such that $\pi_{a,b}(\Sigma^{t,0}F_s)$ and $\pi_{a,b,c}(\Sigma^{t,0,0}F_s)$ vanish for $a+b-t<0$. For $t>0$,
\begin{align*}
    \E_1^{t,t,0}&=\Hom_{\pi_{*,*}A}(\pi_{*,*}\Sigma^{t,0}F_t,\pi_{*,*}Y) \\
    \E_1^{t,t,0,0}&=\Hom_{\pi_{*,*,*}A}(\pi_{*,*,*}\Sigma^{t,0,0}F_t,\pi_{*,*,*}Y)
\end{align*}
are both 0 for degree reasons since $Y$ is in the heart. Then the $\E_1$-page is concentrated in $t=0$ and, for degree reasons, the differentials entering and exiting $\E_1^{0,0,0}$ and $\E_1^{0,0,0,0}$ are 0. Hence the edge homomorphisms
\begin{align*}
[X,Y]_{\Stable(\Gamma)^{\cell}}&\to\Hom_{\pi_{*,*}A}(\pi_{*,*}(X),\pi_{*,*}(Y)) \\
[X,Y]_{\Stable(\Gamma)^{\cell}}&\to\Hom_{\pi_{*,*,*}A}(\pi_{*,*,*}(X),\pi_{*,*,*}(Y))
\end{align*}
 are isomorphisms, which finishes the proof.
\end{proof}

\begin{lemma}
\label{eslemma}
   The functors
    \begin{align*}
        \Ext_{\Gamma}^{*,*}(A,-)&:(\Stable(\Gamma)^{\cell})^{\heartsuit}\to\Ext_{\Gamma}^{t=s}\text{-}\mathrm{mod} \\
        \Ext_{\Gamma}^{*,*,*}(A,-)&:(\Stable(\Gamma)^{\cell})^{\heartsuit}\to\Ext_{\Gamma}^{t=s,*}\text{-}\mathrm{mod}
    \end{align*}
    are essentially surjective.
\end{lemma}

\begin{proof}
We just need to show that every $M\in\Ext^{t=s}_{\Gamma}$-mod or $M\in\Ext^{t=s,*}_{\Gamma}$-mod can be realized as the homotopy of an object in $(\Stable(\Gamma)^{\cell})^{\heartsuit}$. Note that the Ext groups of the unit $F:=\pi_0^{\heartsuit}(A)\in (\Stable(\Gamma)^{\cell})^{\heartsuit}$ satisfy
\begin{align*}
    \Ext^{*,*}_{\Gamma}(A,F)&\cong\Ext_{\Gamma}^{t=s}, \\
    \Ext^{*,*,*}_{\Gamma}(A,F)&\cong\Ext_{\Gamma}^{t=s,*}.
\end{align*}
If we have a free resolution of $M$
\begin{equation*}
\cdots\to F_2\to F_1\to F_0\to M\to 0,    
\end{equation*}
then by taking wedges of $F\in(\Stable(\Gamma)^{\cell})^{\heartsuit}$, each $F_i$ can be realized as the homotopy of an object $Z_i\in(\Stable(\Gamma)^{\cell})^{\heartsuit}$. By Lemma~\ref{fflemma}, the map $F_1\to F_0$ can be realized as Ext groups applied to a map of stable comodules $Z_1\to Z_0$. Consider the cofiber $X_1$ of this map. Its Ext groups satisfy
\begin{align*}
    \bigoplus_{s=-\infty}^{\infty}\Ext^{s-k,s}_{\Gamma}(A,X_1)&=\begin{cases}
        \mathrm{coker}(F_1\to F_0)=M, & k=0 \\
        \mathrm{ker}(F_1\to F_0) & k=1 \\
        0 & \mathrm{otherwise},
    \end{cases} \\
    \bigoplus_{s=-\infty}^{\infty}\Ext^{s-k,s,*}_{\Gamma}(A,X_1)&=\begin{cases}
        \mathrm{coker}(F_1\to F_0)=M, & k=0 \\
        \mathrm{ker}(F_1\to F_0) & k=1 \\
        0 & \mathrm{otherwise}.
    \end{cases}
\end{align*}
Then the Ext groups of the truncation $\tau_{\leq 0}X_1$ are exactly isomorphic to $M$, which finishes the proof.
\end{proof}

We can imitate the above lemmas and proofs in $\Syn_E^{\cell}$ to compute the heart $(\Syn_E^{\cell})^{\heartsuit}$. We produce the statements without proof, as their proofs are very similar to those in $\Stable(\Gamma)^{\cell}$:

\begin{lemma}
\label{synuctlemma}
    For any $X,Y\in\Syn_E^{\cell}$,
    there is a conditionally convergent spectral sequence with $\E_2$-page
    \begin{equation*}        \E_2^{s,t,w}=\Ext_{\pi_{*,*}\bS_E}^{s,t,w}(\pi_{*,*}(X),\pi_{*,*}(Y))
    \end{equation*}
    and differentials
    \begin{equation*}
     d_r:\E_r^{s,t,w}\to \E_r^{s+r,t+r-1,w}
    \end{equation*}
    converging to $[\Sigma^{t-s,w}X,Y]_{\Syn_E^{\cell}}$.
\end{lemma}

\begin{theorem}
\label{synspectraheartthm}
    Suppose $\bS_E\in (\Syn_E^\cell)_{\geq 0}$. Then the functor
    \begin{equation*}
\pi_{*,*}:(\Syn_E^{\cell})^{\heartsuit}\to\pi_{0,*}\bS_E\text{-mod}
    \end{equation*}
    is fully faithful and essentially surjective; i.e. it is an equivalence of $\infty$-categories.
\end{theorem}

\begin{remark}
    The above result is also shown in \cite[Theorem 2.2]{CD24} using Barr-Beck and monadicity techniques.
\end{remark}

We finish this section with a few relevant examples:

\begin{example}
\label{comodexampA}
    Suppose $(A,\Gamma)=(\bF_p,\cA_*)$, the dual Steenrod algebra at a prime $p$. Note that $\Stable(\cA_*)^{\cell}=\Stable(\cA_*)$ by \cite[Thm. 2.3.1]{HPS97}. By a classical calculation,
    $$\Ext_{\cA_*}^{t,t}(\bF_p,\bF_p)=\bF_p\{h_0^t\}$$
    with $t\geq 0$ and $h_0\in\Ext_{\cA_*}^{1,1}(\bF_p,\bF_p)$, and the module action corresponds to multiplication by powers of $h_0$. Hence $$\Stable(\cA_*)^{\heartsuit}=\bF_p[h_0]\text{-}\mathrm{mod},$$
    where we consider $\bF_p[h_0]$ as a graded polynomial ring with $\vert h_0\vert=1$.
\end{example}

\begin{example}
\label{comodexampBP}
    Suppose $(A,\Gamma)=(\BP_*,\BP_*\BP)$ at a prime $p$. Note that $$\Stable(\BP_*\BP)^{\cell}=\Stable(\BP_*\BP)$$ by \cite[Cor. 6.7]{Hov04}. By a classical calculation,
    $$\Ext_{\BP_*\BP}^{t,t}(\BP_*,\BP_*)=\begin{cases}
        \bZ_{(p)}, & t=0 \\
        0, & \mathrm{otherwise}.
    \end{cases}$$
    Hence we have that$$\Stable(\BP_*\BP)^{\heartsuit}=(\bZ_{(p)}\text{-}\mathrm{mod})^{\mathrm{gr}}.$$ 
\end{example}

\begin{example}
\label{comodexampAsyn}
    Suppose $(A,\Gamma)=(\bF_p[\tau],\cA_{\starstar}^\BP)$ is the $\BP$-synthetic dual Steenrod algebra at a prime $p$. This is isomorphic to the $\bC$-motivic dual Steenrod algebra $\cA_{\starstar}^{\mathrm{mot}}$ up to a doubling of the weight \cite{Pst22}. Then we have
    \begin{equation*}
        \Ext_{\cA_{\starstar}^\BP}^{t,t,*}(\bF_p[\tau],\bF_p[\tau])=\bF_p[\tau]\{h_0^t\}
    \end{equation*}
    with $t\geq 0$ and $h_0\in\Ext_{\cA_{\starstar}^\BP}^{1,1,0}$, and the module action corresponds to multiplication by powers of $\tau$ and $h_0$. Hence    $$(\Stable(\cA_{\starstar}^\BP)^{\cell})^{\heartsuit}=\bF_p[\tau,h_0]\text{-}\mathrm{mod},$$
    where we consider $\bF_p[\tau,h_0]$ as a bigraded polynomial ring with $\vert\tau\vert=(0,-1)$ and $\vert h_0\vert=(1,0)$.
\end{example}

\begin{example}
\label{comodexampBPsyn}
Suppose $(A,\Gamma)=(\nubp_\starstar,\nubpbp)$ from Section~\ref{BPsynanalogsection} at a prime $p$. By Theorem~\ref{lambdabockstein} and the fact that there are no algebraic Novikov spectral sequence differentials at stem 0, we have that
\begin{equation*}
    \Ext_{\nubpbp}^{t,t,*}(\nubp_\starstar,\nubp_\starstar)=\begin{cases}\bZ_{(p)}[\lambda], & t=0 \\
        0, & \mathrm{otherwise}.
    \end{cases}
\end{equation*}
Hence we have that
$$
(\Stable(\nubpbp)^{\cell})^{\heartsuit}=\bZ_{(p)}[\lambda]\text{-}\mathrm{mod},
$$
where we consider $\bZ_{(p)}[\lambda]$ as a bigraded polynomial ring with $\vert\lambda\vert=(0,-1)$.
\end{example}

\begin{example}
\label{BPsynexample}
    Consider the category $(\Syn_{\BP}^{\cell})_\tau^\wedge=(\Syn_{\BP})_\tau^\wedge$ of $\tau$-complete $\BP$-synthetic spectra with parameter $\tau$. By the classical calculations of $\Ext_{\BP_*\BP}^{t=s}=\bZ_{(p)}$ and $\pi_0\bS=\bZ$, we see that
    \begin{align*}
        \pi_{0,*}(\bS_{\BP})_\tau^{\wedge}&\cong\bZ_{(p)}[\tau].
    \end{align*}
    Hence, $((\Syn_{\BP})_\tau^\wedge)^{\heartsuit}=\bZ_{(p)}[\tau]\text{-mod}.$
\end{example}

\begin{remark}
    Interestingly enough, the unit $\bS_{\BP}\in\Syn_{\BP}$ is not connective (i.e. not in $(\Syn_\BP)_{\geq 0}$). Via the long exact sequence in homotopy applied to the fracture square as in \ref{fracture_diagram} with $X=\bS_\BP$ and $\lambda$ replaced with $\tau$, one can calculate that
    \begin{equation*}
        \pi_{-1,w}\bS_{\BP}\cong\begin{cases}
            0, & w\leq 0 \\
            \bZ_{(p)}/\bZ, & w > 0 
        \end{cases}
    \end{equation*}
    and $\tau:\pi_{-1,w+1}\bS_{\BP}\to\pi_{-1,w}\bS_{\BP}$ is an isomorphism for $w\geq 1$. In particular, each generator of $\pi_{-1,w}$ for $w\geq 1$ is infinitely $\tau$-divisible. We thank Robert Burklund for pointing this out to us. 
\end{remark}

\begin{remark}
    The unit $\bS_{\BP}\in\Syn_{\BP}$ does become connective if you instead consider the $\tau$-completion $(\bS_{\BP})_\tau^{\wedge}\simeq (\bS_{\BP})_{(p)}$, as in Example~\ref{BPsynexample}, or the $p$-completion $(\bS_\BP)_p^{\wedge}$. We thank William Balderrama for pointing this out to us. 
\end{remark}

\begin{example}
\label{Fpsynexample}
    Consider the category $\Syn_{\bF_p}^{\cell}=\Syn_{\bF_p}$ of $\bF_p$-synthetic spectra with parameter $\lambda$. An application of \cite[Thm. 9.19]{BHS23} together with the calculations $\Ext_{\cA_{*}}^{t=s}=\bF_p[h_0]$ and $\pi_0\bS=\bZ$ gets us
    \begin{align*}
        \pi_{0,*}(\bS_{\bF_p})_\lambda^{\wedge}&\cong\bZ_{p}^{\wedge}[\lambda,h]/(\lambda h=p), \\
        \pi_{0,*}\bS_{\bF_p}[\lambda^{-1}]&\cong\bZ[\lambda^{\pm 1}], \\
        \pi_{0,*}(\bS_{\bF_p})_\lambda^{\wedge}[\lambda^{-1}]&\cong\bZ_{p}^{\wedge}[\lambda^{\pm 1}],
    \end{align*}
    where $h$ maps to $h_0$ modulo $\lambda$. Again, there are no contributions from $\pi_{1,*}$ in the long exact sequence of homotopy groups associated with the fracture square. Hence    $$\pi_{0,*}\bS_{\bF_p}\cong\bZ[\lambda,h]/(\lambda h=p)$$ and $$\Syn_{\bF_p}^{\heartsuit}=\bZ[\lambda,h]/(\lambda h=p)\text{-mod}.$$
\end{example}

\begin{remark}
    One can check that for $\bS_{\bF_p}\in\Syn_{\bF_p}$, $\pi_{k,*}\bS_{\bF_p}=0$ for $k<0$ so that, in fact, $\bS_{\bF_p}$ is connective and Example~\ref{Fpsynexample} does define a $t$-structure. See Lemma~\ref{Fpsyntstructurelemma}.
\end{remark}

\section{Deformations of Stable Comodule Categories}\label{defcomod}

In this section, we explore deformations of $\Stable(\cA_{*})$ and $\Stable(\BP_*\BP)$ associated to the Cartan-Eilenberg spectral sequence for the extension $\cP_*\to \cA_{*}\to \cE_*$ and the algebraic Novikov spectral sequence, respectively, and identify them with categories of stable comodules over $\cA_{*}^\BP$ and $\nubpbp$. Before doing so, we briefly review the two spectral sequences.

\bigskip

Let $(A,\Gamma)$ be a Hopf algebroid and let $\Comod(\Gamma)$ denote its 1-category of comodules. A comodule $C$ is said to be an \textit{extended comodule} if it is isomorphic to one of the form $\Gamma\otimes_A M$ for some $A$-module $M$. A comodule is said to be a \textit{relative injective} if it is a summand of an extended comodule \cite[Def. A1.2.7]{Rav86}. Given a comodule $C$, a relative injective resolution of $C$ is a long exact sequence
\[
0\to C\to C^0 \to C^1 \to C^2 \to ... 
\]
which is split-exact over $A$ and such that all of the $C_i$ are relative injectives. Such a resolution always exists: one can take the cobar complex
\[
C^i = \Gamma \otimes \bar \Gamma^{\otimes i}\otimes C
\]
where $\bar \Gamma$ is the kernel of the augmentation $\Gamma \to A$. Given such a $C^*$ resolving $C$, the inclusion $\Comod(\Gamma)\hookrightarrow \Stable(\Gamma)$ of the 1-category of $\Gamma$-comodules into $\Stable(\Gamma)$ as complexes concentrated in degree $0$ induces an equivalence $C\simeq C^*$. As a result, given any filtration
\[
C^*\leftarrow \Fil^1C^{*} \leftarrow \Fil^2C^{*} \leftarrow \Fil^3C^{*} \leftarrow...
\]
of the resolution by chain complexes $C_*^j$ yields a filtration of $C$ in $\Stable(\Gamma)$. As is explained in Section 9 of \cite{GWX21}, when $(A,\Gamma)=(\BP_*,\BP_*\BP)$ we may form the algebraic Novikov spectral sequence as the spectral sequence associated to the filtration by powers of the ideal $I=(p,v_1,v_2,...)$. Explicitly, we put:
\[
(\Fil^jC^*)^i = I^{j-i}C^i
\]
where $C^*$ is the cobar complex resolution of $\BP_*$. The authors of \cite{GWX21} proceed to prove that the $\hfp^\BP$-Adams spectral sequence for the cofiber of $\tau$ is isomorphic to the algebraic Novikov spectral sequence, which is isomorphic to the $\h_*\BP$-Adams spectral sequence in $\Stable(\BP_*\BP)$, by demonstrating that this tower of objects satisfies an axiomatic definition of an Adams resolution in $\Stable(\BP_*\BP)$\cite[Def. 9.1]{GWX21}.

\bigskip

A similar story can be told for the classical CESS. The CESS may be similarly constructed by taking the cobar resolution $C^*$ of $\hfp$ over $\cA_{*}$ and filtering it by:
\[
(\Fil^j C^*)^i = I^{j-i}C^i
\]
where
\begin{equation}
\label{Steenrodideal}
 I=\begin{cases}
    (\xi_1^2,\xi_2^2,\xi_3^2,...) & p=2\\
    (\xi_1, \xi_2, \xi_3,...) & p>2
\end{cases}   
\end{equation}
generates the positive degree part of the even dual Steenrod algebra. In the same way, we get a resolution in $\Stable(\cA_{*})$ for $\hfp$. From this point of view, both spectral sequences have a unified construction coming from the Thom reduction $\BP\to \hfp$. Let $\cP_*:=\h_*\BP$. Taking $\BP$-homology, we get the map $\BP_*\BP\to \cP_*$. For the algebraic Novikov, the ideal $I=(p,v_1,v_2,\ldots)$ above is the $\BP_*$-module kernel of this morphism of Hopf algebras. Taking $\hfp$-homology, we get the map $\cP_*\to \cA_{*}$. For the Cartan-Eilenberg one instead looks at the extension of Hopf algebras:
\[
\cP_*\to \cA_{*} \to \cE_*
\]
where $\cE_*$ is exterior on either the $\xi_i$ ($p=2$) or the $\tau_i$ ($p>2$). Then the ideal $I$ in (\ref{Steenrodideal}) is the $\bF_p$-module kernel of the second map.

\bigskip

We now proceed to study the promised deformations. In Section~\ref{filtered}, we recall how one can associate to an $R$-Adams spectral sequence in a presentably symmetric monoidal stable $\infty$-category $\cC$ with $t$-structure a deformation $\Adams_R(\cC)$, closely related to synthetic spectra when $\cC=\Sp$. We will make use of the results and notation of the appendix freely in this section. For a category $\cC$ with $p$-completion $(-)_p^{\wedge}$, we copy the notation of \cite{BHS20} and write $\cC_{ip}:=\Mod(\cC;\one_p^{\wedge})$. 

\bigskip

The main results of this section are the identification of two deformations:

\begin{theorem}[Cartan-Eilenberg Deformation]\label{cedef}
    There is a $\Stable(\cA_{*})$-linear equivalence of presentably symmetric monoidal stable $\infty$-categories
    \[
    \Stable(\cA_{\starstar}^\BP)^{\cell} \simeq \Adams_{\cP_*}(\Stable(\cA_{*}))
    \]
    with parameter $\tau$ whose generic fiber recovers $\Stable(\cA_{*})$ and whose special fiber is given by:
    \[
    \Mod(\Stable(\cA_{\starstar}^\BP)^{\cell}; \bF_p[\tau]/\tau)\simeq \Mod(\Stable(\cA_*)^\Gr; \Gr_*\Gamma_{\cP_*}\bF_p)\\
    \]
    where the graded object on the right has homotopy groups isomorphic to $\text{E}^{\text{CESS}}_2$. 
\end{theorem}

\begin{theorem}[Algebraic Novikov Deformation]\label{andef}
    There is a $\Stable(\BP_{*}\BP)$-linear equivalence of presentably symmetric monoidal stable $\infty$-categories
    \[
    \Stable(\nubpbp)_{ip}^{\cell} \simeq \Adams_{\cP_*}(\Stable(\BP_*\BP))
    \]
     with parameter $\lambda$ whose generic fiber recovers $\Stable(\BP_*\BP)_{ip}$ and whose special fiber is given by:
     \[
     \Mod(\Stable(\nubpbp)_{ip}^{\cell}; (\nubp_\starstar)_p^{\wedge}/\lambda)\simeq \Mod(\Stable(\BP_*\BP)^\Gr; \Gr_*\Gamma_{\cP_*}(\BP_*))
     \]
     where the graded object on the right has homotopy groups isomorphic to $\text{E}_2^{\text{aNSS}}$.
\end{theorem}

\begin{remark}
    After proving the categorical equivalences above, the remainder of the claims in the above theorems will follow immediately from \cite[Prop. C.2]{BHS20}. 
\end{remark}

The above result implicitly requires a $t$-structure to define the Adams deformations, which we defined in Section~\ref{tstructuresection}. The proof of both results occurs in two stages, using the technology of \cite{BHS20}. We first demonstrate that both categories are examples of 1-parameter deformations pairs \cite{BHS20}. Before the proofs, we explain the following corollaries:

\begin{corollary}\label{qcecor}
    In $\Stable(\cA_{\starstar}^\BP)$, there is a cofiber sequence
    \[
    \Gamma_{\cP_*}\hfp \xrightarrow{\tau} \Gamma_{\cP_*}\hfp \xrightarrow{i} \Gamma_{\cP_*}\hfp/\tau \xrightarrow{q_1} \Sigma\Gamma_{\cP_*}\hfp
    \]
    such that the map $(q_1)_\starstar$ on trigraded homotopy groups induces a map from the $\E_2$-page of the Cartan-Eilenberg spectral sequence to the $\E_2$-page of the $\hfp^\BP$-Adams spectral sequence. If $x\in\pi_{-*,*,*}\Gamma_{\cP_*}\hfp/\tau$ survives until the $\E_r$-page, then $x$ supports a nonzero differential $d^{\mathrm{CE}}_r(x)=y$ if and only if $(q_1)_\starstar(x)$ is detected by $-\tau^{r-1}y$.
\end{corollary}

\begin{corollary}\label{qancor}
    In the sequence
    \[
    \Gamma_{\cP_*}\BP_* \xrightarrow{\lambda} \Gamma_{\cP_*}\BP_* \xrightarrow{i} \Gamma_{\cP_*}\BP_*/\lambda \xrightarrow{q_2} \Sigma\Gamma_{\cP_*}\BP_*
    \]
    the map $(q_2)_\starstar$ on trigraded homotopy groups induces a map from the $\E_2$-page of the algebraic Novikov spectral sequence to the $\E_2$-page of the $\BP^\hfp$-Adams spectral sequence. If $x\in\pi_{-*,*,*}\Gamma_{\cP_*}\BP_*/\tau$ survives until the $\E_r$-page, then $x$ supports a nonzero differential $d^{\mathrm{aN}}_r(x)=y$ if and only if $(q_2)_\starstar(x)$ is detected by $-\lambda^{r-1}y$.
\end{corollary}

\begin{proof}[Proof of Corollaries \ref{qcecor} and
\ref{qancor}]
The categorical equivalences of Theorems~\ref{cedef} and \ref{andef} identify the filtered objects and spectral sequences associated with the objects in the cofiber sequence. The result then follows from Lemma \ref{qlemma}.
\end{proof}

The above results are used computationally in Section~\ref{smodlambda}. By work of \cite{IWX20}, \cite{aNSS}, the map $(q_1)_\starstar$ and the aNSS differentials for the sphere were determined up to stem 110 for $p=2$. This allows us to determine the map $(q_2)_\starstar$ and the CESS differentials for the sphere up to stem 45.

\bigskip

We now set about proving the theorems, relying heavily on the machinery discussed in Section \ref{recognitionsection} originally due to \cite{BHS20}. Our first goal is to prove that the pairs
\begin{align*}
  &(\Stable(\cA_*),\Stable(\cA_{\starstar}^\BP)^{\cell}) \\
 (\Stable&(\BP_*\BP),\Stable(\nubpbp)^{\cell})
\end{align*}
are deformation pairs. Define the following morphisms of Hopf algebroids:

\begin{align*}
 &\iota_{\hfp}: \cA_* \to \cA_{\starstar}^\BP\\
    &\xi_i \mapsto \tau^{(2^{k+1}-2)}\tau_{i-1} && && (p=2) \\&\xi_i \mapsto \tau^{(2p^k-2)}b_i  &&  \tau_i \mapsto \tau^{(2p^k-2)}\tau_i   &&(p>2)\\
     \\
     &\iota_{\BP}:\BP_{*}\BP \to \BP_{\starstar}\BP^\hfp\\
    &v_i \mapsto \tau^{2p^i-1}v_i && t_i \mapsto \tau^{2p^i-2}t_i && 
\end{align*}

constructed so that we may upgrade them to bigraded homomorphisms after giving the singly-graded objects bigradings of the form $n\mapsto (n,0)$. These induce symmetric monoidal left adjoints of stable comodule categories
\begin{align*}
   c_{\hfp}&: \Stable(\cA_*)\to \Stable(\cA_{\starstar}^\BP)^{\cell} \\
   c_{\BP}&: \Stable(\BP_*\BP)\to \Stable(\BP_\starstar\BP^\hfp)^{\cell}
\end{align*}
via  \cite[Prop. 2.2.1, Prop. 5.3.1]{Hov04}. 

\begin{remark}
    These particular morphisms of Hopf algebroids were chosen so that $c_{\bF_p}(\Sigma^{s,t}\bF_p)=\Sigma^{s,t,0}\bF_p[\tau]$ and $c_{\BP}(\Sigma^{s,t}\BP_*)=\Sigma^{s,t,0}\nubp_{\starstar}$.
\end{remark}

The functors $\Re$ can be constructed straightforwardly by inverting $\lambda,\tau$ on the appropriate synthetic Hopf algebroids and taking the sub-comodules concentrated in second degree $0$, whence the functors $\Re_{\hfp}$ and $\Re_{\BP}$. That the functors $c_{\hfp},c_{\BP}$ are sections of the realizations follows from the fact that all of these adjoints are determined by their actions on graded comodules which can be computed explicitly.
 \newline
 
 We will write $i_{\hfp}$ and $i_{\BP}$ for the homomorphisms out of $\bZ$ which send an integer $n$ to the invertible graded comodules $\Sigma^{0,0,n}\hfp[\tau]$ and $\Sigma^{0,0,n}\nubp_\starstar$ respectively. It is easy to see that these realize to the appropriate monoidal units $\hfp$ and $\BP_*$. It remains to show the final two axioms of the deformation pair. Here we will solve the problem of generation by passing to cellular subcategories. 

\begin{remark}
     Note that $\Stable(\nubpbp)^{\cell}$ and $\Stable(\cA_{\starstar}^\BP)^{\cell}$ are exactly generated under colimits by the compact dualizable objects
    \[
    \{\Sigma^{l,m,0}\one\otimes i(n)\mid l,m,n\in \bZ\}
    \]
    where $i=i_{\BP}$ or $i_{\hfp}$ respectively, by definition, so that we have the required generation statements in the definition of a deformation pair.
\end{remark}

\begin{remark}
    It seems plausible to us that, as with their more classical analogs $\Stable(\cA_*)$ and $\Stable(\BP_*\BP)$, the synthetic stable comodules categories are already cellular. We do not pursue this question here.
\end{remark}

\begin{lemma}
The realization functor induces an equivalence:
\[
\Map(i(n),i(m))\to \Map(\one, \one)
\]
whenever $n\leq m$, $i=i_{\BP}$ or $i_{\hfp}$, and $\one$ is either $\hfp$ or $\BP_*$
\end{lemma}

\begin{proof}
    The objects $i(n)$ live in the heart, the inclusion of which is fully faithful, so that the claim can be checked in the 1-category of comodules.
\end{proof}

So far we have confirmed the following:

\begin{corollary}
    The pairs of categories
    \begin{align*}
  &(\Stable(\cA_*),\Stable(\cA_{\starstar}^\BP)^{\cell}) \\
 (\Stable&(\BP_*\BP),\Stable(\nubpbp)^{\cell})
\end{align*} are deformation pairs.
\end{corollary}

It now remains to compute the functor $i_*$ described above. Our approach is similar to the computation \cite[Example C.22]{BHS20}. As in loc. cit., we introduce a $p$-completion in the case of $\BP_*\BP$. In fact, a completion is used for both results, and the $\cA$ case happens to already have a complete unit. The point is that $i_*$ preserves limits, and so we can commute past it the totalization involved in a nilpotent completion. This simplifies the computation as we instead work with a cosimplicial object in terms of $\cP$.

\bigskip

For a commutative ring object $R$, let $R^{\bullet +1}$ denote the cosimplicial object with $R^{n+1}=R^{\otimes (n+1)}$ and coface maps induced by the unit $\one\to R$. We first identify two nilpotent completions.

\begin{proposition}
\label{synAnilpotentcomplete}
Consider the unit $\bF_p[\tau]\in\Stable(\cA_{\starstar}^\BP)$ and the commutative ring object $$\cP_*[\tau]=\pi_{*,*}\nu_{\BP}(\bF_p\otimes\BP)\in\Stable(\cA_{\starstar}^\BP).$$ There is an equivalence $$\bF_p[\tau]\xrightarrow{\simeq}(\bF_p[\tau])_{\cP_*[\tau]}^{\wedge}\simeq \Tot(\cP_*[\tau]^{\bullet +1}).$$      
\end{proposition}

\begin{proof}
This is Proposition~\ref{synAcompletiontheorem} for $X=\bF_p[\tau]$.
\end{proof}

\begin{proposition}
   \label{synBPnilpotentcomplete}
    Consider the unit $\nubp_{\starstar}\in\Stable(\nubpbp)$ and the commutative ring object $$\cP_*[\lambda]=\pi_{*,*}\nu_{\bF_p}(\bF_p\otimes\BP)\in\Stable(\nubpbp).$$
    The map $$\BP_{\starstar}^{\bF_p}\to(\BP_{\starstar}^{\bF_p})_{\cP_*[\lambda]}^{\wedge}\simeq\Tot(\cP_*[\lambda]^{\bullet +1})$$ is a $p$-completion.
\end{proposition}

\begin{proof}
This is Corollary~\ref{synBPunitcompletion}.  
\end{proof}

For the proofs of Theorem~\ref{cedef} and Theorem~\ref{andef}, we will need some computations of particular $\Ext$ groups. As a reminder, the algebra structure of the $\BP$-synthetic dual Steenrod algebra $\cA_{\starstar}^\BP$ \cite[Section 6.2]{Pst22} at a prime $p$ is as follows:
\begin{equation*}
    \cA_{\starstar}^\BP\cong\begin{cases}
        \bF_2[\tau, b_1,b_2,\ldots,\tau_0,\tau_1,\ldots]/(\tau_i^2=\tau^2b_{i+1}), & p=2, \\
        \bF_p[\tau,b_1,b_2,\ldots]\otimes_{\bF_p}\Lambda_{\bF_p}(\tau_0,\tau_1,\ldots), & p>2,
    \end{cases}
\end{equation*}
where $\degree{b_k}=(2p^k-2,2p^k-2)$ and $\degree{\tau_k}=(2p^k-1,2p^k-2)$.
Note that there is a conormal extension of bigraded Hopf algebras
$$
\cP_*[\tau]\to\cA_{\starstar}^\BP\to \cE_*[\tau]
$$
over $\bF_p[\tau]$ with algebra structures $$\cP_*[\tau]\cong\bF_p[\tau, b_1,b_2,\ldots]$$ for all primes $p$ and
\begin{equation*}
\cE_*[\tau]\cong\begin{cases}
\bF_2[\tau, \tau_0,\tau_1,\ldots]/(\tau_i^2), & p=2 \\
\bF_p[\tau]\otimes_{\bF_p}\Lambda_{\bF_p}(\tau_0,\tau_1,\ldots), & p>2.
\end{cases}    
\end{equation*}
In particular, $\cA_{\starstar}^\BP\square_{\cE_*[\tau]}\bF_p[\tau]\cong \cP_*[\tau].$ Then we have the following lemma:

\begin{lemma}
\label{synAextlemma}
For $m\geq 0$, there is an isomorphism of tri-graded commutative $\bF_p[\tau]$-algebras
$$
\Ext^{*,*,*}_{\cA_{\starstar}^\BP}(\bF_p[\tau],\cP_*[\tau]^{\otimes (m+1)})\cong \bF_p[\tau,a_0,a_1,\ldots]\otimes_{\bF_p[\tau]}(\bF_p[\tau,b_1,b_2,\ldots])^{\otimes m} 
$$
where $\degree{\tau}=(0,0,-1)$, $\degree{a_i}=(1,2p^i-1,2p^i-2)$, and $\degree{b_k}=(0,2p^k-2,2p^k-2)$.
\end{lemma}

\begin{proof}
This follows from a sequence of isomorphisms:
\begin{align*}
    \Ext^{*,*,*}_{\cA_{\starstar}^\BP}(\bF_p[\tau],\cP_*[\tau]^{\otimes (m+1)}) &\cong \Ext^{*,*,*}_{\cA_{\starstar}^\BP}(\bF_p[\tau],(\cA_{\starstar}^\BP\square_{\cE_*[\tau]}\bF_p[\tau])\otimes \cP_*[\tau]^{\otimes m}) \\
    &\cong \Ext^{*,*,*}_{\cE_*[\tau]}(\bF_p[\tau],\cP_*[\tau]^{\otimes m}) \\
    &\cong \Ext^{*,*,*}_{\cE_*[\tau]}(\bF_p[\tau],\bF_p[\tau])\otimes_{\bF_p[\tau]} \cP_*[\tau]^{\otimes m} \\
    &\cong \Ext^{*,*}_{\cE_*}(\bF_p,\bF_p)[\tau]\otimes_{\bF_p[\tau]} \cP_*[\tau]^{\otimes m}
\end{align*}
where the second isomorphism follows from change-of-rings, the third isomorphism follows because $\cP_*[\tau]^{\otimes m}$ is a trivial $\cE_*[\tau]$-comodule, and the last isomorphism follows by doing another change-of-ring isomorphism along the map $(\bF_p[\tau],\cE_*[\tau])\to (\bF_p,\cE_*)$ which kills $\tau$ and using the fact that $\bF_p[\tau]=\bF_p\otimes\bZ[\tau]$ is a trivial $\cE_*$-comodule. Note we consider $(\bF_p,\cE_*)$ to be bigraded by letting the weight of $\tau_i\in \cE_*$ coincide with its weight in $\cE_*[\tau]$. The result then follows from the classical calculation
$$\Ext_{\cE_*}^{*,*}(\bF_p,\bF_p)\cong\bF_p[a_0,a_1\ldots]$$
with degrees as above as in the statement of the lemma.
\end{proof}

Now we will prove an analogous calculation over the Hopf algebroid $(\nubp_{\starstar},\nubpbp)$ from Section~\ref{BPsynanalogsection}. Consider the $\nubpbp$-comodule algebra $\cP_*[\lambda]:=\nu_{\bF_p}(\bF_p\otimes\BP)_{\starstar}\cong\bF_p[\lambda,t_1,t_2,\ldots]$:

\begin{lemma}
\label{synBPextlemma}
For $m\geq 0$, there is an isomorphism of tri-graded commutative $\bF_p[\lambda]$-algebras
$$\Ext^{*,*,*}_{\BP_{\starstar}\BP^{\bF_p}}(\BP_{\starstar}^{\bF_p},\cP_*[\lambda]^{\otimes (m+1)})\cong (\bF_p[\lambda,t_1,t_2,\ldots])^{\otimes m},$$
where $\degree{\lambda}=(0,0,-1)$ and $\degree{t_k}=(0,2p^k-2,2p^k-2)$.
\end{lemma}

\begin{proof}
We use the synthetic aNSS we constructed in Section~\ref{synalgnsssection} to compute this Ext group. The $\E_2$-page of the synthetic aNSS for this example is
$$
\E_2^{f,t,w,u}=\Ext^{f,t,w,u}_{\Gr_*(\BP_*\BP)[\lambda]}(\Gr_*(\BP_*)[\lambda],\Gr_*(\cP_*[\lambda]^{\otimes (m+1)})).
$$
We claim that for synthetic algebraic Novikov filtration $u>0$, this Ext group vanishes. This is because for $\cP_*[\lambda]$, the ideal $J=(h,v_1,\ldots)\subset \BP_{\starstar}^{\bF_p}$ acts by zero. Hence $\Gr_*(\cP_*[\lambda]^{\otimes (m+1)})\cong \cP_*[\lambda]^{\otimes (m+1)}$ and the synthetic aNSS collapses at $\E_2$, giving us
\begin{equation*}
    \begin{split}
        \Ext^{*,*,*}_{\BP_{\starstar}\BP^{\bF_p}}(\BP_{\starstar}^{\bF_p},\cP_*[\lambda]^{\otimes (m+1)})&\cong \Ext^{*,*,*,0}_{\Gr_*(\BP_*\BP)[\lambda]}(\Gr_*(\BP_*)[\lambda],\cP_*[\lambda]^{\otimes (m+1)}) \\
        &\cong \Ext^{*,*,*}_{\cP_*[\lambda]}(\bF_p[\lambda],\cP_*[\lambda]^{\otimes (m+1)}) \\
    \end{split}
\end{equation*}
where the second isomorphism follows by change-of-ring along the morphism of Hopf algebroids $(\Gr_*(\BP_*)[\lambda],\Gr_*(\BP_*\BP)[\lambda])\to (\bF_p[\lambda],\cP_*[\lambda])$ which arises from killing off $a_i\in\Gr_*(\BP_*)$. These Ext groups vanish in positive filtration. Hence, as a graded $\bF_p[\lambda]$-algebra
\begin{equation*}
    \begin{split}
        \Ext^{*,*,*}_{\BP_{\starstar}\BP^{\bF_p}}(\BP_{\starstar}^{\bF_p},\cP_*[\lambda]^{\otimes (m+1)})&\cong\Hom_{\cP_*[\lambda]}(\bF_p[\lambda],\cP_*[\lambda]^{\otimes (m+1)}) \\
        &\cong \Hom_{\bF_p[\lambda]}(\bF_p[\lambda],\cP_*[\lambda]^{\otimes m}) \\
        &\cong \cP_*[\lambda]^{\otimes m},
    \end{split}
\end{equation*}
and the result follows.
\end{proof}

\begin{notation}
    We write $\Map^{\cC}_{\cD}(X,Y)$ for a $\cC$-enriched mapping object between objects in $\cD$. 
\end{notation}

\begin{proof}[Proof of Theorem \ref{cedef}]
    Let $i_*:\Stable(\cA_{\starstar}^\BP)^{\cell}\to \Stable(\cA_*)^\Fil$ be as guaranteed by \cite[Prop. C.20]{BHS20}. For $Y\in \Stable(\cA_*)$, let $\Sigma^{n,t,w}Y$ denote the filtered object
    \[
    ...\to 0 \to 0 \to \Sigma^{n,t}Y \xrightarrow{\id} \Sigma^{n,t}Y \to ...
    \]
    which is $0$ in filtered pieces $> w$. The left adjoint $i^*$ of $i_*$ satisfies
    \[
    i^*\Sigma^{n,t,w}Y \simeq \Sigma^{n,t,w}i^*Y.
    \]
    Let $c^!$ denote the right adjoint to $c_{\hfp}$. There is an identification
    \[
    c^!\Map^{\Stable(\cA_{\starstar}^\BP)}_{\Stable(\cA_{\starstar}^\BP)}(X,Y)\simeq \Map^{\Stable(\cA_*)}_{\Stable(\cA_{\starstar}^\BP)}(X,Y).
    \]
    For $X\in \Stable(\cA_{\starstar}^\BP)$ the $n$th filtered piece may be extracted as
    \[
    (i_*X)_n\simeq \Map^{\Stable(\cA_*)}_{\Stable(\cA_{\starstar}^\BP)}(i^*\Sigma^{0,0,n}\hfp[\tau],X)\simeq \Map^{\Stable(\cA_*)}_{\Stable(\cA_{\starstar}^\BP)}(\Sigma^{0,0,n}\hfp[\tau],X)
    \]
    with connecting maps given by $\tau$. Because $i_*$ preserves limits, we will use Proposition~\ref{synAnilpotentcomplete} to reduce to:
    \[
    i_*\hfp[\tau] \simeq i_*(\Tot (\cP_*[\tau]^{\bullet +1}))\simeq \Tot(i_*(\cP_*[\tau]^{\bullet +1}))
    \]
    Let $\tau_{\geq k}$ denote the Whitehead truncation of the $t$-structure in Example~\ref{comodexampA}. For $m\geq 0$, we have an identification
    \[
    (i_*(\cP_*[\tau]^{m+1}))_n \simeq \Map^{\Stable(\cA_*)}_{\Stable(\cA_{\starstar}^\BP)}(\Sigma^{0,0,n}\hfp[\tau],\cP_*[\tau]^{m+1})
    \]
    We claim that $(i_*(\cP_*[\tau]^{m+1}))_n\in \Stable(\cA_*)_{\geq k}$ so that there is a factorization of the natural maps $(i_*\cP_*[\tau]^{m+1})_n \to Y(\cP_*^{m+1})_n=\cP_*^{m+1}$ in $\Stable(\cA_*)$ through the Postnikov filtration $\tau_{\geq n}\cP_*^{m+1}$. This may be seen through an Ext computation:
    \begin{equation*}
        \begin{split}
            \pi_{-f,t}(i_*(\cP_*[\tau]^{m+1})_n)&=[\Sigma^{-f,t}\mathbb{F}_p,c^!(\Map^{\Stable(\cA_{\starstar}^\BP)}_{\Stable(\cA_{\starstar}^\BP)}(\Sigma^{0,0,n}\bF_p[\tau],\cP_*[\tau]^{m+1}))]_{\Stable(\cA_*)} \\
            &\cong [\Sigma^{-f,t,0}\mathbb{F}_p[\tau],\Map^{\Stable(\cA_{\starstar}^\BP)}_{\Stable(\cA_{\starstar}^\BP)}(\Sigma^{0,0,n}\bF_p[\tau],\cP_*[\tau]^{m+1})]_{\Stable(\cA_{\starstar}^\BP)} \\
            &\cong \Ext^{f,t,n}_{\cA_{\starstar}^\BP}(\bF_p[\tau],\cP_*[\tau]^{\otimes (m+1)}).
        \end{split}
    \end{equation*}

By Lemma~\ref{synAextlemma}, we get that
$$
\pi_{-*,*}(i_*(\cP_*[\tau]^{m+1}))\cong \Ext^{*,*}_{\cE_*}(\bF_p,\bF_p)[\tau]\otimes_{\bF_p[\tau]} \cP_*[\tau]^{\otimes m},
$$
which happens to coincide with the homotopy groups
\begin{align*}
    \pi_{-*,*}(\tau_{\geq *}\cP_*^{m+1})&\cong\pi_{-*,*}(\cP_*^{m+1})[\tau]\cong\Ext^{*,*}_{\cA_*}(\bF_p,\cP_*^{\otimes (m+1)})[\tau] \\
    &\cong \Ext^{*,*}_{\cE_*}(\bF_p,\bF_p)[\tau]\otimes_{\bF_p[\tau]} \cP_*[\tau]^{\otimes m},
\end{align*}
whose generators live in Chow degree $(t-f)-n=0$. The homotopy groups $\pi_{-f,t}(i_*(\cP_*[\tau]^{m+1})_n)$ vanish for $t-f<n$ and this guarantees a factorization $$(i_*\cP_*[\tau]^{m+1})_* \to \tau_{\geq *}\cP_*^{m+1}$$ in $\Stable(\cA_*)^{\fil}$. Because $\Stable(\cA_*)=\Stable(\cA_*)^{\cell}$, it suffices to check on homotopy that this map is an equivalence. The homotopy groups in question are $\tau$-free, so we can check this after applying the realization functor $\mathrm{Re}^{\fil}:\Stable(\cA_*)^{\fil}\to\Stable(\cA_*)$. However, $(i_*\cP_*[\tau]^{m+1})_* \to \tau_{\geq *}\cP_*^{m+1}$ realizes to the identity $\cP_*^{m+1}\to\cP_*^{m+1}$, thanks to Lemma~\ref{leftcomplemma}. This clearly induces an isomorphism on homotopy. Hence we have equivalences of filtered objects
$$
i_*(\bF_p[\tau])\simeq \Tot(i_*(\cP_*[\tau]^{\bullet +1}))\simeq \Tot(\tau_{\geq *}(\cP_*^{\bullet +1})),
$$
and the result follows by applying \cite[Prop. C.20]{BHS20}.
\end{proof}

We now turn to the analogous result for $\Stable(\nubpbp)$.

\begin{proof}[Proof of Theorem \ref{andef}]
    Again we have a lax monoidal right adjoint 
    \[
    i_*:\Stable(\BP_{\starstar}\BP^{\bF_p})^{\cell}\to\Stable(\BP_*\BP)^{\fil}\
    \]
    and we wish to compute $(i_*(\BP_{\starstar}^{\bF_p}))_p^{\wedge}\simeq i_*((\BP_{\starstar}^{\bF_p})_p^{\wedge})$. Similar to before, we have an equivalence
    \[
    (i_*X)_n \simeq c^{!}(\Map^{\Stable(\BP_{\starstar}\BP^{\bF_p})^{\cell}}_{\Stable(\BP_{\starstar}\BP^{\bF_p})^{\cell}}(\Sigma^{0,0,n}\BP_{\starstar}^{\bF_p},X)).
    \]
    By Lemma~\ref{synBPnilpotentcomplete} we have equivalences
    \[
    i_*((\BP_{\starstar}^{\bF_p})_p^{\wedge})\simeq i_*(\Tot(\cP_*[\lambda]^{\bullet +1}))\simeq \Tot(i_*(\cP_*[\lambda]^{\bullet+1})).
    \]

To compute $\pi_{*,*}(i_*(\cP_*[\lambda]^{m +1}))$ for $m\geq 0$, we note that
\begin{equation*}
    \begin{split}
        \pi_{-f,t}&(i_*(\cP_*[\lambda]^{m +1})_n)\\
        &=[S^{-f,t},c^{!}(\Map^{\Stable(\BP_{\starstar}\BP^{\bF_p})^{\cell}}_{\Stable(\BP_{\starstar}\BP^{\bF_p})^{\cell}}(\Sigma^{0,0,n}\BP_{\starstar}^{\bF_p},\cP_*[\lambda]^{m+1}))]_{\Stable(\BP_*\BP)} \\
            &\cong [S^{-f,t,0},\Map^{\Stable(\BP_{\starstar}\BP^{\bF_p})^{\cell}}_{\Stable(\BP_{\starstar}\BP^{\bF_p})^{\cell}}(\Sigma^{0,0,n}\BP_{\starstar}^{\bF_p},\cP_*[\lambda]^{m+1})]_{\Stable(\BP_{\starstar}\BP^{\bF_p})^{\cell}} \\
            &\cong \Ext^{f,t,n}_{\BP_{\starstar}\BP^{\bF_p}}(\BP_{\starstar}^{\bF_p},\cP_*[\lambda]^{\otimes (m+1)}).
    \end{split}
\end{equation*}

By Lemma~\ref{synBPextlemma}, we get that
$$
\pi_{-*,*}(i_*(\cP_*[\lambda]^{m +1}))\cong \cP_*[\lambda]^{\otimes m},
$$
which coincides with the homotopy groups
\begin{align*}
\pi_{-*,*}(\tau_{\geq *}\cP_*^{m+1})&\cong\pi_{-*,*}(\cP_*^{m+1})[\lambda]\cong\Ext_{\BP_*\BP}^{*,*}(\BP_*,\cP_*^{\otimes (m+1)})[\lambda]  \\
&\cong \cP_*[\lambda]^{\otimes m}.
\end{align*}
Similar to the case of $\Stable(\cA_{\starstar}^\BP)^{\cell}$, this gives us a factorization
$$
i_*(\cP_*[\lambda]^{m +1})_n\to\tau_{\geq n}(\cP_*^{m +1})
$$ 
which we can check is an equivalence via homotopy groups since $\Stable(\BP_*\BP)=\Stable(\BP_*\BP)^{\cell}$. Again, this will be an equivalence since the homotopy is $\lambda$-free and, hence, we have an equivalence of filtered objects
$$
(i_*(\BP_{\starstar}^{\bF_p}))_{p}^{\wedge}\simeq\Tot(\tau_{\geq *}(\cP_*^{\bullet +1})).
$$
After applying \cite[Prop. C.20]{BHS20}, this finishes the proof.
\end{proof}

\numberwithin{theorem}{subsection}

\section{\texorpdfstring{The Synthetic Adams-Novikov Spectral Sequence for $\bS_{\mathbb{F}_2}/\lambda$}{The Synthetic Adams-Novikov Spectral Sequence for Cofiber of lambda}}\label{smodlambda}
In this section, we compute the Adams-Novikov spectral sequence for $\bS_{\mathbb{F}_2}/\lambda$ through the 45-stem. For this section, we use the shorthands
\begin{align*}
    \Ext_{\nubpbp}^{*,*,*}&:=\Ext_{\nubpbp}^{*,*,*}(\nubp_\starstar,\nubp_\starstar), \\
    \Ext_{\nubpbp/\lambda}^{*,*,*}&:=\Ext_{\nubpbp/\lambda}^{*,*,*}(\nubp_\starstar/\lambda,\nubp_\starstar/\lambda), \\
    \Ext_{\cA_{\starstar}^\BP}^{*,*,*}&:= \Ext_{\cA_{\starstar}^\BP}^{*,*,*}(\bF_p[\tau],\bF_p[\tau]),\\
    \Ext_{\cA_{\starstar}^\BP/\tau}^{*,*,*}&:= \Ext_{\cA_{\starstar}^\BP/\tau}^{*,*,*}(\bF_p,\bF_p).
\end{align*}
Because we use the same names for elements of $\E_r$ and $\pi_\starstar$ in our computations, we abuse notation when using Lemma~\ref{qlemma} to analyze the map $q:\bS_E/\tau\to\Sigma\bS_E$ for either $E=\BP$ or $\bF_2$. To see interactive charts of the synthetic Adams-Novikov spectral sequence for $\bS_{\bF_2}/\lambda$, we refer the reader to \cite{BJM25}

\subsection{The Synthetic Algebraic-Novikov Spectral Sequence}
Theorem~\ref{BPhomotopytheorem} says that
\begin{align*}
    &\nubpbptwo\cong \bZ_{(2)}[\lambda,h,v_1,\ldots,t_1\ldots]/(\lambda h=2), &&  \nubpbptwo/\lambda\cong \bF_2 [h,v_1,...,t_1,\ldots].
\end{align*}

Since $\nubp_{\starstar}$ is $\lambda$-torsion free, Proposition \ref{synANSSe2page} and a change-of-rings isomorphism gives us that the $\E_2$-page of the Adams-Novikov spectral sequence for $\bS_{\mathbb{F}_2}/\lambda$ is isomorphic to $\Ext_{\nubpbptwo/\lambda}^{*,*,*}$. Recall from Example \ref{cessexample} that the Adam-Novikov spectral sequence for $\bS_{\mathbb{F}_2}/\lambda$ is isomorphic to the Cartan-Eilenberg spectral sequence for the conormal extension
\begin{align}\label{cessextn}
    \bF_2[\zeta_1^2, \zeta_2^2, ...] \to \bF_2[\zeta_1, \zeta_2, ...] \to E(\zeta_1, \zeta_2, ...).
\end{align}
The $\E_2$-page of this Cartan-Eilenberg spectral sequence is isomorphic to the $\E_2$-page of the algebraic Novikov spectral sequence \cite{Rav86}. The $\E_2$-page and differentials for the algebraic Novikov spectral sequence can be computed by machine \cite{wang}.

\bigskip

While computer data makes it easy to compute $\Ext_{\nubpbptwo/\lambda}^{*, *, *}$ as a ring, it takes more work to compute $\Ext_{\nubpbptwo/\lambda}^{*, *, *}$ as a module over $\Ext_{\nubpbptwo}^{*, *, *}$. For this we must first know something about $\Ext_{\nubpbptwo}^{*, *, *}$. Proposition \ref{synANSSe2page} gives us that the $\E_2$-page of the Adams-Novikov spectral sequence for $\bS_{\bF_2}$ is isomorphic to $\Ext_{\nubpbptwo}^{*, *, *}$. 

\begin{proposition}\label{synE2}
The $\E_2$-page of the synthetic Adams–Novikov spectral sequence for $\bS_{\bF_2}$ is computed through the 45-stem.
\end{proposition}
\begin{proof}
     We compute $\Ext_{\nubpbptwo}^{*, *, *}$ with the synthetic algebraic Novikov spectral sequence. By Theorem \ref{lambdabockstein}, the differentials in the synthetic algebraic Novikov spectral sequence are fully determined by differentials in the classical algebraic Novikov spectral sequence. The classical algebraic Novikov spectral sequence has been computed by machine out to the 110-stem \cite{aNSS}. It remains to compute extensions in $\Ext_{\nubpbptwo}^{*, *, *}$ which are hidden on the synthetic algebraic Novikov $\E_\infty$-page. Many of these are hidden extensions in the classical algebraic Novikov spectral sequence and are recorded in \cite{IWX20}. The remaining hidden extensions are extensions between $\lambda$-torsion classes and are computed in Section \ref{ANSSsphere}.
\end{proof}

\subsection{Inclusion and Projection}
Since $\nubpbptwo$ is $\lambda$-torsion free, multiplication by $\lambda$ induces a short exact sequence
\begin{align*}\label{iq}
    0 \to \nubpbptwo \xrightarrow{\lambda} \nubpbptwo \to \nubpbptwo/\lambda \to 0
\end{align*}

which induces a long exact sequence

\begin{center}
\begin{equation}\label{les}
\begin{tikzcd}
  \cdots \rar & \Ext_{\nubpbptwo}^{*, *, * + 1} \rar{\lambda}
             \ar[draw=none]{d}[name=X, anchor=center]{}
    & \Ext_{\nubpbptwo}^{*, *, *} \rar{i_*} & \Ext_{\nubpbptwo/\lambda}^{*, *, *} \ar[rounded corners,
            to path={ -- ([xshift=2ex]\tikztostart.east)
                      |- (X.center) \tikztonodes
                      -| ([xshift=-2ex]\tikztotarget.west)
                      -- (\tikztotarget)}]{dll}[at end]{q_*} \\      
  &\Ext_{\nubpbptwo}^{* - 1, * + 1, * + 1} \rar{\lambda}
    & \cdots 
\end{tikzcd}
\end{equation}
\end{center}
on $\Ext$ groups.

\begin{notation}
Let $x\in\Ext_{\nubpbptwo}^{*, *, *}$. We denote an element of $\Ext_{\nubpbptwo/\lambda}^{*, *, *}$ by $x$ if $i_*(x) = x$. We denote an element of $\Ext_{\nubpbptwo}^{* + 1, * - 1, * - 1}(\bS/\lambda)$ by $\overline{x}$ if $q_*(\overline{x}) = x$.
\end{notation}

\begin{proposition}\label{iqvalues}
    The values of the maps 
    \begin{align}\label{iqalg}
        &i_*: \Ext_{\nubpbptwo}^{*, *, *} \to \Ext_{\nubpbptwo/\lambda}^{*, *, *} && q_*: \Ext_{\nubpbptwo/\lambda}^{*, *, *} \to \Ext_{\nubpbptwo}^{*-1, *+1, *+1}
    \end{align}
    are computed through the 45-stem.
\end{proposition}
\begin{proof}
Corollary \ref{qancor} gives us that if $d_r^{aN}(x) = y$ in the algebraic Novikov spectral sequence, then $q_*(x) = y$ on the synthetic Adams-Novikov $\E_2$-page. Algebraic Novikov differentials can be computed by machine \cite{wang} \cite{aNSS}. For example, there is an algebraic Novikov differential $d^{aN}_2(\alpha_{8/8}) = 2\beta_{4/4}$ so we have $q_*(\alpha_{8/8}) = h\beta_{4/4}$ on the synthetic Adams-Novikov $\E_2$-page. We record this by denoting $\alpha_{8/8}$ by $\overline{h\beta_{4/4}}$ when considered as an element of $\Ext_{\nubpbptwo/\lambda}^{*, *, *}$. All values of $q_*$ are determined similarly. The values of $i_*$ are then forced for degree reasons by inspecting the long exact sequence (\ref{les}).
\end{proof}

\subsection{The Cartan-Eilenberg Spectral Sequence}
In this section, we compute Adams-Novikov differentials for $\bS_{\bF_2}/\lambda$ through the 45-stem. 

\begin{proposition}\label{cedif}
    Differentials in the Adams-Novikov spectral sequence for $\bS_{\bF_2}/\lambda$ are computed through the 45-stem.
\end{proposition}
\begin{proof}
    Recall from Example \ref{cessexample} that the Adam-Novikov spectral sequence for $\bS_{\mathbb{F}_p}/\lambda$ is isomorphic to the Cartan-Eilenberg spectral sequence for the conormal extension (\ref{cessextn}).

By Corollary \ref{qcecor}, it suffices to understand the projection map
\begin{align}\label{Cmotq}
        q_*: \Ext_{\cA_{\starstar}^\BP/\tau}^{*, *, *} \to \Ext_{\cA_{\starstar}^\BP}^{*-1, *+1, *+1}
\end{align}
in $\Stable(\cA_{\starstar}^\BP)$. The values of the map (\ref{Cmotq}) have been determined through the 110-stem in \cite{aNSS}. For example, \cite{aNSS} gives us that $q_*(\overline{h_1^2 c_0}) = h_1^2 c_0$. Recall that $\Ext_{\cA_{\starstar}^\BP/\tau}$ is isomorphic to $\Ext_{\nubpbptwo/\lambda}$. The class $\overline{h_1^2 c_0}$ in $\Ext_{\cA_{\starstar}^\BP/\tau}$ is denoted by $\alpha_{6/3}$ in $\Ext_{\nubpbptwo/\lambda}$. Therefore there is a differential $d_3(\alpha_{6/3}) = \alpha_1^2 c_0$ in the Adams-Novikov spectral sequence for $\bS_{\bF_2}/\lambda$.
\end{proof}

\begin{remark}
    Remark on computing Cartan-Eilenberg differentials by imposing the relevant filtration in the Curtis algorithm.
\end{remark}

\subsection{Hidden Extensions}
In this section, we compute hidden extensions in the Adams-Novikov spectral sequence for $\bS_{\bF_2}/\lambda$. 

\begin{proposition}
    Hidden extensions by $2$, $\alpha_1$, and $\alpha_{2/2}$ in the Adams-Novikov spectral sequence for $\bS_{\bF_2}/\lambda$ are computed through the 45-stem.
\end{proposition}
\begin{proof}\label{hidext}
    Theorem \ref{stablethm} gives us that $\pi_{\starstar}\bS_{\bF_2}/\lambda$ is isomorphic to $\Ext_{\mathcal{A}_*}$. The bigraded ring $\Ext_{\mathcal{A}_*}$ can be computed by machine \cite{bob}. We can compute extensions in $\pi_{\starstar}\bS_{\bF_2}/\lambda$ which are hidden in the Adams-Novikov spectral sequence by comparison with $\Ext_{\mathcal{A}_*}$. For example, we have $h_0^2 h_2 = h_1^3$ in $\Ext_{\cA_*}$ so there is a hidden $2$-extension from $2\alpha_{2/2}$ to $\alpha_1^3$ in $\pi_{\starstar}\bS_{\bF_2}/\lambda$. All other hidden extensions are determined similarly. 
\end{proof}

\section{\texorpdfstring{The Adams-Novikov Spectral Sequence for $\mathbb{S}_{\mathbb{F}_2}$}{The Adams-Novikov Spectral Sequence for Synthetic Sphere}}\label{ANSSsphere}

In this section, we compute the Adams-Novikov spectral sequence for $\bS_{\mathbb{F}_2}$ through the 45-stem. Up to hidden extensions between $\lambda$-torsion classes, the $\E_2$-page of the synthetic Adams-Novikov spectral sequence is computed in Proposition \ref{synE2}. We compute differentials and hidden extensions in the synthetic Adams-Novikov spectral sequence by analyzing the maps
\begin{align}\label{iq2}
    &\mathbb{S}_{\bF_2} \xrightarrow{i} \mathbb{S}_{\bF_2}/\lambda && \mathbb{S}_{\bF_2}/\lambda \xrightarrow{q} \Sigma^{1, -1}\mathbb{S}_{\bF_2}.
\end{align}

To see interactive charts of the synthetic Adams-Novikov spectral sequence for $\bS_{\bF_2}$, we refer the reader to \cite{BJM25}
\begin{remark}
    On the Adams-Novikov $\E_2$-page the maps in (\ref{iq2}) agree with the maps in (\ref{iqalg}).
\end{remark}

\begin{proposition}\label{hidE2}
    On the $\E_2$-page of the synthetic Adams-Novikov spectral sequence, there are hidden $h$-extensions from $h\beta_{4/4}$ to $\lambda h \beta_3$, from $h^3 \beta_{8/8}$ to $\lambda h  \beta_{6/2}$, and from $h^3 \beta_{6/2}$ to $\lambda^2 P^2 \beta_3$.
\end{proposition}
\begin{proof}
   The hidden extensions in Proposition \ref{hidE2} are computed using Proposition \ref{iqvalues} and the fact that the map $q_*$ is an $\Ext_{\nubpbptwo}$-module map. For example, Proposition \ref{iqvalues} gives us that $q_*(\overline{h\beta_{4/4}}) = h\beta_{4/4}$ and $q_*(\overline{\lambda h \beta_3}) = \lambda h \beta_3$. Recall that in $\Ext_{\nubpbptwo/\lambda}$ there is a $h$-extension from $\overline{h\beta_{4/4}}$ to $\overline{\lambda h \beta_3}$. Since $q_*$ is an $\Ext_{\nubpbptwo}$-module map there must be a $h$-extension from $h\beta_{4/4}$ to $\lambda h \beta_3$ in $\Ext_{\nubpbptwo}$. All hidden extensions in Proposition \ref{hidE2} are computed similarly. 
\end{proof}

\subsection{Synthetic Adams-Novikov Differentials}
In this section, we describe the differentials in the synthetic Adams-Novikov spectral sequence through the 45-stem. In this range, the synthetic Adams-Novikov spectral sequence collapses at the $\E_9$-page. 

\begin{proposition}\label{ANdif}
    Differentials in the Adams-Novikov spectral sequence for $\mathbb{S}_{\bF_2}$ are computed through the 45-stem.
\end{proposition}
\begin{proof}
    Through the 45-stem, many synthetic Adams-Novikov differentials are obtained by direct comparison to the Adams-Novikov spectral sequence for $\mathbb{S}_{\bF_2}/\lambda$. This is done by analyzing the maps induced by (\ref{iq2}) on Adams-Novikov spectral sequences. For example, Proposition~\ref{iqvalues} gives us that the classes $\alpha_1^2 c_0$ and $\alpha_{6/3}$ in the Adams-Novikov $\E_2$-page take nontrivial values under the map $i_*$ in (\ref{iqalg}). Proposition~\ref{cedif} gives us that there is a differential $d_3(\alpha_{6/3}) = \alpha_1^2 c_0$ in the Adams-Novikov spectral sequence for $\mathbb{S}_{\bF_2}/\lambda$. Since the map $i$ in (\ref{iq2}) induces a map of Adams-Novikov spectral sequences which agrees with $i_*$ on $\E_2$-pages, we have a differential $d_3(\alpha_{6/3}) = \alpha_1^2 c_0$ in the Adams-Novikov spectral sequence for $\mathbb{S}_{\bF_2}$. The remaining differentials are accounted for in Table~\ref{a} and are proved in Lemma \ref{leibniz}.
\end{proof}

\begin{lemma}\label{leibniz}
    The synthetic Adams-Novikov differentials which can not be obtained by the method in the proof of Proposition~\ref{ANdif} are given in Table~\ref{a}.
\end{lemma}
\begin{proof}
    We compute the differentials in Table \ref{a} by applying the Leibniz rule to differentials lifted from the Adams-Novikov spectral sequence for $\mathbb{S}_{\mathbb{F}_2}/\lambda$ in the proof of Proposition \ref{ANdif}. For example, Proposition \ref{ANdif} gives us that there is a synthetic Adams-Novikov differential $d_5(\Delta h_2 d_0) = \alpha_{2/2}^2 e_0^2$. Proposition \ref{synE2} gives us that there is an $\alpha_{2/2}$-extension from $\beta_7$ to $\lambda \Delta h_2 d_0$ on the synthetic Adams-Novikov $\E_2$-page. Therefore the Leibniz rule gives us a differential $d_5(\beta_7) = \lambda \alpha_{2/2} e_0^2$. The remaining differentials in the first column of Table \ref{a} are proved similarly. The relations used for proving each differential are given in the second column of Table \ref{a}.
\end{proof}

\begin{table}[hbt!]
\centering
\caption{Leibniz rule applications}
\begin{tabular}{l l l l} 
 \toprule
  Differential & Proof \\ [0.5ex] 
 \hline
  $d_3(P e_0) = \alpha_1^2 c_0 d_0$ & $\beta_3 \cdot \alpha_{10/3} = P e_0$   \\ 
  $d_3(P c_0 e_0) = \alpha_1^4 \beta_3^2$& $c_0 \cdot c_0 = \alpha_1^2 \beta_3$   \\ 
  $d_3(P^2 e_0) = P \alpha_1^2 c_0 \beta_3$& $\alpha_1 \cdot \overline{P^2 \alpha_1 \beta_3} = P^2 e_0$    \\ 
  $d_3(\beta_{8/6, 2}) = \lambda^2 \alpha_1 d_1$& $\alpha_1 \cdot \beta_{8/6, 2} = \lambda^2 \alpha_{4/4}\overline{\alpha_1 \beta_5}$   \\ 
  $d_3(c_0 \beta_3 e_0) = h_1^4 e_0^2$& $\alpha_1^2 \cdot \overline{c_0 \beta_3^2} = c_0 \beta_3 e_0$    \\ 
  $d_3(P \beta_3 e_0) = \alpha_1^2 c_0 \beta_3^2$&$\alpha_1^2 \cdot \overline{P \beta_3^2} = P \beta_3 e_0$    \\ 
  $d_3(h_0 c_2) = h_1 h_3 d_1$& $\alpha_1 \cdot h c_2 = \alpha_{4/4}^2 \overline{\alpha_1 \beta_5}$   \\ 
  $d_3(P^2 c_0 e_0) =  P \alpha_1^4\beta_3^2$ & $\alpha_1^2 \cdot \overline{P^2 c_0 \beta_3} = P^2 c_0 e_0 $    \\ 
  $d_3(P^3 e_0) = P^2 h_1^2 c_0 \beta_3$ & $\alpha_1^2 \cdot \overline{P^3 \beta_3} = P^3 e_0$  \\ 
  $d_3(\beta_3^2 e_0) = \alpha_1^2 c_0 e_0^2$ & $\alpha_1^2 \cdot \overline{\beta_3^3} = \beta_3^2 e_0$&    \\ 
   $d_5(\beta_7) = \lambda \alpha_{2/2} e_0^2$& $\alpha_{2/2} \cdot \beta_7 = \lambda \Delta h_2 d_0$   \\ 
   $d_5(\beta_{6/2}\beta_3) = \alpha_1 c_0 e_0^2$ & $\alpha_{2/2} \cdot \alpha_{2/2}^2 e_0^2 = \lambda \alpha_1 c_0 e_0^2$ and $d_5(\Delta h_2 d_0) = \alpha_{2/2}^2 e_0^2$   \\ 
  $d_5(\beta_{6/2}) = \alpha_1 \beta_2 \beta_4$ & $\beta_3 \cdot \alpha_1 \beta_2 \beta_4 = \alpha_{2/2}^3 e_0^2$   \\ 
 \bottomrule
\end{tabular}
\label{a}
\end{table}

\subsection{Hidden extensions}
In this section, we compute hidden extensions on the synthetic Adams-Novikov $\E_\infty$-page by analyzing the maps
\begin{align*}
    &i_\star: \pi_{\starstar} \mathbb{S}_{\mathbb{F}_2} \to \pi_{\starstar} \mathbb{S}_{\mathbb{F}_2}/\lambda && q_\star: \pi_{\starstar} \mathbb{S}_{\mathbb{F}_2}/\lambda \to  \pi_{\starstar} \mathbb{S}_{\mathbb{F}_2}
\end{align*}
induced on homotopy groups by the maps in (\ref{iq2}).

\begin{proposition}
    Hidden extensions by $\lambda$ on the synthetic Adams-Novikov $\E_\infty$-page are computed through the 45-stem. 
\end{proposition}
\begin{proof}
    To prove hidden extensions by $\lambda$, we analyze the long exact sequence induced on homotopy groups by multiplication by $\lambda$. Recall that the map $q_\star$ above surjects onto the kernel of multiplication by $\lambda$ on homotopy. For example, from Proposition $\ref{synE2}$ we know that the element $P c_0 \beta_3$ is $\lambda$-torsion on the synthetic ANSS $E_2$-page. Proposition \ref{ANdif} gives us that $P c_0 \beta_3$ survives the synthetic ANSS and detects an element in $\pi_{30,11}\bS_{\bF_2}$. Analyzing the synthetic ANSS for $\bS_{\bF_2}/\lambda$ and Lemma \ref{iqvalues} shows that there is nothing in $\pi_{31,10}\mathbb{S}_{\bF_2}/\lambda$ which can be sent to $P c_0 \beta_3$ by the map $q_\star$. So $P c_0 \beta_3$ can not be $\lambda$-torsion so it must support a hidden $\lambda$-extension. The only possible target is $\alpha_1^2 \beta_3^2$ so we have a hidden $\lambda$-extension on the synthetic ANSS $\E_\infty$-page from $P c_0 \beta_3$ to  $\alpha_1^2 \beta_3^2$.  All other hidden $\lambda$-extensions are computed similarly. 
\end{proof}

\begin{proposition}
    Hidden extensions by $h$, $\alpha_1$, and $\alpha_{2/2}$ on the synthetic Adams-Novikov $\E_\infty$-page are computed through the 45-stem. 
\end{proposition}

\begin{proof}
    We compute hidden extensions on the synthetic Adams-Novikov $\E_\infty$-page by comparison with the homotopy of $\bS_{\bF_2}/\lambda$. For example, there is an $\alpha_{2/2}$-extension from $\beta_4$ to $h_2 g$ in the homotopy of $\bS_{\bF_2}/\lambda$. Therefore, there is an analogous hidden extension in the homotopy of $\bS_{\bF_2}$ since $\beta_4$ and $h_2 g$ map to these values via the map $i$ and $i$ preserves multiplication. Most hidden extensions can be obtained in this way or by propagating such extensions by taking products. There are a few exceptions, namely an $\alpha_1$-extension from $\beta_{8/4, 2}$ to $\lambda^2 c_1 g$ and $\alpha_{2/2}$-extensions from $\beta_6$ to $\lambda^2 \alpha_1 e_0^2$ and from $\Delta h_1 d_0$ to $c_0 e_0^2$. These extensions can be obtained from the analogous extensions in the classical stable stems since their sources and targets are all $\lambda$-periodic. 
\end{proof}

\appendix

\section{\texorpdfstring{Nilpotent Completion in Stable $\infty$-Categories}{Nilpotent Completion in Stable Infinity Categories}}
\label{nilcompappendix}

The classical $E$-Adams spectral sequence can be used to compute the homotopy groups $\pi_*(X_E^{\wedge})$ of the $E$-nilpotent completion of a spectrum $X$. A celebrated result of \cite{Bou79} relates the $E$-nilpotent completion to Bousfield localization at $E$; namely, the $E$-nilpotent completion is $E$-local and the natural map $L_EX\to X_E^{\wedge}$ from the $E$-localization of $X$ is an equivalence when $E$ is connective, $X$ is bounded-below, and $\pi_0E$ satisfies certain conditions. In addition, there is an equivalence $L_EX\simeq L_{M\pi_0 E}X$ with localization with respect to the Moore spectrum $M\pi_0 E$ of $\pi_0E$. This is useful because Bousfield localization can be easier to understand than nilpotent completion. For example, if $E=\bF_p$ and $X$ bounded-below, then $L_{\bF_p}X\xrightarrow{\simeq}X_{\bF_p}^{\wedge}$ and $L_{\bF_p}X\simeq L_{\bS/p}X\simeq X_p^{\wedge}$, the $p$-completion of $X$.

\bigskip

Recent work of \cite{Man21} generalizes Bousfield's results from spectra to presentable, stable $\infty$-categories. Roughly speaking, $E$ is instead a homotopy commutative algebra, connectivity corresponds to the connective part of a $t$-structure, Moore spectra correspond to the cofibers and/or localizations of a set of maps, and conditions on $\pi_0E$ correspond to conditions on $\pi_0^{\heartsuit}E$, the object associated to $E$ in the heart of the $t$-structure.

\bigskip

In this appendix, we recall the background and precise results of \cite{Man21} and use them to study examples of nilpotent completion relevant to this paper.

\bigskip

\textbf{Notation and Assumptions.} Throughout we let $(\cC,\otimes,\one)$ denote a presentable, symmetric monoidal stable $\infty$-category $\cC$ with monoidal product $\otimes$ and unit object $\one$ and $E$ a homotopy commutative algebra object of $\cC$. These $\cC$ are equipped with an accessible $t$-structure $(\cC_{\geq 0},\cC_{\leq 0})$ with truncation functors $\tau_{\geq n},\tau_{\leq n}$ satisfying the following (see \cite[Sec. 2.1.1]{Man21}):

\begin{itemize}
    \item the $t$-structure is left-complete, i.e. $\lim_n\tau_{\leq n}X\simeq X$ for every $X\in\cC$;
    \item $\one\in\cC_{\geq 0}$, i.e. the unit is connective;
    \item $\cC_{\geq p}\otimes\cC_{\geq q}\subseteq \cC_{\geq (p+q)}$ for any $p,q\in\bZ$;
    \item the truncation functors $\tau_{\geq n}(-)$ commute with filtered colimits.
\end{itemize}

In this situation, the heart $\cC^{\heartsuit}$ gets the structure of a symmetric monoidal $\infty$-category and the functor $\pi_0^{\heartsuit}:\cC\to\cC^{\heartsuit}$ is symmetric monoidal (see \cite[Sec. 2.1.3]{Man21}).

\subsection{Adams resolutions and nilpotent completion}

Now for every object $X\in\cC$, we can associate the canonical $E$-Adams resolution of $X$. This is done in the standard way: let $\overline{E}:=\mathrm{fib}(\one\to E)$ and $\overline{E}^n:=\overline{E}^{\otimes n}$. Then the canonical Adams resolution is a tower
\begin{equation*}
    \begin{tikzcd}
        X \ar[d] & \overline{E}\otimes X \ar[l] \ar[d] & \overline{E}^2\otimes X \ar[d] \ar[l] & \cdots \ar[l] \\
        E\otimes X & E\otimes\overline{E}\otimes X & E\otimes\overline{E}^2\otimes X &  
    \end{tikzcd}
\end{equation*}
with cofiber sequences
\begin{equation*}
    \overline{E}^{n+1}\otimes X\to \overline{E}^n\otimes X\to E\otimes\overline{E}^n\otimes X.
\end{equation*}

Composing successive maps in the Adams tower gives maps $\overline{E}^n\otimes X\to X$. For $X=\one$, let $\overline{E}_n:=\mathrm{cof}(\overline{E}^n\otimes\to \one)$. There are induced maps $\overline{E}_n\otimes X\to \overline{E}_{n-1}\otimes X$ and we define the $E$-\textit{nilpotent completion} $X_E^{\wedge}$ \textit{of} $X$ to be the limit
\begin{equation*}
    X_E^{\wedge}:= \lim (\cdots\to\overline{E}_{2}\otimes X\to \overline{E}_{1}\otimes X).
\end{equation*}

\begin{remark}
    When $E$ is an $\bE_1$-ring, an equivalent definition of the nilpotent completion can be obtained as the totalization of
    \begin{equation*}
        \begin{tikzcd}
            X \ar[r] & E\otimes X  \arrow[r, shift left] \arrow[r, shift right] & E\otimes E \otimes X \arrow[r]
\arrow[r, shift left=2]
\arrow[r, shift right=2] & \cdots 
        \end{tikzcd}
    \end{equation*}
    which as a coaugmented cosimplicial object we denote by $X\to E^{\bullet +1} \otimes X$. Here the coface maps insert unit maps $\one\to E$ and codegeneracy maps are given by the multiplication $E\otimes E\to E$ on $E$.
\end{remark}

\subsection{Bousfield Localization}
Given a presentably symmetric monoidal stable $\infty$-category $\cC$ and an object $A\in \cC$, one can define the $A$-Bousfield localization by taking the Verdier quotient by those objects that vanish after tensoring with $A$. We let $L_A$ denote the localization functor. The $A$-nilpotent completion of an object $X$ is always $A$-local, inducing a factorization $X\to L_AX\to X_A^\wedge$. When $\cC=\Sp$, Bousfield \cite{Bou79} famously gave very general conditions on which this second map is an equivalence. This is generalized by \cite{Man21}.

\begin{assumption} (\cite[Ass. 4.2.1]{Man21}) Suppose a homotopy commutative algebra object $E\in\cC$ satisfies the following:
\label{pi0assumption}
\begin{enumerate}
    \item $E\in\cC_{\geq 0}$.
    \item There is a finite set $\{K_i\}$ and countable set $\{L_j\}$ of $\otimes$-invertible objects in $\cC$ such that each functor $K_i\otimes (-)$ and $L_j\otimes (-)$ sends (co)connective objects to (co)connective objects, maps $f_i:K_i\to\one$ and $g_j:L_j\to\one$, and a morphism of $\pi_0^{\heartsuit}(\one)$-algebras
    \begin{equation*}
        \varphi:(\pi_0^{\heartsuit}(\one)/\mathcal{I})[\mathcal{J}^{-1}]\to\pi_0^{\heartsuit}E,
    \end{equation*}
    where $\mathcal{I}$ is the ideal of $\pi_0^{\heartsuit}(\one)$ generated by $\{f_i\}$ and $\mathcal{J}$ is the collection $\{g_j\}$.
    \item The map $\varphi$ in (2) is an isomorphism.
\end{enumerate}
\end{assumption}

\begin{remark}
    Item (2) in the above is slightly different in \cite{Man21}. However, since we are assuming the functors $\tau_{\geq n}$ commute with filtered colimits, it is equivalent as stated.
\end{remark}

The following theorems relate $E$-localization to $E$-nilpotent completion:

\begin{theorem}(\cite[Thm. 7.3.5]{Man21})
\label{J0BousEquivNil}
Suppose $E$ satisfies Assumption~\ref{pi0assumption} with $\mathcal{J}=\emptyset$. Then for every bounded-below object $X\in\cC_{\geq k}$, the natural map $L_EX\to X_E^{\wedge}$ is an equivalence.    
\end{theorem}

\begin{theorem}(\cite[Thm. 7.3.8]{Man21})
\label{I0BousEquivNil}
Suppose $E$ satisfies Assumption~\ref{pi0assumption} with $\mathcal{I}=\emptyset$. Then for every bounded-below object $X\in\cC_{\geq k}$, the natural map $L_EX\to X_E^{\wedge}$ is an equivalence.    
\end{theorem}

The following theorems relate $E$-localization to localization with respect to some Moore object:

\begin{theorem}(\cite[Thm. 4.3.7]{Man21})
\label{MooreObjectThm1}
Suppose $E$ satisfies Assumption~\ref{pi0assumption} with $\mathcal{J}=\emptyset$. Let $M\pi_0^{\heartsuit}E$ denote $C(f_1)\otimes\cdots \otimes C(f_r)$. Then for every bounded-below object $X\in\cC_{\geq k}$, we have an equivalence
\begin{equation*}
    L_{M\pi_0^{\heartsuit}E}X\simeq L_{E}X.
\end{equation*}
\end{theorem}

The following theorem is not recorded in \cite{Man21}; however, it follows from the same arguments used to prove \cite[Thm. 4.3.7]{Man21}:

\begin{theorem}(\cite{Man21})
\label{MooreObjectThm2}
Suppose $E$ satisfies Assumption~\ref{pi0assumption} with $\mathcal{I}=\emptyset$. Let $M\pi_0^{\heartsuit}E$ denote $\one[\mathcal{J}^{-1}]$. Then for every bounded-below object $X\in\cC_{\geq k}$, we have an equivalence
\begin{equation*}
    L_{M\pi_0^{\heartsuit}E}X\simeq L_{E}X.
\end{equation*}
\end{theorem}

\subsection{Application to Stable Comodule Categories}

We can now apply the above results to several examples of interest. In particular, we prove new results about completions in $\Stable(\Gamma)$ for Hopf algebroids $\Gamma$ studied in this paper.

\bigskip

As a warm-up, we prove completion results in $\Stable(\cA_*)$ and $\Stable(\BP_*\BP)$. The authors of this paper were unable to find $\Stable(\BP_*\BP)$ completion results elsewhere in the literature and believe this is a new result. To do so, we first need to prove a lemma concerning the $t$-structures on $\Stable(\cA_*)$ and $\Stable(\BP_*\BP)$:

\begin{lemma}
\label{BPtstructurelemma}
    Consider the $t$-structures on $\Stable(\cA_*)$ and $\Stable(\BP_*\BP)$ described in Example~\ref{comodexampA} and Example~\ref{comodexampBP}. These $t$-structures satisfy the following:
    \begin{enumerate}
        \item They are left-complete; i.e. for every object $X$, $$X\simeq\lim_n\tau_{\leq n}X.$$
        \item $\bF_p\in\Stable(\cA_*)_{\geq 0}$ and $\BP_*\in\Stable(\BP_*\BP)_{\geq 0}$.
        \item For $n,m\in\bZ$,
        \begin{align*}
           \Stable(\cA_*)_{\geq n}\otimes \Stable(\cA_*)_{\geq m}&\subseteq \Stable(\cA_*)_{\geq (n+m)}, \\
           \Stable(\BP_*\BP)_{\geq n}\otimes \Stable(\BP_*\BP)_{\geq m}&\subseteq \Stable(\BP_*\BP)_{\geq (n+m)}.
        \end{align*}
        \item The functor $\tau_{\geq n}(-)$ commutes with filtered colimits.
    \end{enumerate}
\end{lemma}

\begin{proof}
(1) follows from Lemma~\ref{leftcomplemma}, (3) follows from Lemma~\ref{tmultlemma}, and (4) follows from Lemma~\ref{filtcolimlemma}. (2) follows the fact that $\cA_*$ and $\BP_*\BP$ are connective; i.e. $\cA_t=0=\BP_t\BP$ for $t<0$ and the counit maps
\begin{align*}
\cA_*&\to\bF_p \\
\BP_*\BP&\to\BP_*
\end{align*}
are isomorphisms in degree 0.
\end{proof}

The following is not a new result and follows from \cite[Prop. 1.4.3]{Pal01}. However, we reprove it using the machinery of \cite{Man21}:

\begin{proposition}
\label{Acompletiontheorem}
    Consider the commutative algebra object $E=\cP_*=\h_*\BP\in\Stable(\cA_*)$ and the $t$-structure on $\Stable(\cA_*)$ from Lemma~\ref{BPtstructurelemma}. Then for every bounded-below $X\in\Stable(\cA_*)_{\geq k}$,
    \begin{equation*}
        X_{\cP_*}^{\wedge}\simeq L_{\bF_p}X\simeq X.
    \end{equation*}   
\end{proposition}

\begin{proof}
We first show that Assumption~\ref{pi0assumption} is satisfied. Note that
\begin{equation*}
    \pi_{-*,*}(\cP_*)=\Ext^{*,*}_{\cA_*}(\bF_p,\cP_*)\cong\bF_p[h_0,v_1,v_2,\ldots]
\end{equation*}
with Ext degree $\vert h_0\vert=(1,1)$ and $\vert v_i\vert=(1,2p^i-1)$, so that $\cP_*\in\Stable(\cA_*)_{\geq 0}$. By Example~\ref{comodexampA}, the heart of the $t$-structure on $\Stable(\cA_*)$ is $\bF_p[h_0]$-mod with $$\pi_0^{\heartsuit}(X)_t=\Ext^{t,t}_{\cA_*}(\bF_p,X).$$

In particular,
\begin{align*}
 \pi_0^{\heartsuit}(\cP_*)_*&=\bF_p[h_0], \\
 \pi_0^{\heartsuit}(\bF_p)_*&=\bF_p[h_0].
\end{align*}
If we take $\mathcal{I}=\emptyset=\mathcal{J}$, then we have an isomorphism of $\bF_p[h_0]$-algebras
\begin{equation*}
    \bF_p[h_0]\cong\pi_0^{\heartsuit}(\bF_p)_*\xrightarrow{\varphi} \pi_0^{\heartsuit}(\cP_*)_*\cong\bF_p[h_0].
\end{equation*}
By Theorem~\ref{J0BousEquivNil} and Theorem~\ref{MooreObjectThm1}, this implies that $X_{\cP_*}^{\wedge}\simeq L_{\bF_p}X$. The equivalence $L_{\bF_p}X\simeq X$ follows since $\bF_p$ is the unit of $\Stable(\cA_*)$.
\end{proof}

\begin{theorem}
\label{BPcompletiontheorem}
    Consider the commutative algebra object $E=\cP_*=\h_*\BP\in\Stable(\BP_*\BP)$ and the $t$-structure on $\Stable(\BP_*\BP)$ from Lemma~\ref{BPtstructurelemma}. Then for every bounded-below $X\in\Stable(\BP_*\BP)_{\geq k}$,
    \begin{equation*}
        X_{\cP_*}^{\wedge}\simeq L_{\BP_*/p}X\simeq X_p^{\wedge}
    \end{equation*}
    where $p$-completion is taken with respect to $p\in\bZ_{(p)}\cong\Ext_{\BP_*\BP}^{0,0}(\BP_*,\BP_*)$.
\end{theorem}

\begin{proof}
We first show that Assumption~\ref{pi0assumption} is satisfied. Note that
\begin{equation*}
    \pi_{-f,t}(\cP_*)=\Ext^{f,t}_{\BP_*\BP}(\BP_*,\cP_*)\cong
    \begin{cases}
        \bF_p, & f=t=0, \\
        0, & \mathrm{otherwise}.
    \end{cases}
\end{equation*}
so that $\cP_*\in\Stable(\BP_*\BP)_{\geq 0}$. Now by Example~\ref{comodexampBP}, the heart of the $t$-structure on $\Stable(\BP_*\BP)$ is $(\bZ_{(p)}\text{-mod})^{\mathrm{gr}}$ with $$\pi_0^{\heartsuit}(X)_t=\Ext_{\BP_*\BP}^{t,t}(\BP_*,X).$$ In particular, if we take $K_1=\BP_*$ and $f_1=p:\BP_*\to\BP_*$ and $\mathcal{J}=\emptyset$ then because $\pi_0^{\heartsuit}(\BP_*)_*\cong \bZ_{(p)}$ concentrated in total degree 0, we have an isomorphism $\varphi$ of $\bZ_{(p)}$-algebras
\begin{equation*}
    \bF_p\cong\pi_0^{\heartsuit}(\BP_*/p)_*\cong (\pi_0^{\heartsuit}(\BP_*)_*)/p\xrightarrow{\varphi} \pi_0^{\heartsuit}(\cP_*)_*\cong\bF_p.
\end{equation*}
By Theorem~\ref{J0BousEquivNil} and Theorem~\ref{MooreObjectThm1}, this tells us that for $X\in\Stable(\BP_*\BP)_{\geq k}$, $X_{\cP_*}^{\wedge}\simeq L_{\BP_*/p}X$.

\bigskip

All that's left to show is $L_{\BP_*/p}X\simeq X_p^{\wedge}$. However, this follows from a similar argument to \cite[Prop. 2.5]{Bou79} because $X_p^{\wedge}\simeq F(\Sigma^{-1,0}\BP_*/p^{\infty},X)$ and $X\to F(\Sigma^{-1,0}\BP_*/p^{\infty},X)$ is a $\BP_*/p$-localization, where we let $\BP_*/p^{\infty}$ denote the colimit of the diagram
\begin{equation*}
    \BP_*/p\xrightarrow{p}\BP_*/p^2\xrightarrow{p}\BP_*/p^3\xrightarrow{p}\cdots .
\end{equation*}
\end{proof}

\begin{corollary}
   In $\Stable (\BP_*\BP)$, the unit $\BP_*$ satisfies
   \begin{equation*}
       (\BP_*)_{\cP_*}^{\wedge}\simeq (\BP_*)_p^{\wedge}.
   \end{equation*}
   In particular, the homotopy groups of the completion satisfy $$\pi_{-f,t}((\BP_*)_{\cP_*}^{\wedge})\cong \pi_{-f,t}(\BP_*)\otimes\bZ_p^{\wedge}\cong \Ext_{\BP_*\BP}^{f,t}(\BP_*,\BP_*)\otimes\bZ_p^{\wedge}.$$
\end{corollary}

\begin{proof}
    The first statement follows from the the fact that $\BP_*\in\Stable(\BP_*\BP)_{\geq 0}$ and Theorem~\ref{BPcompletiontheorem}. The second statement follows from the fact that $$\Ext_{\BP_*\BP}^{f,t}(\BP_*,\BP_*)$$ is a finitely generated, $p$-local abelian group for all $f,t\in\bZ$.
\end{proof}

Now we consider the synthetic versions of Proposition~\ref{Acompletiontheorem} and Theorem~\ref{BPcompletiontheorem}. This is useful for the proofs of Proposition~\ref{synAnilpotentcomplete} and Proposition~\ref{synBPnilpotentcomplete}. Here the $t$-structures we use on $\Stable(\cA_{\starstar}^\BP)$ and $\Stable(\BP_{\starstar}\BP^{\bF_p})$ are Example~\ref{comodexampAsyn} and Example~\ref{comodexampBPsyn}. These $t$-structures satisfy the conditions needed to study nilpotent completion:

\begin{lemma}
\label{syntheticBPtstructurelemma}
Consider the $t$-structures on $\Stable(\cA_{\starstar}^\BP)$ and $\Stable(\nubpbp)$ described in Example~\ref{comodexampAsyn} and Example~\ref{comodexampBPsyn}. These $t$-structures satisfy the following:
\begin{enumerate}
    \item They are left complete.
    \item $\bF_p[\tau]\in\Stable(\cA_{\starstar}^\BP)^\cell_{\geq 0}$ and $\BP^{\bF_p}_{\starstar}\in \Stable(\BP_{\starstar}\BP^{\bF_p})^\cell_{\geq 0}$.
    \item For $p,q\in\bZ$,
    \begin{align*}
    \Stable(\cA_{\starstar}^\BP)^\cell_{\geq p}\otimes \Stable(\cA_{\starstar}^\BP)^\cell_{\geq q}&\subseteq \Stable(\cA_{\starstar}^\BP)^\cell_{\geq (p+q)} \\
     \Stable(\BP_{\starstar}\BP^{\bF_p})^\cell_{\geq p}\otimes \Stable(\BP_{\starstar}\BP^{\bF_p})^\cell_{\geq q}&\subseteq \Stable(\BP_{\starstar}\BP^{\bF_p})^\cell_{\geq (p+q)}.   
    \end{align*}
    \item The functor $\tau_{\geq n}(-)$ commutes with filtered colimits.
\end{enumerate}  
\end{lemma}

\begin{proof}
(1) follows from Lemma~\ref{leftcomplemma}, (3) follows from Lemma~\ref{tmultlemma}, and (4) follows from Lemma~\ref{filtcolimlemma}. (2) follows from the fact that
\begin{align*}
    \pi_{k,*}(\nu_{\BP}\bF_p\otimes\nu_{\BP}\bF_p)&=(\cA_{k,*}^\BP)=0 \\
    \pi_{k,*}(\nu_{\bF_p}\BP\otimes\nu_{\bF_p}\BP)&=\BP_{k,*}\BP^{\bF_p}=0
\end{align*}
whenever $k<0$ and the counit maps
\begin{align*}
    \cA_{\starstar}^\BP&\to\bF_p[\tau] \\
    \nubpbp&\to\nubp_{\starstar}
\end{align*}
are isomorphisms in bidegree $(0,w)$, for $w\in\bZ$.
\end{proof}

\begin{theorem}
\label{synAcompletiontheorem}
    Consider the commutative algebra object $$E=\cP_*[\tau]=\pi_{*,*}(\nu_{\BP}(\bF_p\otimes\BP))\in\Stable(\cA_{\starstar}^\BP)$$ and the $t$-structure on $\Stable(\cA_{\starstar}^\BP)$ from Lemma~\ref{syntheticBPtstructurelemma}. Then for every bounded-below $X\in\Stable(\cA_{\starstar}^\BP)_{\geq k}$,
    \begin{equation*}
        X_{\cP_*[\tau]}^{\wedge}\simeq L_{\bF_p[\tau]}X\simeq X.
    \end{equation*}   
\end{theorem}

\begin{proof}
We first show that Assumption~\ref{pi0assumption} is satisfied. Note that
\begin{align*}
    \pi_{-*,*,*}(\cP_*[\tau])&=\Ext^{*,*,*}_{\cA_{\starstar}^\BP}(\bF_p[\tau],\cP_*[\tau])\cong\Ext^{*,*,*}_{\cA_*}(\bF_p,\cP_*)[\tau] \\
    &\cong\bF_p[\tau,h_0,v_1,v_2,\ldots]
\end{align*}
with Ext degree $\vert h_0\vert=(1,1,0)$ and $\vert v_i\vert=(1,2p^i-1,2p^i-2)$, so that $\cP_*\in\Stable(\cA_{\starstar}^\BP)_{\geq 0}$. By Example~\ref{comodexampAsyn}, the heart of the $t$-structure on $\Stable(\cA_{\starstar}^\BP)$ is $\bF_p[\tau,h_0]$-mod with $$\pi_0^{\heartsuit}(X)_{t,w}=\Ext^{t,t,w}_{\cA_{\starstar}^\BP}(\bF_p[\tau],X).$$

In particular,
\begin{align*}
 \pi_0^{\heartsuit}(\cP_*)_{*,*}&=\bF_p[\tau,h_0], \\
 \pi_0^{\heartsuit}(\bF_p)_{*,*}&=\bF_p[\tau,h_0].
\end{align*}
If we take $\mathcal{I}=\emptyset=\mathcal{J}$, then we have an isomorphism of $\bF_p[\tau,h_0]$-algebras
\begin{equation*}
    \bF_p[\tau,h_0]\cong\pi_0^{\heartsuit}(\bF_p[\tau])_{*,*}\xrightarrow{\varphi} \pi_0^{\heartsuit}(\cP_*)_{*,*}\cong\bF_p[\tau,h_0].
\end{equation*}
By Theorem~\ref{J0BousEquivNil} and Theorem~\ref{MooreObjectThm1}, this implies that $X_{\cP_*[\tau]}^{\wedge}\simeq L_{\bF_p[\tau]}X\simeq X$.
\end{proof}

In the case of $\Stable(\nubpbp)$, completion with respect to $\cP_*[\lambda]$ turns out to also be a $p$-completion:

\begin{theorem}
\label{synBPcompletiontheorem}
Consider the commutative algebra object $$E=\cP_*[\lambda]\simeq \pi_{*,*}(\nu_{\bF_p}(\bF_p\otimes\BP))\in \Stable(\BP_{\starstar}\BP^{\bF_p})$$ and the $t$-structure on $\Stable(\nubpbp)$ considered in Lemma~\ref{syntheticBPtstructurelemma}. Then for every bounded-below $X\in\Stable(\BP_{\starstar}\BP^{\bF_p})_{\geq k}$,
\begin{equation*}
    X^{\wedge}_{\cP_*[\lambda]}\simeq L_{\BP_{\starstar}^{\bF_p}/p}X\simeq X^{\wedge}_p,
\end{equation*}
where $p$-completion is taken with respect to $p\in\bZ_{(p)}\cong\Ext^{0,0,0}_{\BP_{\starstar}\BP^{\bF_p}}(\BP^{\bF}_{\starstar},\BP^{\bF}_{\starstar})$.
\end{theorem}

\begin{proof}
    Again we show that Assumption~\ref{pi0assumption} is satisfied. In Lemma~\ref{synBPextlemma}, we computed the Ext groups of tensor powers of $\cP_*[\lambda]$. In particular, we get that
    \begin{equation*}
        \Ext^{f,t,*}_{\BP_{\starstar}\BP^{\bF_p}}(\BP^{\bF}_{\starstar},\cP_*[\lambda])\cong
        \begin{cases}
            \bF_p[\lambda], &f=t=0, \\
            0, & \mathrm{otherwise},
        \end{cases}
    \end{equation*}
so that $\cP_*[\lambda]\in\Stable(\BP_{\starstar}\BP^{\bF_p})_{\geq 0}$ and, as a bigraded object, $\pi_{0,*,*}^{\heartsuit}(\cP_*[\lambda])\cong\bF_p[\lambda]$, with the non-$\lambda$-divisible copy of $\bF_p$ in internal degree 0. Now whenever $t-f=0$, we have that
\begin{equation*}
    \Ext^{f,t,*}_{\BP_{\starstar}\BP^{\bF_p}}(\BP^{\bF}_{\starstar},\BP^{\bF}_{\starstar})\cong
        \begin{cases}
            \bZ_{(p)}[\lambda], &f=t=0, \\
            0, & \mathrm{otherwise},
        \end{cases}
\end{equation*}
In particular, this means that $\pi_{0}^{\heartsuit}(\BP^{\bF_p}_{\starstar})_{*,*}\cong\bZ_{(p)}[\lambda]$ and $\pi_{0}^{\heartsuit}(\BP^{\bF_p}_{\starstar}/p)_{*,*}\cong\bF_p[\lambda]$. Hence if we take $K_1=\BP^{\bF_p}$, $f_1=p:\BP^{\bF_p}\to \BP^{\bF_p}$, and $\mathcal{J}=\emptyset$, we get an isomorphism $\varphi$ of $\bZ_{(p)}[\lambda]$-algebras
\begin{equation*}
 \bF_p[\lambda]\cong\pi_{0}^{\heartsuit}(\BP^{\bF_p}_{\starstar}/p)_{*,*}\cong(\pi_{0}^{\heartsuit}(\BP^{\bF_p}_{\starstar})_{*,*})/p\xrightarrow{\varphi} \pi_{0}^{\heartsuit}(\cP_*[\lambda])_{*,*}\cong\bF_p[\lambda].  
\end{equation*}
Applying Theorems~\ref{J0BousEquivNil} and \ref{MooreObjectThm1}, we get that for any $X\in\Stable(\BP_{\starstar}\BP^{\bF_p})_{\geq k}$, $X_{\cP_*[\lambda]}^{\wedge}\simeq L_{\BP^{\bF_p}_{\starstar}/p}X$. Similar to the proof of Theorem~\ref{BPcompletiontheorem}, we also get that $L_{\BP^{\bF_p}_{\starstar}/p}X\simeq X_p^{\wedge}$.
\end{proof}

The unit is connective in $\Stable (\BP_{\starstar}\BP^{\bF_p})$ so that we immediately get the following corollary:

\begin{corollary}
\label{synBPunitcompletion}
In $\Stable (\BP_{\starstar}\BP^{\bF_p})$, the unit $\BP_{\starstar}^{\bF_p}$ satisfies
   \begin{equation*}
       (\BP_{\starstar}^{\bF_p})_{\cP_*[\lambda]}^{\wedge}\simeq (\BP_{\starstar}^{\bF_p})_p^{\wedge}.
   \end{equation*}
   In particular, the homotopy groups of the completion satisfy $$\pi_{f,t,w}((\BP_{\starstar}^{\bF_p})_{\cP_*[\lambda]}^{\wedge})\cong \pi_{f,t,w}(\BP_{\starstar}^{\bF_p})\otimes\bZ_p^{\wedge}\cong \Ext_{\BP_{\starstar}\BP}^{f,t,w}(\BP_{\starstar}^{\bF_p},\BP_{\starstar}^{\bF_p})\otimes\bZ_p^{\wedge}.$$    
\end{corollary}

\subsection{Application to Synthetic Adams-Novikov Spectral Sequence}

We end this appendix by describing what the synthetic Adams-Novikov spectral sequence for a bounded-below object $X\in(\Syn_{\bF_p})_{\geq k}$ converges to. Analogous to the classical Adams-Novikov spectral sequence, it turns out to converge to the $p$-localization of $X$. 

\begin{remark}
    In this section, we study examples of our twisted $t$-structures in $\Syn^\cell_E$ which have recently also appeared as examples of the linear $t$-structures described in \cite[Defn. 2.1]{CD24}.
\end{remark}

\begin{lemma}
\label{Fpsyntstructurelemma}
    Consider the $t$-structure on $\Syn_{\bF_p}=\Syn_{\bF_p}^{\cell}$ described in Example~\ref{Fpsynexample}. This $t$-structure satisfies the following:
    \begin{enumerate}
    \item It is left complete.
    \item $\bS_{\bF_p}\in(\Syn_{\bF_p})_{\geq 0}$.
    \item For $n,m\in\bZ$,
    \begin{equation*}
    (\Syn_{\bF_p})_{\geq n}\otimes (\Syn_{\bF_p})_{\geq m}\subseteq (\Syn_{\bF_p})_{\geq (n+m)}. 
    \end{equation*}
    \item The functor $\tau_{\geq n}(-)$ commutes with filtered colimits.
    \end{enumerate}
\end{lemma}

\begin{proof}
    (1) follows from Lemma~\ref{leftcomplemma}, (3) follows from Lemma~\ref{tmultlemma}, and (4) follows from Lemma~\ref{filtcolimlemma}. For (2), we need to show that $\pi_{k,w}\bS_{\bF_p}=0$ whenever $k<0$. This follows from an application of the $\nu\bF_p$-Adams spectral sequence for $\bS_{\bF_p}$. By \cite{BHS23}, ${}_{\nu\bF_p}\E_2^{*,*,*}={}_{\mathrm{cl}}\E_2^{*,*}\otimes\bZ[\tau]$. For $k<0$, the classical Adams $\E_2$-page for the sphere is zero, and so the same must be true for ${}_{\nu\bF_p}\E_2^{*,*,*}$ and hence for $\pi_{k,w}((\bS_{\bF_p})_{\lambda}^{\wedge})$. The result for $\pi_{k,w}\bS_{\bF_p}$ then follows by inspection of the long exact sequence in homotopy for the fracture square
    \begin{equation*}
    \begin{tikzcd}
        \bS_{\bF_p} \ar[r] \ar[d] \arrow[dr, phantom, "\usebox\pullback" , very near start, color=black] & (\bS_{\bF_p})_{\lambda}^{\wedge} \ar[d] \\
        \bS_{\bF_p}[\lambda^{-1}] \ar[r] & (\bS_{\bF_p})_{\lambda}^{\wedge}[\lambda^{-1}]
    \end{tikzcd}
\end{equation*}
\end{proof}

\begin{theorem}
\label{synANSSconvergencethm}
    Consider the commutative algebra object
    $$
        E=\nubp\in\Syn_{\bF_p}
    $$
    and the $t$-structure considered in Lemma~\ref{Fpsyntstructurelemma}. Then for every bounded-below $X\in(\Syn_{\bF_p})_{\geq k}$,
    $$
    X_{\nubp}^{\wedge}\simeq X_{(p)}
    $$
    where $p$-localization is taken with respect to $p\in\pi_{0,0}\bS_{\bF_p}\cong\bZ$.
\end{theorem}

\begin{proof}
    Again we show that Assumption~\ref{pi0assumption} is satisfied. By Theorem~\ref{BPhomotopytheorem} and Theorem~\ref{synspectraheartthm},
    $$
    \pi_{0}^{\heartsuit}(\nubp)_*=\pi_{0,*}\nubp=\bZ_{(p)}[\lambda,h]/(\lambda h=p)
    $$
    so that $\nubp\in(\Syn_{\bF_p})_{\geq 0}$. By Example~\ref{Fpsynexample}, we know that $$\pi^{\heartsuit}_{0}(\bS_{\bF_p})_*=\bZ[\lambda,h]/(\lambda h=p).$$ Localizing at $p$ then implies that 
    $$
    \pi_{0}^{\heartsuit}((\bS_{\bF_p})_{(p)})_*=\bZ_{(p)}[\lambda,h]/(\lambda h=p).
    $$
    Hence if we take $\mathcal{I}=\emptyset$ and $\mathcal{J}=\bZ\setminus (p)=\pi_{0,0}\setminus (p)$, then there is an isomorphism of $\bZ[\lambda,h]/(\lambda h=p)$-algebras
    $$
    \bZ_{(p)}[\lambda,h]/(\lambda h=p)=\pi_{0}^{\heartsuit}((\bS_{\bF_p})_{(p)})_*\xrightarrow{\varphi}\pi_{0}^{\heartsuit}(\nubp)_*=\bZ_{(p)}[\lambda,h]/(\lambda h=p).
    $$
    By Theorems~\ref{I0BousEquivNil} and \ref{MooreObjectThm2}, for $X\in(\Syn_{\bF_p})_{\geq k}$ we get equivalences
    $$
        X_{\nubp}^{\wedge}\simeq L_{\nubp}X\simeq L_{(\bS_{\bF_p})_{(p)}}X.
    $$
    By \cite[Prop. 3.4.3.t.2]{Man21}, Bousfield localization at $(\bS_{\bF_p})_{(p)}$ is smashing; i.e. $L_{(\bS_{\bF_p})_{(p)}}X\to (\bS_{\bF_p})_{(p)}\otimes X$ is an equivalence. From this, it's clear then that $$L_{(\bS_{\bF_p})_{(p)}}X\simeq (\bS_{\bF_p})_{(p)}\otimes X\simeq X_{(p)}.$$
\end{proof}

\begin{corollary}
    For the unit $\bS_{\bF_p}\in(\Syn_{\bF_p})_{\geq 0}$, we have isomorphisms
    $$
        \pi_{k,w}((\bS_{\bF_p})^{\wedge}_{\nubp})\cong\pi_{k,w}((\bS_{\bF_p})_{(p)})\cong\pi_{k,w}(\bS_{\bF_p})\otimes\bZ_{(p)}.
    $$
\end{corollary}

\begin{proof}
    The first isomorphism follows from Theorem~\ref{synANSSconvergencethm}. If we prove that $\pi_{k,w}(\bS_{\bF_p})$ is a finitely generated abelian group for all $k,w\in\bZ$, then the second isomorphism follows. Note that for $k-w\geq 0$, $\pi_{k,w}\bS_{\bF_p}\cong\pi_k\bS$, which is finitely generated for all $k$. Using the long exact sequence
    $$
\cdots\xrightarrow{i}\Ext^{w-k-1,w}_{\cA_*}(\bF_p,\bF_p)\to\pi_{k,w+1}\bS_{\bF_p}\xrightarrow{\tau}\pi_{k,w}\bS_{\bF_p}\xrightarrow{i}\Ext_{\cA_*}^{w-k,w}(\bF_p,\bF_p)\to\cdots
    $$
    we see that $\pi_{k,w+1}\bS_{\bF_p}$ sits in a short exact sequence
    \begin{equation*}
        0\to\mathrm{coker}(i)\to\pi_{k,w+1}\bS_{\bF_p}\to\ker(i)\to 0.
    \end{equation*}
    Inductively, $\ker(i)\subset\pi_{k,w}\bS_{\bF_p}$ is a finitely generated abelian group and $\mathrm{coker}(i)$ is always a finite abelian group since for each bidegree $\Ext_{\cA_*}^{f,t}(\bF_p,\bF_p)$ is a finite $\bF_p$-vector space. Hence $\pi_{k,w+1}\bS_{\bF_p}$ must also be finitely generated. 
\end{proof}

\newpage





\printbibliography
\end{document}